\documentclass[10pt]{article}
\usepackage[top=1in,bottom=1in,left=1.25in,right=1.25in]{geometry}

\usepackage[absolute,overlay]{textpos}

\usepackage{graphicx}
\usepackage{ragged2e}
\usepackage{tabto}

\usepackage{amsmath}
\usepackage{amsthm}
\usepackage{amsfonts}

\usepackage{float}
\usepackage{multirow}
\usepackage{tikz}
\usepackage{graphics}
\usepackage{color}
\usepackage{url}
\usepackage{afterpage}
\usepackage{multicol}
\usepackage[font=footnotesize]{caption}
\usepackage[font=footnotesize]{subcaption}
\usepackage{marginnote}
\usepackage{tcolorbox}

\usepackage{lmodern}

\usepackage[sf,bf,medium]{titlesec}
\usepackage{pgfplots}

\usepackage{bm}
\usepackage{booktabs}
\usepackage{siunitx}
\usepackage{isomath}

\tikzset{
    partial ellipse/.style args={#1:#2:#3}{
        insert path={+ (#1:#3) arc (#1:#2:#3)}
    }
}

\usepackage[plainpages=false, colorlinks=true, citecolor=blue, filecolor=blue,
   linkcolor=blue, urlcolor=blue]{hyperref}
\usepackage{enumitem}

\usepackage{multicol}
\newcommand\mtx[1]{\bm{\mathsf{#1}}}
\newcommand\elem[1]{\mathsf{#1}}

\newcommand{\lp}{\left(}
\newcommand{\rp}{\right)}

\usepackage[sort,compress]{cite}

\newcommand\bbR{\mathbb R}

\newcommand\bx{\boldsymbol{x}}

\newcommand\by{\boldsymbol{y}}

\newcommand\br{\boldsymbol r}

\newtheorem{theorem}{\sffamily Theorem}
\newtheorem{remark}{\sffamily Remark}
\newtheorem{definition}{\sffamily Definition}

\newtheorem{lemma}{\sffamily Lemma}
\newtheorem{corollary}{\sffamily Corollary}[theorem]
\newtheorem{conjecture}{\sffamily Conjecture}

\newcommand{\cA}{\mathcal A}
\newcommand{\cB}{\mathcal B}
\newcommand{\cN}{\mathcal N}

\newcommand{\cS}{\mathcal S}
\newcommand{\cO}{\mathcal O}

\newcommand{\cI}{\mathcal I}

\newcommand{\cL}{\mathcal L}

\newcommand{\cK}{\mathcal K}
\newcommand{\cR}{\mathcal R}
\newcommand{\cT}{\mathcal T}

\newcommand{\surflap}{\Delta_\Gamma}
\newcommand{\surfdiv}{\nabla_\Gamma \cdot}
\newcommand{\surfgrad}{\nabla_\Gamma}
\newcommand{\gradg}{\surfgrad}
\newcommand{\divg}{\surfdiv}
\newcommand{\LB}{\surflap}

\numberwithin{equation}{section}

\newcommand{\bs}{\boldsymbol}

\newcommand{\Ho}{H^{1}(\Gamma)}

\newcommand{\Hoo}{H^{1}_{mz}(\Gamma)}

\newcommand{\figref}[1]{\figurename~\hyperref[#1]{\ref{#1}}}
\newcommand{\tableref}[1]{\tablename~\hyperref[#1]{\ref{#1}}}

\newcommand{\DO}{\divg a\gradg}

\renewcommand{\phi}{\varphi}

\binoppenalty=\maxdimen
\relpenalty=\maxdimen

\newcommand{\dx}{\bs r}
\newcommand{\dxnorm}{r^2}

\newcommand{\trace}{\operatorname{tr}}

\newcommand{\KLB}{K_{\text{LB}}}
\newcommand{\RLB}{R_{\text{LB}}}
\newcommand{\RopLB}{\mathcal{R}_{\text{LB}}}

\usepackage{setspace}

\begin{document}
\begin{titlepage}

  \raggedleft
  {\sffamily \bfseries STATUS: arXiv pre-print}
  
  \hrulefill
  
  \raggedright
 \begin{textblock*}{\linewidth}(1.25in,2in) 
    {\LARGE \sffamily \bfseries A parametrix method for elliptic surface PDEs}
  \end{textblock*}

  \vspace{1.5in}
  Tristan Goodwill\footnote{Research supported in part by the Research 
  Training Group in Modeling and Simulation funded by the National Science
  Foundation via grant RTG/DMS-1646339, by the Simons Foundation/SFARI (560651,
  AB), and the Office of Naval Research under awards \#N00014-18-1-2307
  and \#N00014-21-1-2383.}\\ \textit{\small Department of Statistics\\
  University of Chicago\\
  Chicago, IL 60637}\\ \texttt{\small tgoodwill@uchicago.edu}

  \vspace{1.5\baselineskip}
  Michael O'Neil\footnote{Research supported in part by
  the Simons Foundation/SFARI (560651, AB) and the Office of Naval 
  Research under awards \#N00014-18-1-2307
  and \#N00014-21-1-2383.}\\
  \textit{\small Courant Institute, NYU\\
    New York, NY 10012}\\
  \texttt{\small oneil@cims.nyu.edu}

  \begin{textblock*}{\linewidth}(1.25in,7in) 
    \today
  \end{textblock*}


  
\end{titlepage}

\begin{abstract}
  Elliptic problems along smooth surfaces embedded in three dimensions
  occur in thin-membrane mechanics, electromagnetics (harmonic vector
  fields), and computational geometry. In this work, we present a
  parametrix-based integral equation method applicable to several
  forms of variable coefficient surface elliptic problems. Via the use
  of an approximate Green's function, the surface PDEs are transformed
  into well-conditioned integral equations. We demonstrate high-order
  numerical examples of this method applied to problems on general
  surfaces using a variant of the fast multipole method based on
  smooth interpolation properties of the kernel. Lastly, we discuss
  extensions of the
  method to surfaces with boundaries.\\
  
  \noindent {\sffamily\bfseries Keywords}: Surface elliptic PDE,
  Laplace-Beltrami, parametrix, surface boundary value problems

\end{abstract}

\small
\tableofcontents

\newpage

\normalsize

\section{Introduction}

For most of the constant-coefficient (elliptic) partial differential
equations (PDEs) of classical mathematical physics, exact forms of the
solution operator, or Green's function, are widely known and often
used to transform boundary value problems into boundary integral
equation formulations. In many situations, the resulting integral
equations are the preferred formulation for solving the problem
numerically due to their conditioning properties and their ease of
handling complex geometries (including unbounded ones). As a result of
incredible development over the past 20-30 years in areas such as
hierarchical linear algebra, fast multipole methods, quadrature for
singular functions, and computational geometry, obtaining high-order
accurate solutions of PDEs via integral equation formulations is
becoming more and more commonplace.

However, for the majority of PDEs, which are in general not
constant-coefficient, the associated analytical and numerical
machinery is virtually non-existent due to the absence of known
Green's functions (and even ones that are known present difficulties,
due to their non-translational-invariance). Almost all numerical
approaches rely on a direct discretization of the differential
operator via finite elements or finite difference schemes. However,
often times an \emph{approximate} Green's function, also known as a
parametrix, can be constructed that can transform the PDE into an
integral equation, albeit with kernels that need not satisfy the
underlying homogeneous PDE themselves. A particular class of such elliptic PDEs
that are of interest are what we call \emph{surface PDEs}.

Elliptic PDEs on a surface are an extension of the corresponding problems in the
plane. To be more precise, we let~$\Gamma$ be a smooth closed surface embedded
in~$\bbR^3$, and~$\divg$ and~$\gradg$ be the intrinsic surface divergence and
surface gradient on~$\Gamma$. We then define an elliptic PDE along~$\Gamma$ to
be one of the form
\begin{equation}
  \label{eq:surf_PDE}
    \divg a \gradg u + \bs b\cdot \gradg u + c u=f.
\end{equation}
Above, the goal is to find a function~$u$ given some known function~$f$,
some smooth and positive function~$a$, some continuous tangential vector
field~$\bs b$, and some continuous function~$c$. Our assumptions on these functions
and the geometry will be made more explicit later on in the manuscript.

A particularly important surface elliptic problem is the Laplace-Beltrami
problem, in which we wish to find a~$u$ such that
\begin{equation}
    \LB u :=\divg\gradg u = f.\label{eq:LB_prob}
\end{equation}
The operator~$\LB$ is known as the Laplace-Beltrami operator and is the surface
equivalent of the Laplace operator. The Laplace-Beltrami problem appears in many
fields, including computer graphics for interpolation and shape
analysis~\cite{andreux2015,vallet2008spectral}, plasma
physics~\cite{Malhotra2019}, vesicle deformation in biological
flows~\cite{rahimian2010petascale,veerapaneni2011fast}, surface
diffusion~\cite{bansch2005finite,escher1998surface}, computational
geometry~\cite{kromer2018}, and computer vision~\cite{angenent1999brain}. It is
also used in electromagnetics to find the Hodge decomposition of a tangential
vector field into curl-free and divergence-free components. See
Section~\cite{Epstein2009,Epstein2013,ONeil2018,nedelec2001} for more details
regarding such decompositions, their uses, and a proof of their uniqueness.

The Laplace-Beltrami problem is not the only equation of the
form~\eqref{eq:surf_PDE} that is of interest. For example, the oscillations of
thin films~\cite{Kuo2012,Grinfeld2010,Grinfeld2012} can be modeled by the
Helmholtz-Beltrami problem:
\begin{equation}
    \lp \LB  + k^2 \rp u = f,
\end{equation}
where $k$ is a positive function that is typically constant. Another example
that we shall mention occurs when attempting to solve the diffusion equation on
a surface~\cite{chung2004,NOVAK20071271} using the method of
semi-discretization. The equation we consider is
\begin{equation}
    \partial_t u=\LB u, \label{eq:surf_heat}
\end{equation}
where~$u$ represents some quantity diffusing over $\Gamma$. As an example, we
suppose that we discretize~\eqref{eq:surf_heat} in time with the backwards Euler
method:
\begin{equation}
  \frac{u\lp t_{n+1}\rp-u\lp t_n\rp}{\Delta t} = \LB u\lp t_{n+1}\rp,
\end{equation}
where~$\Delta t = t_{{n+1}} - t_{n}$.  Rearranging this equation
results in a modified Helmholtz-Beltrami problem for the solution at the next
timestep,~$u(t_{n+1})$:
\begin{equation}
    \LB u\lp t_{n+1}\rp -\frac{1}{\Delta t}u\lp t_{n+1}\rp = - \frac{1}{\Delta t}u\lp t_{n}\rp.
      \label{eq:yukawa_beltrami}
\end{equation}
If the diffusion equation is complicated by varying the diffusion constant or
adding a drift term, such as is considered in \cite{naumovets1985surface},
the analogous elliptic equation has
a non-constant~$a$ or non-zero~$\bs b$.

There are a few existing numerical schemes for solving surface elliptic
problems. Many methods are based on direct discretization of the differential
operators, including finite element
methods~\cite{Dziuk1988,bansch2005finite,Burman2017,Bonito2020,
  demlow2007adaptive,Burman2020,bonito2018}, the so-called \emph{virtual element
  method}~\cite{frittelli2018virtual,beirao2014hitchhiker}, differencing
methods~\cite{wang2018modified}, pseudo-spectral methods
\cite{Imbert-Gerard2017}, and hierarchical Poincar\'e–Steklov methods
\cite{fortunato2022high}. These schemes all have the advantage that they can be
applied to a broad range of equations and many of them can be made to be
high-order; a significant effort has been made toward the efficiency of these
schemes. However, the fact that they involve discretizations of the differential
operators can make them poorly conditioned. This is of particular concern when
adaptive discretizations are required, such as when the surface, right hand
side, or PDE coefficients have sharp features. Another class of methods that is
commonly applied to solve surface PDEs are level set
methods~\cite{bertalmio2001variational,macdonald2008level,macdonald2010implicit,ruuth2008simple}.
These methods rely on formulas which relate finite differences in the volume
around a surface to the intrinsic surface differentials themselves. Lastly,
there exist a class of methods known as \emph{closest point
  methods}~\cite{chen2015closest,king2023closest} which use an underlying
representation of the surface in terms of \emph{closest points}. The associated
embedding is then discretized using finite differences, similar in style to when
solving PDEs via level set methods. As both of these schemes also involve the
discretized differential operators, they can suffer from the same conditioning
problems as other direct methods.

For these reasons and others, we wish to consider methods that don't involve
discretizing differential operators. In particular, we wish to convert the
surface PDE into an integral equation. This has the advantage that when the
integral equation is Fredholm second-kind, it is generally as well-conditioned
as the underlying physical problem and can be solved accurately with high-order
quadrature schemes that have been developed over the past several decades.
Furthermore, discretized integral equations can often be solved asymptotically
fast when combined with hierarchical compression methods, such as the fast
multipole method (FMM)~\cite{greengard1987fast}. Integral equation methods have
previously been used to solve the Laplace-Beltrami problem, but have required
discretizing the composition of a large number of weakly singular
operators~\cite{ONeil2018,agarwal2023fmm}, or have only been applied to specific
geometries, such as axisymmetric surfaces~\cite{Epstein2019,o2018integral} or
spheres~\cite{kropinski2014fast,Kropinski2016,Alves2018}. Other surface PDEs
have not received as much attention and there only exist integral equation
methods for constant coefficient PDEs on subsets of the
sphere~\cite{Kropinski2016}.

In this paper, we present a parametrix (also referred to as an approximate
Green's function or L\'evy function) for~\eqref{eq:surf_PDE} and use it to
convert~\eqref{eq:surf_PDE} into an integral equation. As we shall see below, if
the kernel of an integral operator is a parametrix for a PDE, then
preconditioning the PDE with that operator results in a Fredholm second-kind
integral equation. Parametrix based methods have been used to solve elliptic
problems in the
plane~\cite{Chkadua2009,Chkadua2010,Chkadua2013,Chkadua2018,Beshley2018,
  Fresneda-Portillo2020,Fresneda-Portillo2021}, but have not been widely applied
to problems along surfaces (despite the underlying theory being closely
related). Our method can be thought of as an extension of the method presented
in~\cite{kropinski2014fast} to more general equations and surfaces.
In~\cite{kropinski2014fast}, it was shown that the Green's function for the
Poisson equation in the plane gives the Green's function for the
Laplace-Beltrami problem on a sphere, up to an additive constant function. In
this paper, we build on this observation and use the Green's functions for the
Poisson and Helmholtz equations in the plane to find a parametrix
for~\eqref{eq:surf_PDE} on general smooth surfaces. We then use this parametrix
to turn~\eqref{eq:surf_PDE} into an integral equation and present a method for
solving that integral equation numerically.

A parametrix-based integral equation has a few key advantages over other
integral equation methods relying on Calder\`on-type identities or other
potential-theoretic manipulations. First, the integral equation involves only a
single integral operator with a kernel that is less singular than those required
in the existing method for solving the Laplace-Beltrami problem on general
surfaces~\cite{ONeil2018}. Second, the approach is immediately applicable to a
broad range of equations along any smooth surface. Finally, as demonstrated
later on, using the parametrix as an analogue for the Green's function allows
for associated formulations for solving \emph{surface boundary value problems},
i.e. elliptic surface PDE problems on surfaces with boundary and specified
conditions along such boundaries.

The rest of the paper is divided into seven sections: in
Section~\ref{sec:analytics}, we define a parametrix, recall some standard
definitions for surface differential operators, and discuss existence and
uniqueness properties of the Laplace-Beltrami problem. In
Section~\ref{sec:LBparametrix}, we introduce a parametrix and prove that it can
be used to solve the Laplace-Beltrami problem. After discussing this simpler
case, in Section~\ref{sec:existence} we discuss existence and uniqueness of
solutions to~\eqref{eq:surf_PDE}; extensions of the parametrix to the more general case are
addressed as well. Then, in Section~\ref{sec:BVP}, we discuss how our parametrix
can be used to solve surface elliptic boundary value problems. In
Section~\ref{sec:Numerics}, we describe a method for discretizing the integral
equations derived in the previous sections. In
Section~\ref{sec:Numeric_Exp}, we describe some numerical experiments that were
used verify the accuracy and efficiency of our method. Finally, in
Section~\ref{sec:conclussion}, we present some concluding remarks about this
method and present some ways to extend this method to other problems.

\section{Analytical preliminaries} \label{sec:analytics}
In this section, we give the definition of a parametrix and discuss the
motivation behind the definition and the typical uses of parametrices.
Afterwards, we define the necessary surface differential operators and give the
weak form of~\eqref{eq:surf_PDE}. Finally, we discuss the solvability of the
Laplace-Beltrami problem. The solvability of other elliptic PDEs on surfaces is
more complicated, and so is left to Section~\ref{sec:existence}.

\subsection{Parametrix basics}
\label{sec:parametix}

Let~$\cL$ be a second order elliptic operator. A function~$G(\bx,\bx')$ is said
to be a Green's function for~$\cL$ if 
\begin{equation}
    \cL G(\bx,\bx') = \delta(\bx-\bx')
\end{equation}
for each~$\bx'$ in the domain in question. For most such~$\cL$, the exact
Green's function~$G$ is not known in closed form; however, it is often possible
to find a function~$K = K(\bx,\bx')$ such that
\begin{equation}
  \cL K(\bx,\bx')=\delta(\bx-\bx') + R(\bx,\bx') \label{eq:parametrix_def}
\end{equation}
for all~$\bx'$. If~$R$ is reasonably well-behaved (as clarified below), then
such a function~$K$ can be useful for both analytical and numerical work. We now
define a particular class of~$K$'s with a notion of well-behaved remainder
functions~$R$.

\begin{definition}[Parametrix \cite{Chkadua2009}]
Let~$\cL$ be a second order differential operator. The function~$K(\bx,\bx')$ is
said to be a parametrix for~$\cL$ if~$\cL K(\bx,\bx')=\delta(\bx-\bx') +
R(\bx,\bx')$, where the remainder function~$R$ is an integrable function
of~$\bx$ e.g. the remainder function is~$\cO\lp
\left\|\bx-\bx'\right\|^{\alpha}\rp$ with~$\alpha>-2$ as~$\bx\to\bx'$ on a two dimensional surface.
\end{definition}

Parametrices exist for other kinds of differential operators, with different
conditions on the remainder function, but in this paper, we only consider second
order elliptic problems and so the definition is sufficient. It is clear that a
Green's function is a parametrix with remainder~$R=0$. For this reason, it is
also common to refer to a parametrix as an approximate Green's function.

Historically, parametrices are used to prove existence for equations with
non-analytic coefficients (see \cite{Garabedian1986}), though they have also
been used for computational methods.  Much of this work uses the method for
deriving parametrices that was introduced in \cite{hadamard1932probleme} and is
summarized in \cite{Hormander1994}. This classical parametrix is poorly suited
for our purposes because it is written in terms of induced geodesic coordinates.
We will therefore use a different parametrix, which will be much easier to
evaluate numerically and give a remainder function that is smooth enough for
our purposes.

\subsection{Surface basics}
\label{sec:surf_basics}

In this section, we introduce some standard notation and ideas for working with
differential equations defined on smooth surfaces. We define the standard
differential operators and the weak form of~\eqref{eq:surf_PDE}. Throughout this
section, we assume that the surface~$\Gamma$ is smooth and is given as the
boundary of an open and bounded domain~$\Omega\subset\bbR^3$. We shall discuss
the standard differential operators on a surface through the use of a single
(local) parameterization.
 \begin{definition}
   Let~$\Gamma$ be a smooth surface and let~$\bs y = \bs y(s,t)$ be a
   local parameterization with image~$\Gamma_{\by}$. We can make the
   following definitions.
 \begin{enumerate}
     \item The surface metric associated with~$\by$ is 
    \begin{equation}
  g = \begin{bmatrix}
    \partial_s\by \cdot \partial_s\by & \partial_s\by \cdot \partial_t\by \\
    \partial_s\by \cdot \partial_t\by & \partial_t\by \cdot \partial_t\by
    \end{bmatrix}.
\end{equation}
\item The surface gradient of a function~$f$ is 
\begin{equation}
      \gradg f(\bx)=\begin{bmatrix}\by_s & \by_t\end{bmatrix}g^{-1}\begin{bmatrix}
    \partial_s\\
    \partial_t
    \end{bmatrix}f.\label{eq:gradg}
\end{equation}
\item If~a tangential vector field~$\bs v$ is written as~$\bs v|_{\Gamma_{\bx}} =
v_s \by_s + v_t\by_t$, then the surface divergence is given by the formula
\begin{equation} \label{eq:divg}
  \divg \bs v(\by(s,t))=\frac{1}{\sqrt{\det g}}
  \begin{bmatrix}
    \partial_s& \partial_t
  \end{bmatrix}\sqrt{\det g}\begin{bmatrix}
    v_s\\
    v_t
  \end{bmatrix}.
\end{equation}
 \end{enumerate}
 \end{definition}
The above definitions allow us to write and enforce~\eqref{eq:surf_PDE} on the
space of smooth functions. For this reason it is called a strong form PDE. In
what follows, however, the weak form of~\eqref{eq:surf_PDE} will allow us to use
the tools of functional analysis to find conditions for existence of solutions
to~\eqref{eq:surf_PDE} and allow us to simplify subsequent proofs regarding
a one-to-one correspondence of~\eqref{eq:surf_PDE} and the associated integral equation. 
The weak form of a PDE is an equation over the Sobolev space
\begin{equation}
  \Ho :=\left\{\rho\in L^2(\Gamma) : \gradg \rho\in (L^2(\Gamma))^3\right\},
\end{equation}
which is a Hilbert space with inner product
    \begin{equation}
        (\sigma,\rho)_{\Ho} = \int_\Gamma \lp \bar \rho \, \sigma + \gradg\bar\rho\cdot\gradg\sigma \rp.
    \end{equation}    
With this space defined, we say that the weak form of~\eqref{eq:surf_PDE} is
  \begin{equation}
    \label{eq:weakform}
    \int_\Gamma \lp -a \gradg\bar \rho \cdot \gradg u +\bar \rho \, \bs b\cdot \gradg u+c \, \bar \rho \, u \rp= \int_\Gamma \bar \rho \, f,\qquad \text{for all } \rho \in \Ho. 
\end{equation}
 A function~$u\in\Ho$ that solves~\eqref{eq:weakform} is called a weak solution
 of~\eqref{eq:surf_PDE}. A simple integration-by-parts argument gives that a
 function~$u\in C^2(\Gamma)$ is a solution of~\eqref{eq:surf_PDE} if and only if
 it is a weak solution (see \cite{John1982}). 

We end this section with two more definitions: We shall let~$\bs n$ be
the outward unit normal to~$\Gamma$ and~$H=\frac12 \divg \bs n$ be the mean
curvature of~$\Gamma$.

\subsection{Existence and uniqueness for the Laplace-Beltrami problem}
\label{sec:LB_posed}

In this section, we will see that the Laplace-Beltrami problem is well-posed
when its domain and range are chosen correctly. 

By inspection, the Laplace-Beltrami operator~\eqref{eq:LB_prob} annihilates
constants. For this reason, the solution of the Laplace-Beltrami problem on a
closed surface can only be determined up to an additive constant. In order for
the equation to have a unique solution, we will therefore assume that the
solution~$u$ has zero mean. It is also necessary to restrict the range of the
Laplace-Beltrami problem: if we plug a constant~$\rho$ into the weak
form~\eqref{eq:weakform}, then~\eqref{eq:weakform} is unsolvable if~$f$ has a
non-zero mean. When we combine these observations, it becomes clear that we
should view the Laplace-Beltrami operator as a map from mean-zero functions to
themselves.

In order to include these restrictions in the weak form of~\eqref{eq:LB_prob}, we define the function space
\begin{equation}
  H^1_\text{mz}(\Gamma) := \left\{\rho\in \Ho :  \int_\Gamma \rho= 0\right\},
\end{equation}
which is also a Hilbert space with inner product~$\lp\cdot,\cdot\rp_{\Ho}$.
Since we need it below, we also define~$L^2_\text{mz}(\Gamma)$ as the subset
of~$L^2(\Gamma)$ with mean zero.

With these function spaces, the weak form of~\eqref{eq:LB_prob} becomes
\begin{equation}
    -\int_\Gamma \gradg u \cdot \gradg\rho = \int_\Gamma f\, \rho,  \qquad 
    \text{for all } \rho \in H^1_\text{mz}(\Gamma), \label{eq:LB_weakform}
\end{equation}
where~$f\in L^2_\text{mz}(\Gamma)$. As before, a smooth function~$u$ is a solution
of~\eqref{eq:surf_PDE} if and only if it is a solution
of~\eqref{eq:LB_weakform}. The following theorem gives that this weak form is
well posed. 

\begin{theorem}\label{thm:LBsolvable} If~$\Gamma$ is a smooth and closed
surface, then for every mean-zero right hand side~$f\in L^2_\text{mz} (\Gamma)$ there
exists a unique mean-zero solution~$u\in H^1_\text{mz}(\Gamma)$ of~\eqref{eq:LB_prob}.
\end{theorem}
This theorem is an immediate consequence of a subsequent result in the
manuscript, Theorem~\ref{thm:well-posed}, with~$a= 1$ and~$c= 0$. It
is also well-known on its own. Theorem~\ref{thm:LBsolvable} can be extended to
give higher regularity on the solution~$u$. In~\cite{Warner2013}, it is shown
that it if the right hand side is in a higher order Sobolev space, then it is
possible to control the norm of the higher order derivatives of~$u$, but we omit
such discussions from this work for the sake of brevity.

\section{A parametrix for the Laplace-Beltrami problem}
\label{sec:LBparametrix}

In this section, we introduce a parametrix for the Laplace-Beltrami
problem~\eqref{eq:LB_prob} on a smooth and closed surface~$\Gamma$. We then use
that parametrix to convert~\eqref{eq:LB_prob} into an equivalent integral
equation. We save the discussion of general elliptic equations for
Section~\ref{sec:gen_elliptic}. 

The function~$$G(r)=\frac{1}{2\pi}\log r $$ is the free-space
Green's function for the Laplace operator in~$\bbR^2$. An analogous function
can be defined in~$\bbR^3$, as in~\cite{kropinski2014fast}:
\begin{equation}
  \label{eq:parametrixLB}
    \KLB(\bx,\bx') = \frac{1}{2\pi} \log \Vert\bx-\bx'\Vert
\end{equation}
for all~$\bx,\bx'\in \bbR^3$ with~$\bx\neq \bx'$. The smoothness of~$\Gamma$ will give that~$\KLB$ is a parametrix for the Laplace-Beltrami problem.

\begin{theorem}\label{thm:parametrixLB}
If~$\Gamma$ is a smooth closed surface, then~$\KLB$ is a candidate parametrix for the Laplace-Beltrami equation, i.e.
\begin{equation}
  \label{eq:diff_formLB}
    \LB \KLB(\bs x,\bs x') = \delta(\bx - \bx') +\RLB(\bx,\bx'),
\end{equation}
where the derivatives in~$\LB$ are taken with respect to~$\bx$. Furthermore,
for~$\bx \neq \bx'$, the remainder function is
\begin{equation}
  \label{eq:explicit_remainderLB}
  \RLB(\bs x,\bx')=\frac{1}{\pi }\left(\frac{ \bs n(\bx)\cdot \dx }{\dxnorm}\right)^2 - \frac{H(\bx)}{\pi } \frac{\bs n(\bx)\cdot \dx}{\dxnorm},
\end{equation}
 where~$\br = \bx - \bx'$ and
$r = \Vert \br \Vert$.
\end{theorem}
\begin{proof}

We begin by choosing a parameterization such that~$\bx'=\by(\bs 0)$, and
define~$C_R(\bx'):=\by(C_R(\bs 0))$ and~$B_R(\bx'):=\by(B_R(\bs 0))$ to be the image of
the circle and ball of radius~$R$ centered at~$\bx'$, respectively.
Since $\KLB$ is smooth for $\bx\neq \bx'$, showing~\eqref{eq:diff_formLB} is equivalent to showing that 
\begin{equation}
  \lim_{R\to 0}\int_{ B_R(\bx')} \LB\KLB(\cdot,\bx') \, \phi = \phi(\bx'),\label{eq:Param_limit}
\end{equation}
for all smooth test functions~$\phi$, where the Laplace-Beltrami operator is interpreted in a weak sense.

First, by the definition of the weak derivative, we have that
\begin{equation}
   \int_\Gamma \LB\KLB(\cdot,\bx')  \, \phi = \int_\Gamma \KLB(\cdot,\bx') \, \LB\phi.\label{eq:IBP_K}
\end{equation}
Next, for~$R>0$, manipulating the domains of integration,
using~\eqref{eq:IBP_K}, and applying Green's third identity gives that
\begin{equation}
  \begin{aligned}
        \int_{B_R(\bx')}\LB\KLB(\cdot,\bx')  \, \phi
        &= \int_\Gamma \KLB(\cdot,\bx') \, \LB\phi
         - \int_{\Gamma\setminus B_R(\bx')}  \LB\KLB(\cdot,\bx') \,  \phi \\
        &=\int_{B_R(\bx')} \KLB(\cdot,\bx') \, \LB \phi
        +  \int_{C_R(\bs x')}   \bs \nu\cdot \gradg \KLB(\cdot,\bx') \,  \phi\\
        &\qquad\qquad\qquad-\int_{C_R(\bs x')} \KLB(\cdot,\bx')\,\bs \nu\cdot \gradg \phi,
\end{aligned}
\end{equation}
where~$\bs \nu$ is the unit~\emph{binormal} vector to~$C_R(\bx')$, which is
defined to be tangent to~$\Gamma$, normal~$C_R(\bx')$, and to point away
from~$B_R(\bx')$. The first and third terms above will vanish as~$R\to 0$ due to
the smoothness of~$\phi$ and the weak logarithmic singularity in~$\KLB$. For the
second term, we add and subtract~$\phi(\bx')$, yielding:
\begin{multline}
  \int_{C_R(\bs x')}   \bs \nu \cdot \gradg \KLB(\cdot,\bx')  \, \phi = \int_{C_R(\bs x')}   \bs \nu\cdot \gradg \KLB(\cdot,\bx')  \, \phi \\
  + \int_{C_R(\bs x')}   \bs \nu\cdot \gradg \KLB(\cdot,\bx') \,   (\phi-\phi(\bx')).
\end{multline}
Since $\phi$ is smooth, and $\KLB$ has a logarithmic singularity, the second term
vanishes as~$R\to0$. We therefore have that
\begin{equation}
  \lim_{R\to 0}\int_{\Gamma\cap B_R(\bx')} \LB\KLB(\cdot,\bx')  \, \phi  = \phi(\bx') \lim_{R\to 0}\int_{\Gamma\setminus B_R(\bx')}   \bs \nu\cdot \gradg \KLB(\cdot,\bx').\label{eq:phi_lim}
\end{equation}
We now show that the right hand limit is equal to one. First, we note
that~$$\bs \nu=\frac{\bx-\bx'}{ \|\bx-\bx'\|}+\cO(R)$$  and define~$I_R$ by
\begin{equation}
  \begin{aligned}
    I_R &:=\int_{C_R(\bs x')} \bs \nu \cdot \gradg \KLB(\cdot ,\bs x') \\
        &=\int_{C_R(\bs x')} \lp\frac{\bx-\bx'}
          {\|\bx-\bx'\|}+\cO(R)\rp \cdot \frac{\bx\cdot \gradg\bx}{2\pi\Vert\bs x-\bs x'\Vert^2}.
  \end{aligned}
\end{equation}
We next parameterize~$C_R(\bx')$ by~$\by(\Phi_R(\theta))$
where~$\Phi_R(\theta):=R\begin{bmatrix}\cos\theta&\sin\theta\end{bmatrix}^T$.
Taylor expanding~$\by$ around the origin gives that
\begin{equation}
    I_R=\int_0^{2\pi} \left(\frac{g(\bs0)\Phi_R(\theta)}{R}+\cO(R)\right)^T \frac{g^{-1}(\bs0) \, \Phi_R(\theta) + \cO(R^2)}{2\pi R^{2}+O(R^{3})} R \, d\theta+ \cO(R).
\end{equation}
Some more simplification gives 
\begin{equation}
    I_R=(2\pi)^{-1}\int_0^{2\pi} R^{-2} \, \Phi_R(\theta)^T g(\bs0) \, g^{-1}(\bs0) \, 
    \Phi_R(\theta) \, d\theta+\cO(R).
\end{equation}
Evaluating this integral gives
\begin{equation}
    I_R= 1+\cO(R),
\end{equation}
which, when combined with~\eqref{eq:phi_lim}, proves~\eqref{eq:Param_limit} and
so that~$K_{LB}$ is a candidate to be a parametrix.

The expression for the remainder~\eqref{eq:explicit_remainderLB} is obtained
from the following formula for the surface Laplacian of a smooth function
defined in a neighborhood of~$\Gamma$ (see~\cite{nedelec2001}), which will be
valid wherever~$\KLB$ is smooth, i.e. for $\bx\neq\bx'$:
\begin{equation}
  \label{eq:diff_remainderLB}
  \begin{aligned}
    \RLB(\bx,\bx') &=\LB \KLB(\bx,\bx') \\
    &=   \Delta \KLB(\bx,\bx') - \partial_{\bs n(\bx)}^2 \KLB(\bx,\bx') - 2H(\bx) \partial_{\bs n(\bx)} \KLB(\bx,\bx').
  \end{aligned}
\end{equation}
\end{proof}

We now verify that~$\KLB$ is indeed a parametrix for the Laplace-Beltrami
problem by showing that~$\RLB$ is integrable.
\begin{theorem}\label{thm:remainderLB}
If~$\Gamma$ is a smooth surface, then~$\RLB(\bx,\bx')=\cO(1)$ as~$\bx'\to\bx$ and~$\KLB$ is a parametrix for the Laplace-Beltrami problem.
\end{theorem}
\begin{proof}
By Theorem~\ref{thm:parametrixLB}, it is enough to show the first statement.

Let~$\gamma$ be a curve along~$\Gamma$, passing through~$\bx$, obtained from
the intersection of a plane normal to~$\Gamma$ at~$\bx$ with~$\Gamma$. We shall show
that~$\RLB(\bx,\gamma(s))$ is a continuous function of~$s$, where~$\gamma(s)$ is
an arclength parameterization of~$\gamma$.

We suppose that~$\gamma(\tilde s)=\bs x$. As an intermediate step, we Taylor
expand~$\gamma$ and compute
\begin{equation}
    \lim_{s\to \tilde s} \frac{\bs n(\bx) \cdot (\gamma(\tilde s)-\gamma(s))}{\Vert\gamma(\tilde s)-\gamma(s)\Vert^2} = \lim_{s\to \tilde s} \frac{\bs n(\bx) \cdot \lp(\tilde s-s) \gamma'(\tilde s) + \frac{(\tilde s-s)^2}2\gamma''(\tilde s) + \cO((\tilde s-s)^3)\rp}{(\tilde s-s)^2 \Vert\gamma'(s)\Vert^2 + \cO((\tilde s-s)^4)}.
\end{equation}
Since~$\gamma$ is parameterized by arclength,~$\Vert\gamma'\Vert=1$,~$\bs n(\bx)\cdot \gamma'(\tilde s)=0$, and~$\bs n(\bx)\cdot \gamma''(\tilde s)=\kappa_\gamma$, where~$\kappa_\gamma$ is the curvature of~$\gamma$ at~$\bs x$. The limit thus becomes
\begin{equation}
    \lim_{s\to \tilde s} \frac{\bs n(\bx) \cdot (\gamma(\tilde s)-\gamma(s))}{\Vert\gamma(\tilde s)-\gamma(s)\Vert^2} = \frac{\kappa_\gamma}2.
\end{equation}
Inserting this limit into~\eqref{eq:explicit_remainderLB}, we obtain
\begin{equation}
  \begin{aligned}
     \lim_{s\to \tilde s}\RLB(\gamma(\tilde s),\gamma(s))
     &=\lim_{s\to \tilde s}\left[2\left(\frac{\bs n(\bx) \cdot (\gamma(\tilde s)-\gamma(s))}{\Vert\gamma(\tilde s)-\gamma(s)\Vert^2}\right)^2 - 2H(\bx) \frac{\bs n(\bx) \cdot (\gamma(\tilde s)-\gamma(s))}{\Vert\gamma(\tilde s)-\gamma(s)\Vert^2}\right] \\
     &=\frac{\kappa_\gamma^2}2-H(\bx)\kappa_\gamma.
  \end{aligned}
\end{equation}
The curvature~$\kappa_\gamma$ will vary between the principal
curvatures~$\kappa_1$ and~$\kappa_2$, depending on the angle of approach.
Specifically, if~$\theta$ is the angle between the first principal direction
and~$\gamma'(\tilde s)$, then 
\begin{equation}
    \kappa_\gamma(\theta) = \kappa_1 \cos^2(\theta) + \kappa_2\sin^2(\theta).
\end{equation}
Using this expression for the curvature and writing the principal curvatures in terms of the mean and Gaussian curvatures gives that
\begin{equation}
    \lim_{s\to \tilde s}\RLB(\gamma(\tilde s),\gamma(s))= -\frac{H^2(\bx)\sin^2(2\theta)+\kappa_G(\bx)\cos^2(2\theta)}{4\pi }, \label{eq:limiting_behaviour}
\end{equation}
where~$\kappa_G(\bx)$ is the Gaussian curvature of the surface at~$\bx$. Since
the surface is smooth, this limit is bounded, and therefore~$\RLB(\bx,\bx')$
is~$\mathcal O(1)$ as~$\bx'\to\bx$.
\end{proof}

We shall pause here to note that the above proof gives that the remainder
function is not smooth as~$\bx'\to\bx$ because~$R_{LB}(\bx,\bx')$ depends on the
direction of approach. A plot of the remainder kernel along a toroidal like
object, for a fixed~$\bx$, is shown in Figure~\ref{fig:remainderLB}. 

\begin{figure}[t]
    \centering
    \includegraphics[width=0.75\textwidth]{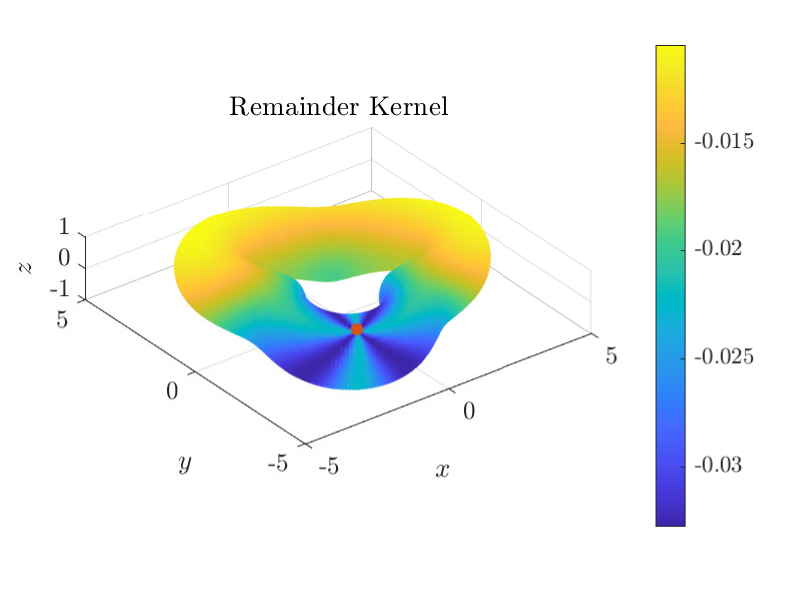}
    \caption{This plot shows the remainder kernel~$\RLB(\bx,\cdot)$ for the
     Laplace-Beltrami problem on an example surface. We see that it is bounded, 
     but not smooth around~$\bx$ (red dot).}
    \label{fig:remainderLB}
\end{figure}

We can now show that using this parametrix yields a second-kind
integral equation for the Laplace-Beltrami problem along general smooth
surfaces.

\begin{theorem}
  \label{thm:second_kindLB}
The integral equation
\begin{equation}
  \label{eq:IELB}
  \sigma(\bs x) + \int_\Gamma \RLB(\bs x,\bs x') \, \sigma(\bx') \, da(\bx') =  f(\bx), \qquad \bx \in \Gamma,
\end{equation}
is a second-kind Fredholm integral equation on~$L^2(\Gamma)$. Furthermore,
if~$f\in L^2_\text{mz}(\Gamma)$ and~$\sigma$ is the solution of~\eqref{eq:IELB}, then
the function   
\begin{equation}
  \label{eq:representationLB}
  \begin{aligned}
  u(\bx)&=\mathcal\KLB[\sigma](\bx) \\
  &=\int_\Gamma \KLB(\bx,\bx') \, \sigma(\bx') \, da(\bx')
  \end{aligned}
\end{equation}
is the unique mean-zero solution of the Laplace-Beltrami problem.
\end{theorem}
\begin{proof}
We begin by noting that the remainder function is bounded. The integral 
operator 
\begin{equation}
    \RopLB [\sigma] (\bs x)=\int_\Gamma \RLB(\bs x,\bs x') \, \sigma(\bx') 
    \, da(\bx')
\end{equation}
is therefore a Hilbert-Schmidt operator and so is compact as a map
from~$L^2(\Gamma)$ to itself by Theorem~{0.45} in~\cite{Folland1995}. The
equation~\eqref{eq:IELB} is therefore a second-kind Fredholm integral equation.

We shall now verify that when~$\sigma$ solves~\eqref{eq:IELB}, the formula
in~\eqref{eq:representationLB} gives a weak solution of~\eqref{eq:LB_prob},
i.e., that for all~$\phi \in C^2(\Gamma)$
\begin{equation}
  \label{eq:wkLB}
 -\int_\Gamma \gradg u \cdot\gradg\phi 
   = \int_\Gamma f\,  \phi.
\end{equation}
Integrating by parts in the above expression and plugging in the representation
for~$u$ gives us that 
\begin{equation}
  \int_\Gamma u \, \LB\phi 
    =  \int_\Gamma \left( \int_\Gamma \KLB(\bs x,\bs x') \, \sigma(\bx') 
    \, da(\bx')\right) \LB\phi(\bx) \, da(\bx).
\end{equation}
We note that since~\eqref{eq:IELB} is Fredholm second-kind,~$\sigma$ will be
in~$L^2_\text{mz}(\Gamma)$. Since~$\KLB$ is also integrable, we may use
Fubini's theorem to find that
\begin{equation}
  \int_\Gamma u \, \LB\phi 
   =\int_\Gamma \sigma(\bx') \int_\Gamma \KLB(\bs x,\bs x') \, 
    \LB\phi(\bx) \, da(\bx) da(\bx').
\end{equation}
We now wish to bring the derivatives on~$\phi$ onto~$\KLB$. We do this by integrating by parts in the weak sense. Since we have computed the weak derivative of~$\KLB$ in Theorem~\ref{thm:parametrixLB}, we have that
\begin{equation}
  \begin{aligned}
    \int_\Gamma u \, \LB\phi  &= \int_\Gamma \sigma(\bx') \int_\Gamma \LB 
  \KLB(\bx,\bx') \, \phi(\bx) \, da(\bx) da(\bx') \\
  &=\int_\Gamma \sigma(\bx') \left(\phi(\bx')+\int_\Gamma \RLB(\bx,\bx') \, \phi(\bx) \, da(\bx)  \right) da(\bx').
  \end{aligned}
\end{equation}
Another application of Fubini's theorem combined with a change of variables gives
\begin{equation}
  \begin{aligned}
    \int_\Gamma u \, \LB\phi &= \int_\Gamma\phi(\bx) \left(\sigma(\bx) +\int_\Gamma \RLB(\bx,\bx') \, \sigma(\bx') \, da(\bx') \right) da(\bx) \\
    &= \int_\Gamma \phi \, f.
  \end{aligned}
\end{equation}
Thus~$\LB u=f$ in the weak sense, and so~$u$ is the desired solution.
\end{proof}

The integral equation in~\eqref{eq:IELB} involves the
operator~$\cI+\mathcal\RLB$, which is the composition of the Laplace-Beltrami
operator with~$\mathcal\KLB$. It is therefore \emph{not} solvable for any right
hand side with non-zero mean. Since~\eqref{eq:IELB} is Fredholm second-kind, it
must therefore have at least a one-dimensional null-space consisting of functions
that are mapped to constants by~$\mathcal\KLB$. 
For numerical reasons, it is often advantageous to modify this
equation so that it is uniquely invertible for \emph{any} right-hand side, and
handle the mean-zero constraint separately and explicitly.
We can do this by adding the term~$\int_\Gamma \sigma$. 

This kind of
one-dimensional update is discussed in detail~\cite{Rachh2015}. In short, this
update makes the solution of~\eqref{eq:IELB} unique by fixing the projection of
the solution onto the null vector, which we denote by~$\sigma_n$. This method is known to remove the
null-space of the integral equation provided that~$\sigma_n$ has a non-zero
mean. While there is no formula for~$\sigma_n$ on general surfaces, some
calculus shows that if~$\Gamma$ is a sphere, then~$\sigma_n$ is a constant
function and if~$\Gamma$ is a torus, then~$\sigma_n = \alpha-\rho$,
where~$\rho = \rho(\bx)$ is the distance from~$\bx$ to the axis of rotation. The
constant~$\alpha$ depends on the inner and outer radii of the torus, but is not
the mean of~$\sqrt{x^2+y^2}$ over~$\Gamma$.

We now make the following conjecture:
\begin{conjecture}
    The integral equation~\eqref{eq:IELB} has only a one-dimensional null space on all smooth surfaces.
\end{conjecture}
This conjecture is based on a few observations:
\begin{enumerate}
   \item When~$\Gamma$ is a sphere, the work in~\cite{kropinski2014fast} showed
that~$R_{LB} = 1/ |\Gamma|$. The integral equation~\eqref{eq:IELB} thus
becomes~$\sigma -\hat\sigma=f$, where~$\hat \sigma$ is the mean of~$\sigma$,
which clearly has only a one-dimensional null space for any sized sphere. 
\item The associated discretization of the integral equation had only one zero
singular value for all of the surfaces considered in subsequent numerical
examples (Section~\ref{sec:Numeric_Exp}), which include examples of genus
0 and genus 1.
\item The integral equation~\eqref{eq:IELB} is invariant under surface dilation.
By this we mean that if~$a\Gamma=\left\{a\bx \text{ such that }
\bx\in\Gamma\right\}$, then~$\mathcal{R}_{\text{LB},a\Gamma}[\sigma](a\bx)=
\mathcal{R}_{\text{LB},\Gamma}[\sigma](\bx)$ for all~$a>0$. This implies that the sign
change in the parametrix at~$\Vert\bx-\bx'\Vert=1$ does not introduce a null
space, since~$\Gamma$ can be assumed to have a diameter less than 1.
\end{enumerate}
In addition to being interesting in and of itself, proving the above
conjecture would have computational ramifications, as the discretized
linear systems could be solved by any preferred means (iteratively,
directly, etc.) without worry of unknown rank deficiencies.

\section{Elliptic PDEs on surfaces}
\label{sec:existence}

The above parametrix method may be extended to a broader class of surface
elliptic equations of the form~\eqref{eq:surf_PDE}. In this section, we
introduce this class of equations and discuss solvability conditions and
introduce our parametrix. For simplicity, we shall assume that~$\bs b = 0$.
Introducing a non-zero~$\bs b$ only requires minor changes, and is discussed in
Section~\ref{sec:nonzerob}.

\subsection{Existence of solutions}
In this section, we shall discuss some aspects of existence and uniqueness
of~\eqref{eq:surf_PDE}; a thorough investigation of conditions under
which~\eqref{eq:surf_PDE} is well-posed is beyond the scope of this paper.
Instead, we present Theorems~\ref{thm:well-posed},~\ref{thm:Helm_posed}, and~\ref{thm:closed_eval} as
examples of such conditions, and the results which they lead to.

In order to prove these theorems,
we shall need following two well-known theorems.

\begin{theorem}[Poincar\'e's inequality for surfaces (Theorem 2.10 in~\cite{HebeyEmmanuel1999Naom})]
If~$\Gamma$ is a smooth surface then there exists a constant~$C_\Gamma>0$ such
that for every function~$u\in H^1(\Gamma)$,
\begin{equation}
        \int_\Gamma |u-\hat u|^2 \leq C_\Gamma \int_\Gamma \Vert\gradg u\Vert^2,\label{eq:poincare}
    \end{equation}
     where~$\hat u$ is the mean of~$u$.
    \end{theorem}

    \begin{theorem}[Lax-Milgram (see~\cite{Brezis2011})]
    Let~$\mathcal H$ be a Hilbert space and the sesquilinear  form~$\cB(u,\rho)$
    be bounded and coercive on~$\mathcal H$, i.e.
    \begin{equation}
        |\cB(u,\rho)|
        \leq C_B \Vert u\Vert_{\mathcal H}\Vert\rho\Vert_{\mathcal H} \quad \text{and} \quad \Re (\xi \cB(u,u)) \geq C_C \Vert u\Vert_{\mathcal H}^2,
    \end{equation}
    for some positive constants~$C_B$,~$C_C$, and some~$\xi\in \mathbb C$
    with~$|\xi|=1$. If the conjugate linear form~$\cA(\rho)$ is bounded, i.e.
    \begin{equation}
        |\cA(\rho)| \leq C_A \Vert\rho\Vert_{\mathcal H}
    \end{equation}
    for some positive constant~$C_A$, then there exists a
    unique~$u\in\mathcal{H}$ such that~$\cB(u,\rho)=\cA(\rho)$ for
    all~$\rho\in\mathcal H$.
    \end{theorem}

We are now ready to prove an existence theorem for~\eqref{eq:surf_PDE} given
some mildly restrictive conditions on~$a$ and~$c$. The proof of the first
statement appeared in~\cite{dziuk2013finite}, and  the proof of the second is
very similar to the equivalent problem in subsets of~$\bbR^n$ considered in
example 9.5.3 in~\cite{Brezis2011}.

\begin{theorem}\label{thm:well-posed}
Let~$\Gamma$ be a smooth and closed surface and~$f$ be in~$L^2(\Gamma)$. If one of the following holds, then there exists a weak solution~$u\in H^1(\Gamma)$ of~\eqref{eq:surf_PDE}:
\begin{enumerate}
    \item The coefficient function~$c$ is identically zero and~$f$ is mean-zero.
    \item There exists a~$\theta\in \lp -\frac\pi2, \frac\pi2\rp$ such that~$\Re
    (c e^{i\theta}) < 0$.
\end{enumerate}
Furthermore, the solution is unique if condition 2 holds. If condition 1 holds,
then there is a unique mean-zero solution.  
\end{theorem}
\begin{proof}
We begin by defining the sesquilinear form 
\begin{equation}
    \cB(u,\rho) = \int_\Gamma \lp -a \, \gradg\bar \rho \cdot \gradg u +c\, \bar \rho \, u \rp
\end{equation}
and the conjugate linear form
\begin{equation}
    \cA(\rho) = \int_\Gamma f \, \bar\rho.
\end{equation}
With these definitions, the weak form~\eqref{eq:weakform} becomes the equation
\begin{equation}
    \cB(u,\rho) = \cA(\rho), \qquad \text{for all } \rho \in \Ho.
\end{equation}
We note that~$\cB$ is bounded on~$H^1(\Gamma)\times H^1(\Gamma)$ :
\begin{equation}
  \begin{aligned}
    |\cB(u,\rho)| &\leq \max_\Gamma a \int_\Gamma |\gradg\bar \rho \cdot \gradg u| +\max_\Gamma |c| \int_\Gamma |\bar \, \rho u| \\
    &\leq \lp \max_\Gamma a +\max_\Gamma |c| \rp \Vert\rho\Vert_{H^1(\Gamma)} \, \Vert u\Vert_{H^1(\Gamma)}.
  \end{aligned}
\end{equation}
The conjugate linear form~$\cA$ is bounded on~$H^1(\Gamma)$ via the
Cauchy-Schwartz inequality.

We first consider Condition 1. In this case, constant functions are in the null
space of the PDE and the range is mean-zero functions, as in the case of the
Laplace-Beltrami problem. We shall therefore apply the Lax-Milgram theorem on
the space~$\Hoo$. The forms~$\cB(u,\rho)$ and~$\cA(\rho)$ are bounded on~$\Hoo$ for
the same reason that they are bounded on~$\Ho$. To see that the sesquilinear
form is coercive, we compute
    \begin{equation}
        -\cB(u,u) =   \int_\Gamma a \Vert\gradg u \Vert^2 \geq \min_\Gamma a \int_\Gamma \Vert\gradg u \Vert^2.
    \end{equation}
    Applying the Poincar\'e inequality then gives that
    \begin{equation}
      \begin{aligned}
      -\cB(u,u) &\geq  \frac{\min_\Gamma a}{1+C_\Gamma}\Vert\gradg u\Vert_{L^2(\Gamma)}^2 + \frac{\min_\Gamma a}{1+C_\Gamma}\Vert u\Vert_{L^2(\Gamma)}^2 \\ 
      &\geq \frac{\min_\Gamma a}{1+C_\Gamma} \Vert u\Vert_{H^1(\Gamma)}^2.
    \end{aligned}
    \end{equation}
Since the multiplicative constant on the right hand side is positive,~$\cB$ is
coercive. The Lax-Milgram theorem then gives that there exists a unique~$u\in
\Hoo$ that solves~\eqref{eq:weakform}.

We now consider Condition 2. Multiplying~$\cB(u,u)$ by~$e^{i\theta}$ and taking
the real part gives that
\begin{equation}
\begin{aligned}  
    \Re (-e^{i\theta}\cB(u,u))  &= \int_\Gamma \lp a \cos \theta \left\Vert\gradg u\right\Vert^2 - \Re ( c e^{i\theta}) |u|^2 \rp \\
    &\geq \cos \theta \min_\Gamma a \Vert\gradg u\Vert_{L^2(\Gamma)}^2 - \max_\Gamma \Re (c e^{i\theta})\Vert u\Vert_{L^2(\Gamma)}^2.
    \end{aligned}
\end{equation}
Since~$a$ is strictly positive,~$\theta\in \lp -\frac\pi2, \frac\pi2\rp$, and
the real part of~$c e^{i\theta} $ is strictly negative, we have that~$\cB$ is
coercive with constant~$$C_C=\min\{\cos\theta \min_\Gamma a ,-\max_\Gamma \Re (c
e^{i\theta}) \}  <0.$$ The Lax-Milgram Theorem then gives that  there exists a
unique~$u\in \Ho$ that solves~\eqref{eq:weakform}.
\end{proof}


\begin{theorem}
Suppose that~$\Gamma$ is a smooth surface and the coefficient functions are
smooth. If the solution~$u$ of~\eqref{eq:surf_PDE} exists and the right hand
side~$f$ is smooth, then~$u$ is also smooth.
\end{theorem}
\begin{proof}
    We prove this by mapping~\eqref{eq:surf_PDE} to an elliptic equation
    in~$\bbR^2$ and using standard regularity results. If~$\by:U\to \Gamma$ is
    local parameterization of part of~$\Gamma$, then the pull-back of the
    differential operator onto $U$ is
    \begin{equation}
      \mathcal E v := \lp \divg \lp a\gradg (v\circ \by)
      \rp \rp \circ \by^{-1} + (c\circ \by^{-1}) v.
    \end{equation}
Because~$\by$ is smooth,~$\mathcal E$ is an elliptic operator with smooth coefficients.
Since we also have that~$\mathcal E (u\circ \by) = f\circ \by$, Theorem 3 in Section 6.3.1 of
\cite{Evans2010} gives that~$u\circ \by$ is smooth in
the interior of the domain~$U$. We therefore have that~$u$ is smooth on the interior of the range of~$\bs y$. By choosing different parameterizations we can
cover~$\Gamma$ and prove that~$u$ is smooth everywhere.
\end{proof}

\subsubsection*{Helmholtz-type equations}
\label{sec:helm}


As a special case, we now address existence and uniqueness
for~\eqref{eq:surf_PDE} when~$c$ is real and positive, which corresponds to
equations of Helmholtz-type. This case is more complicated as the associated
sesquilinear form will not be coercive. For simplicity, we suppose that~$c >0$
is constant and not zero. In this case, we can use an argument very similar to
that used for the equivalent problem in the plane, see~\cite{Evans2010}.
When~$c$ is constant, equation~\eqref{eq:surf_PDE} is well-posed if and only
if~$-c$ is in the resolvent set of
\begin{equation}
    \mathcal L:= \divg a \gradg.
\end{equation}
We will study the resolvent set of~$\mathcal L$
by proving that its inverse, $\mathcal S:= \mathcal L^{-1}$, is a self-adjoint and negative semi-definite compact
operator. This will allow us to use the spectral theorem for compact operators,
Theorem 28.3 in~\cite{lax2002functional}, to characterize the resolvent set
of~$\mathcal L$. We now formalize this argument. 

\begin{theorem}[Analogue of Theorem 6.2.5 in \cite{Evans2010}]\label{thm:Helm_posed}
  Let~$\mathcal S$ be the map from functions~$g\in L^2_\text{mz}(\Gamma)$ to the
  mean-zero solution~$v$ of~$\mathcal L v = g$. If~$c$ is a non-zero constant
  function, then there exists a unique weak solution~$u\in H^1(\Gamma)$
  of~\eqref{eq:surf_PDE} if and only if ~$-1/c$ is not in the spectrum
  of~$\mathcal S$.
\end{theorem}
\begin{proof}
  Suppose that~$u$ and~$f\in L^2_\text{mz}(\Gamma)$. If we apply~$\mathcal S$ to
  both sides of~\eqref{eq:surf_PDE} and divide by~$c$, we find
  that~\eqref{eq:surf_PDE} becomes
  \begin{equation}
    \frac{1}{c}u + \mathcal Su = \frac{1}{c} \mathcal Sf. \label{eq:inverseequation}
  \end{equation}
  Since~$c$ was assumed to be constant,~\eqref{eq:inverseequation} (and
  thus~\eqref{eq:surf_PDE}) is well-posed in~$L^2_\text{mz}(\Gamma)$ if and only
  if~$-\frac 1c$ is not in the spectrum of~$\mathcal S$.
    
  We now suppose that~$f\in L^2(\Gamma)$. In this case, we
  write~$f=\tilde f + \hat f$, where~$\hat f$ is the mean of~$f$, and write the
  function~$u = \tilde u+\hat u$, where~$\hat u$ is the mean of~$u$. Since~$c$
  is constant,~\eqref{eq:surf_PDE} is equivalent to the uncoupled equations
  \begin{equation}
    \begin{aligned}
      \DO \tilde u + c\tilde u &= \tilde f,\\
      c\hat u &= \hat f.
    \end{aligned}
  \end{equation}
  Since~$c$ is non-zero, there always exists a unique $\hat u$ solving the
  second equation. The first equation is equivalent to the mean-zero case, and
  so has a unique solution if and only if~$-1/c$ is not in the spectrum
  of~$\mathcal S$, and we have the desired result.
\end{proof}
We now verify the spectrum of~$\mathcal S$ is countable.

\begin{theorem}[Analogue of Theorems 6.2.4 and 6.2.5 in
\cite{Evans2010}]\label{thm:closed_eval} Let~$\Gamma$ be a smooth and closed
surface and the function~$a$ be continuous and positive. The solution
operator~$\mathcal S$ is a compact map from~$L^2_\text{mz}(\Gamma)$ to itself.
Furthermore, the spectrum of~$\mathcal S$ is real, positive, and countable.
\end{theorem}
\begin{proof}
  We suppose that~$f,g\in L^2_\text{mz}(\Gamma)$ and that~$v=\mathcal Sg$. We
  also let
  \begin{equation}
    B(\mu,\rho) = -\int_\Gamma a \, \gradg \mu \cdot \gradg \bar \rho.
  \end{equation}
  We begin by proving that~$\cS$ is a bounded map from~$L^2_\text{mz}(\Gamma)$
  to~$H^1_\text{mz}(\Gamma)$. By the definition of~$\cS$,
  \begin{equation}
    B(v,v) = (g,v)_{L^2}.
  \end{equation}
  The coerciveness of~$B$ discussed in the proof of Theorem~\ref{thm:well-posed}
  and the Cauchy-Schwartz inequality then give that
  \begin{equation}
    \begin{aligned}
        c\Vert v\Vert_{H^1(\Gamma)}^2 &\leq  B(v,v) \\
        &= (g,v)_{L^2} \\
        &\leq \Vert g \Vert_{L^2(\Gamma)} \, \Vert v\Vert_{L^2(\Gamma)} \\
        &\leq \Vert g\Vert_{L^2(\Gamma)} \, \Vert v\Vert_{H^1(\Gamma)} 
    \end{aligned}
  \end{equation}
  The operator~$\mathcal S$ is thus a bounded map from~$L^2_\text{mz}(\Gamma)$
  to~$H^1_\text{mz}(\Gamma)$. Since~$H^1_\text{mz}(\Gamma)$ is compactly
  embedded in~$L^2_\text{mz}(\Gamma)$, we have that~$\cS$ is a bounded and
  compact map from~$L^2_\text{mz}(\Gamma)$ to itself.
    
  We now prove that~$\mathcal S$ is Hermitian on~$L^2_\text{mz}(\Gamma)$. By the
  definition of~$Sf$,
  \begin{equation}
    B(v,Sf)  = (v,f)_{L^2}  = (\mathcal Sg,f)_{L^2}.
  \end{equation}
  Similarly,~$B(\cS f,v)=(\cS f,g)_{L^2}$. The conjugate symmetry of~$B$ then gives
  that~$\mathcal S$ is a Hermitian operator on~$L^2_\text{mz}(\Gamma)$.

  Finally,~$\cS$ is a negative semidefinite operator because
  \begin{equation*}
    (\mathcal Sg,g)_{L^2}=B(v,v) = - \int_\Gamma a\, \Vert\gradg v\Vert^2 \leq 0
  \end{equation*} for all~$g\in L^2_\text{mz}(\Gamma)$.
    
  By the spectral theorem for self adjoint compact operators
  (see~\cite{lax2002functional}), assembling the above observations gives that
  the spectrum of~$\mathcal S$ is composed of a countable number of isolated
  negative eigenvalues and zero.
\end{proof}

\subsection{Elliptic problems without advection}
\label{sec:gen_elliptic}

In this section, we introduce a parametrix for~\eqref{eq:surf_PDE} and extend
the argument in Section~\ref{sec:LBparametrix} to find an integral equation form
of~\eqref{eq:surf_PDE}. We will assume that~$a$ is a smooth and positive
function. We shall also assume that~$c$ is a smooth complex-valued function and
that~\eqref{eq:surf_PDE} has a unique solution (see
Theorems~\ref{thm:well-posed},~\ref{thm:Helm_posed}, and~\ref{thm:closed_eval}). For the same reasons
discussed in Section~\ref{sec:LB_posed}, if~$c= 0$ we will also assume that~$u$
and~$f$ are mean-zero. We also suppose for now that~$\bs b = \bs 0$, and discuss the
case with non-zero~$\bs b$ in the following section.

Analogous to the Laplace-Beltrami case, we let~$G(r;\tilde
a,\tilde c)$ be the Green's function for the constant coefficient elliptic
equation in the plane:
\begin{equation}
    \lp \tilde a\Delta_2  + \tilde c \rp  G(r;\tilde a,\tilde c) = \delta(r),
\end{equation}
where~$\tilde a$ and~$\tilde c$ are arbitrary constants.

Note that if~$\tilde c = 0$, then the Green's function~$G$ is given by
\begin{equation}
  G(r;\tilde a,\tilde c) = 
  \frac{1}{2\pi \tilde a} \log r,
\end{equation}
and if~$\tilde c \neq 0$, then
\begin{equation}
  G(r;\tilde a,\tilde c) =
  \frac{i}{4\tilde a} \, H^{(1)}_0 \lp \sqrt{\tilde c /\tilde a }\;r \rp ,
\end{equation}
where~$H^{(1)}_0$ is the zeroth-order Hankel function of the first kind.  

We use this Green's function to define a parametrix for the general variable
coefficient case:
\begin{equation}
  \label{eq:diff_param}
    K(\bs x,\bs x') := G(\Vert\bx - \bx'\Vert; a(\bx'),c(\bx ')),
\end{equation}
for~$\bx,\bx' \in \mathbb R^3$. Before proving that~$K$ is a parametrix
for~\eqref{eq:surf_PDE}, however, we note that we could have chosen~$K$ to
depend on the values of~$a$ and~$c$ at~$\bx$ instead of~$\bx'$. We made the
above choice so that~$K(\bs x,\bs x')$ is defined for~$\bx$ in a neighborhood
of~$\Gamma$. This allow us to apply the formula~\eqref{eq:diff_remainderLB} to
differentiate~$K$.

We also note that the Green's function must have a branch cut due to the square
root in the argument and logarithmic terms in the Hankel function expansions. We
therefore add the assumption that~$c$ is contained in some wedge in the complex
plane. In what follows, since~$r = \Vert \bx - \bx' \Vert$ is real, we can avoid the branch cut in the Green's
function by choosing the branch cut to be outside that wedge.

Since we will need them below, we note that for small~$z \in \mathbb C$ 
we have the following asymptotic expansions for~$H^{(1)}_0$ and its derivatives:
\begin{equation}
  \label{eq:asymptotic_H}
  \begin{aligned}
    H^{(1)}_0(z) &= \frac{2i}{\pi}\log z +\cO(1),\\
    \frac{d}{dz}   H^{(1)}_0(z) &= \frac{2i}{\pi z}+\cO(z\log z),\\
   \frac{d^2}{dz^2} H^{(1)}_0(z) &=-\frac{2i}{\pi z^2}+\cO(\log z). 
  \end{aligned}
\end{equation}
We now verify that~$K$ is a candidate to be parametrix for~\eqref{eq:surf_PDE}.
Afterwards, we shall prove properties of the remainder function such that~$K$ is
indeed a parametrix.

\begin{theorem}
  \label{thm:parametrix}
The function~$K$ satisfies
\begin{equation}
    \divg a(\bx) \gradg K(\bs x,\bs x') +c(\bx)K(\bs x,\bs x') = \delta(\bx - \bx') + R(\bx,\bx'),\label{eq:diff_form}
\end{equation}
where the derivatives are taken with respect to~$\bx$. Furthermore,
wherever~$K(\bs x, \bs x')$ is smooth the remainder function is
\begin{multline}
  \label{eq:explicit_remainder}
 R(\bs x,\bx')
 =\gradg a(\bx) \cdot \hat \br \, G'(r ; a(\bx'),c(\bx ')) -a(\bx) \, G''(r;a
 (\bx'),c(\bx ')) \frac{ (\bs n(\bx) \cdot \br )^2}{r^2}  \\
  +a(\bx) \, G'(r;a(\bx'),c(\bx ')) \frac{(\bs n(\bx)\cdot \br )^2}{r^3} 
  - 2a(\bx)H(\bx) \bs n(\bx)\cdot \hat\br \, G'(r; a(\bx'),c(\bx ')) \\
 + \lp c(\bx)-\frac{a(\bx)c(\bx')}{a(\bx')}\rp G(r; a(\bx'),c(\bx ')),
\end{multline}
where~$\br = \bx - \bx'$,
$r = \Vert \br \Vert$, ~$\hat \br = \br/r$,
\begin{equation}
  G'(r;a,c) = \partial_r G(r;a,c) , \qquad\text{and}\qquad  G''(r;a,c) = \partial_r^2G'(r;a,c) .
\end{equation}
\end{theorem}
\begin{proof}
The proof of~\eqref{eq:diff_form} is exactly the same as the proof
of~\eqref{eq:diff_formLB} because of the asymptotic formula for the gradient
of~$G$, given in~\eqref{eq:asymptotic_H}.

In order to verify~\eqref{eq:explicit_remainder} we assume that~$K(\bx,\bx')$ is
smooth at~$\bx$ and compute:
\begin{equation*}
    \divg a(\bx)\gradg K(\bx,\bx') = \gradg a(\bx) \cdot \gradg K(\bx,\bx')  + a(\bx)\LB K(\bx,\bx').
\end{equation*}
We then use~\eqref{eq:diff_remainderLB} and the fact that~$\gradg =\nabla - \bs
n(\bx) (\bs n(\bx)\cdot \nabla)$ to find that for all~$\bx \neq \bx'$,
\begin{equation}
  \label{eq:diff_remainder}
    \DO K= \gradg a \cdot \nabla K  - 2aH \partial_{\bs n(\bx)} K + a(\Delta K - \partial_{\bs n(\bx)}^2 K).
\end{equation}


In order to continue, we compute some derivatives of~$G(r;a(\bx'),c(\bx '))$. If~$\bs v$ is a unit vector, then
\begin{equation}
     \partial_{\bs v} G\lp r;a(\bx'),c(\bx ')\rp = G'\lp r;a(\bx'),c(\bx ')\rp \, \partial_{\bs v} r ,\label{eq:grad_K}
\end{equation}
and
\begin{equation}
     \partial_{\bs v}^2 G(r;a(\bx'),c(\bx ')) 
     = G''(r;a(\bx'),c(\bx ')) \lp \partial_{\bs v} r\rp^2 + G'(r;a(\bx'),c(\bx ')) \, 
     \partial_{\bs v}^2 r.
\end{equation}
The required derivatives of the distance function are
\begin{equation}
    \partial_{\bs v} r = \frac{\bs v\cdot \br}{r} \quad \text{and} \quad \partial_{\bs v}^2 r=\frac{\bs v\cdot\bs v}{r}-\frac{(\bs v\cdot\br )^2}{r^3}.
\end{equation}
Grouping the second derivatives into the tangential Laplacian gives that
\begin{multline}
  \label{eq:tmp_der}
    \Delta G(r;a(\bx'),c(\bx ')) - \partial_{\bs n(\bx)}^2 G(r;a(\bx'),c(\bx ')) \\
    = G''(r;a(\bx'),c(\bx '))  \lp 1-\frac{ (\bs n(\bx)\cdot\br)^2}{r^2}\rp
    +G'(r;a(\bx'),c(\bx '))\lp\frac{1}{r} + \frac{(\bs n(\bx)\cdot\br)^2}{r^3} \rp.
 \end{multline}
 We now note that by the definition of~$G$, we have that
 \begin{equation}
     a(\bx') \, G''(r;a(\bx'),c(\bx '))
          +\frac{a(\bx')}{r}G'(r; a(\bx'),c(\bx '))  
          + c(\bx ') \, G(r; a(\bx'),c(\bx ')) = \delta(r).
 \end{equation}
 The tangential Laplacian thus becomes
 \begin{multline}
  \label{eq:tang_der}
     \Delta G(r;a(\bx'),c(\bx ')) - \partial_{\bs n(\bx)}^2 G(r;a(\bx'),c(\bx ')) \\
     = - G''(r;a(\bx'),c(\bx '))\frac{(\bs n(\bx)\cdot\br)^2}{r^2} 
  +\frac{(\bs n(\bx)\cdot\br)^2}{r^3} G'(r;a(\bx'),c(\bx ')) \\ -\frac{c(\bx ')}{a(\bx')}G(r;a(\bx'),c(\bx '))
 \end{multline}
 for all~$\bx'\neq\bx$. If we plug~\eqref{eq:diff_remainder} and~\eqref{eq:tang_der} into~\eqref{eq:diff_form}, we get that~$R(\bx,\bx')$ is given by~\eqref{eq:explicit_remainder}.
\end{proof}

When the coefficients~$a$ and~$c$ are constant, the remainder has a simpler
form, which is clear from~\eqref{eq:explicit_remainder}, and the following Hankel function identities:
\begin{equation}
    \partial_z H^{(1)}_0(z)=-H^{(1)}_1(z)\qquad \text{and}\qquad \partial_z^2 H^{(1)}_0(z)-z^{-1}\partial_z H^{(1)}_0(z)=H^{(1)}_2(z).
\end{equation}

\begin{corollary}
If~$a=a_0$ is a non-zero constant and~$c$ is zero, then~$R(\bs x,\bx')= \RLB(\bx,\bx')$.
If~$a=a_0$ and~$c=c_0$ are both non-zero constants, then~\eqref{eq:explicit_remainder} becomes
\begin{equation}
  \label{eq:bdd_remainder_helm}
    R(\bs x,\bs x')=-\frac{c_0}{4 ia_0}H^{(1)}_2\lp\sqrt{\frac {c_0}{a_0}}r\rp\frac{ \left(\bs n(\bx)\cdot (\dx)\right)^2}{r^2} +\frac{\sqrt{c_0}}{2 i\sqrt{a_0}}H^{(1)}_1\lp\sqrt{\frac {c_0}{a_0}}r\rp    H(\bx)\frac{\bs n\cdot(\dx)}{\dxnorm}.
\end{equation}
\end{corollary}

In order to prove that an integral operator with kernel~$R$ is compact, we
provide the following theorem about the behavior of the remainder term
as~$\bx\to \bx'$:
\begin{theorem}\label{thm:remainder}
If the coefficient~$a$ is constant, then~$R(\bx,\bx')=\cO(1)$ as~$\bx'\to\bx$.
Otherwise,~$R(\bx,\bx')=\cO(\Vert\bx-\bx'\Vert^{-1})$ as~$\bx'\to\bx$.
The function~$K$ is therefore a paramterix in either case.
\end{theorem}
\begin{proof}
When~$a$ is constant and~$c$ is zero, the remainder function is~$R(\bs x,\bx')= \RLB(\bx,\bx')$ and Theorem~\ref{thm:remainderLB} gives that~$R(\bs x,\bx')$ is bounded. 
When~$c$ is not zero, the asymptotic formulas for the derivatives of Hankel functions and~\eqref{eq:explicit_remainder} gives that~$R$ is asymptotic to the~$c=0$ remainder as~$\bx'\to\bx$, so the remainder has the same, bounded behaviour. 

When~$a$ is not constant, we add the term~$\frac{1}{a(\bx')}\gradg a \cdot\gradg
K$ to the remainder function. We may see that this term
is~$\cO (r^{-1})$ through~\eqref{eq:grad_K} and the asymptotic
formula for the derivative of~$G(z;a,c)$. The other non-regular term is
\begin{equation*}
    \lp c(\bx) - \frac{a(\bx)c(\bx')}{a(\bx')} \rp G(\Vert\bx-\bx'\Vert;a(\bx'),c(\bx ')),
\end{equation*}
which is~$\cO(r\log r)$
as~$\bx'\to\bx$. Therefore the~$\cO\lp r^{-1}\rp$ term is the
dominant singularity and~$K$ is a parametrix.
\end{proof}
\begin{remark}
  The stronger singularity when~$a$ is not constant is expected, as it is the
  same as is observed when the equivalent parametrix is used for the
  equation~$\nabla\cdot a\nabla u=f$ in the plane~\cite{Beshley2018}.
\end{remark}
An example of what the remainder kernel looks like when~$c$ is real and not
constant is shown in Figure~\ref{fig:remainder}. As we can see in the
figure, if~$c$ approaches zero, then the remainder function is singular. We
give the exact nature of this singularity in the following theorem, which
follows directly from~\eqref{eq:asymptotic_H}.
\begin{figure}
  \centering
  \includegraphics[width=1\linewidth]{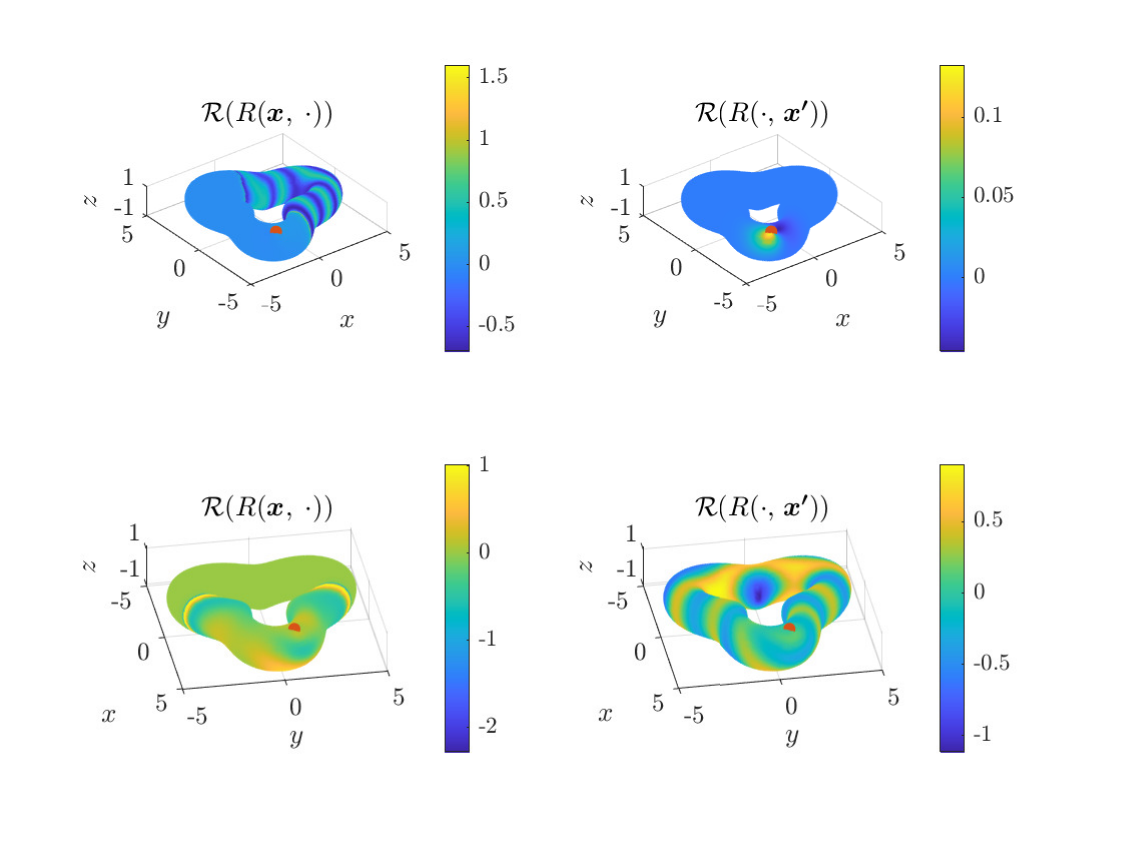}
  \caption[The remainder kernel for the general surface elliptic problems]{The
    figure shows the remainder kernel for an example surface PDE ~$c(x,y,z)=2x$
    and~$a=1$. We see that the remainder function has a singularity at points
    where~$c(\bx)=0$ (the left images) and is a bounded function of~$\bx$ (the
    right images). It is also clear that the remainder inherits the character of
    the PDE: it is oscillatory when~$c>0$ and rapidly decays when~$c<0$.}
    \label{fig:remainder}
\end{figure}

\begin{theorem}
Suppose~$\Gamma_c:=\{\bx\in\Gamma:c(\bx)=0\}$ is not empty and not all of~$\Gamma$. If~$\tilde \bx\in\partial\Gamma_c$ and~$\tilde\bx\neq \bx$, then both~$K(\bx, \bx')$ and~$R(\bx, \bx')$
are~$\cO(\log c(\bx'))$ as~$\bx'\to\tilde\bx$ from outside~$\Gamma_c$.
\end{theorem}
We are now ready to use our parametrix to find a second kind integral equation form of~\eqref{eq:surf_PDE}.
\begin{theorem}\label{thm:second_kind}
If~$a$ and~$c$ are smooth, then the integral equation
\begin{equation}
  \label{eq:IE}
  \sigma(\bs x) + \int_\Gamma R(\bs x,\bs x')\, \sigma(\bx') \, da(\bx') =  f(\bx), \qquad \text{for } \bx \in \Gamma
\end{equation}
is a second kind Fredholm integral equation on~$L^2(\Gamma)$. Furthermore,
if~$f\in L^2(\Gamma)$ and~$\sigma$ is the solution of~\eqref{eq:IE}, then the
function   
\begin{equation}
u(\bx)=\int_\Gamma K(\bx,\bx') \, \sigma(\bx') \, da(\bx')
\label{eq:representation}
\end{equation}
is the unique solution of~\eqref{eq:surf_PDE}.
\end{theorem}
\begin{proof}
We begin by proving that the integral operator 
\begin{equation}
    \mathcal R [\sigma] (\bs x)=\int_\Gamma R(\bs x,\bs x') \, \sigma(\bx') \, 
    da(\bx')
\end{equation}
is a compact map from~$L^2(\Gamma)$ to itself. The remainder function has two
singularities: (1) when~$\bx'\to\bx$ and, if it is non empty, (2) when~$\bx'$
approaches~$\gamma_c$, the boundary of the
subset~$\Gamma_c=\left\{ \tilde{\bx}\in \Gamma:c(\tilde{\bx})=0\right\}$. In order to simplify our
discussion, we start by ignoring the contribution of an~$\epsilon$ neighborhood
of~$\gamma_c$. This allows to focus on the singularity when $\bx' \to \bx$. What
remains is an integral operator~$\mathcal R_\epsilon$ with the kernel
\begin{equation}
    R_\epsilon (\bx,\bx'):=\begin{cases} R(\bx,\bx')& \operatorname{dist}(\bx',\gamma_c)>\epsilon\\
    0  & \text{otherwise}
    \end{cases}.
\end{equation}

For any~$\epsilon>0$,~$R_\epsilon$ is bounded apart from
a~$O(r^{-1})$ singularity at~$\bx'\approx\bx$
(Theorem~\ref{thm:remainder}), so Proposition 3.11 in~\cite{Folland1995} gives
that the operator~$\mathcal R_\epsilon$ is compact on~$L^2(\Gamma)$. The
singularity that occurs when~$c(\bx')\to 0$ is integrable, so by the argument in
the proof of Proposition 3.10 in~\cite{Folland1995},~$\mathcal R$ is
the~$L^2(\Gamma)$-norm limit of~$\mathcal R_\epsilon$. The operator~$\mathcal R$
is therefore a compact map from~$L^2(\Gamma)$ to itself. The
equation~\eqref{eq:IE} is thus a second kind Fredholm integral equation.
  
The proof that the~$u$ defined in~\eqref{eq:representation}
solves~\eqref{eq:surf_PDE} is exactly the same as in
Theorem~\ref{thm:second_kindLB}, except that it starts from the the weak form
of~\eqref{eq:surf_PDE} :
\begin{equation}
  \begin{aligned}
     B(u,\phi):=& \int_\Gamma \lp -a\gradg u \cdot \gradg\bar\phi 
     + c\, u \bar\phi \rp\\
     =& \int_\Gamma f\, \bar\phi,
  \end{aligned}
\end{equation}
for all~$\phi \in H^1(\Gamma)$.
\end{proof}

\subsection{Elliptic problems with advection}
\label{sec:nonzerob}
In some cases, we wish to consider surface PDEs with a first order term, such as
the semi-discretization of an advection-diffusion equation with an
implicit time stepping method. In this section, we describe how the above
integral equation form can easily be extended to this case, which corresponds to
introducing a non-zero~$\bs b$ into~\eqref{eq:surf_PDE}. 

\begin{theorem}\label{thm:advectionparametrix}
  The function~$K$ defined in~\eqref{eq:diff_param} is a parametrix
  for~\eqref{eq:surf_PDE} with remainder function
  \begin{equation}
    \label{eq:adv_remainder}
        R_{\bs b}(\bx,\bx'):= R(\bx,\bx') + \bs b \cdot \gradg K(\bx,\bx'),
      \end{equation}
      where the function~$R$ is defined in~\eqref{eq:bdd_remainder_helm}.
\end{theorem}
\begin{proof}
Adding a lower-order term to the equation does not change the proof of the fact
that~$K$ satisfies~\eqref{eq:parametrix_def}. The advection term is additive,
so~\eqref{eq:adv_remainder} follows immediately from the definition of the
remainder function. Finally, we note that the term~$\bs b \cdot \gradg K$ has a
singularity of~$\cO(\Vert\bx-\bx'\Vert^{-1})$ as~$\bx'\to \bx$, which still
allows for~$K$ to be a parametrix.
\end{proof}
With the knowledge that~$K$ is a parametrix, we can use it to derive an integral
equation form of~\eqref{eq:surf_PDE} in the same manner as
Theorem~\ref{thm:second_kind}.

\begin{theorem}
The integral equation
\begin{equation}
  \label{eq:IE_adv}
    \sigma(\bs x) + \int_\Gamma R_{\bs b}(\bs x,\bs x')\, \sigma(\bx') \, da(\bx') =  f(\bx), \qquad \text{for } \bx \in \Gamma,
\end{equation}
is a second-kind Fredholm integral equation on~$L^2(\Gamma)$.  Furthermore, if~$f\in L^2(\Gamma)$ and~$\sigma$ is the solution of~\eqref{eq:IE_adv}, then the function~$u$ given by~\eqref{eq:representation}
is the unique solution of~\eqref{eq:surf_PDE}.
\end{theorem}
\begin{proof}
    Since the additional term in~$R_{\bs b}$ is not more singular than~$R$, the integral
    equation~\eqref{eq:IE_adv} will still be second-kind. The proof of the
    second part follows in the same manner as in the case without advection.
\end{proof}

\section{Elliptic boundary value problems on surfaces}
\label{sec:BVP} 

We now address the situation of surface PDEs on open surfaces with boundary
conditions applied at surface edges. We suppose that~$\Gamma$ is an open subset of a
larger closed surface~$\tilde \Gamma$ whose boundary is comprised of smooth and
non-intersecting closed curves, see Figure~\ref{fig:island} for an example of
such a surface. Throughout this section, we let~$\gamma:=\partial\Gamma$ be the
boundary of the surface and~$\bs \nu$ be the binormal vector
to~$\gamma$, i.e. the normalized vector field on~$\gamma$ that
is defined to be tangent to~$\Gamma$, perpendicular to~$\gamma$, and to
point away from~$\Gamma$. We shall also assume that~$\Gamma$ does not include it's boundary, and so is an open subset of~$\tilde \Gamma$.

We start by briefly discussing the existence of solutions to such problems. We
then consider problems with Neumann boundary conditions imposed at~$\Gamma$, and
then move on to the case where Dirichlet boundary conditions are imposed
on~$\gamma$.

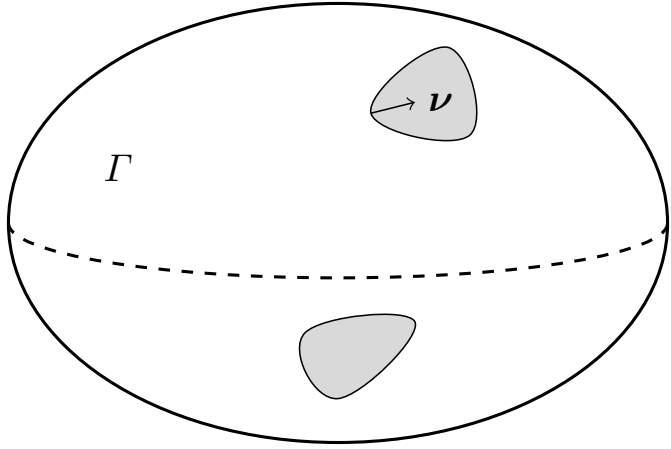
\begin{figure}
  \centering
  \resizebox{.6\linewidth}{!}{
    \begin{tikzpicture}
        \draw[thick, black] (0,0) [partial ellipse=0:360:3cm and 2cm];
        \draw[thick, black,dashed] (0,0) [partial ellipse=180:360:3cm and 0.5cm];

        \filldraw[fill=gray!30, draw=black] plot [smooth cycle, tension=0.8] coordinates {(1.2,0.8) (1,1.6) (0.3,1)};

                \filldraw[fill=gray!30, draw=black] plot [smooth cycle, tension=0.8] coordinates {(0.7,-0.9) (0,-1.6) (-0.3,-1)};
        \draw[->, draw=black] (0.3,1) -- (0.7,1.1) node[right]{$\bs \nu$};

        \node(label) at (-2,.5) {$\Gamma$};
      \end{tikzpicture}
      }
    \caption[A surface with boundaries]{This figure shows an example of an open
      surface~$\Gamma$, which is a subset of the ellipsoid~$\tilde\Gamma$ with
      two patches removed. It also shows the binormal vector~$\bs\nu$, which is
      normal to the~$\partial\Gamma$, but tangent to $\Gamma$.}
    \label{fig:island}
\end{figure}

\subsection{Existence of solutions}

Before we continue, we note that
Theorems~\ref{thm:well-posed},~\ref{thm:Helm_posed}, and~\ref{thm:closed_eval}
can be extended to problems on open surfaces where boundary conditions are
applied on the edge~$\gamma$. In order to study these in detail, we first define
the trace operator, the weaker analogue of the restriction operator. 
\begin{definition}[Trace]
  Suppose~$\by$ is a smooth local parameterization of~$\bar \Gamma$ and that~$\gamma_{\by}$ is the intersection of $\gamma$ and the range of~$\by$. If~$u\in H^1(\Gamma)$, then the trace
  of~$u$ is defined as
\begin{equation}
    \trace_{\gamma_{\by}}u:= \trace_{\by^{-1}(\gamma_{\by})} (u\circ \by)\circ \by^{-1},
\end{equation}
where~$\trace_{\by^{-1}(\gamma_{\by})}$ is the usual trace operator for curves in~$\bbR^2$ (see Theorem~18.1 of~\cite{Leoni2017}).

\end{definition}
By Theorem~18.1 of~\cite{Leoni2017}, we know that the range of the trace
operator is~$H^{\frac{1}{2}}(\gamma_{\bs y})$, which may be defined through the
use of a parameterization and the usual one dimensional Sobolev
space~$H^{\frac12}([0,L])$. To define the trace onto all of~$\gamma$, we can use a partition of unity to combine the traces onto individual portions $\gamma_{\by_i}$.
For this to be well defined, we must show that the trace is
parameterization-independent.

\begin{theorem}\label{thm:trace-welldefined}
    The operator~$\operatorname{tr}_{\gamma_{\by}}$ is independent of the local parameterization used in its definition.
\end{theorem}
\begin{proof}
Let~$\by$ and~$\bs z$ be two smooth local parameterizations that map the
domains~$U$ and~$V$ onto~$\bar{\tilde\Gamma}\subset\bar\Gamma$, and
such that~$\gamma_{\by} = \gamma_{\bs z}$. We begin by showing
that~$C^\infty(\tilde\Gamma)$ is dense in~$H^1(\tilde\Gamma)$. By the
definitions of the surface gradient and~$H^1(\tilde\Gamma)$, we have that
\begin{equation}
  \begin{aligned}
    H^1(\tilde\Gamma)&=\left\{f\in L^2(\tilde\Gamma)| f\circ\by \in H^1(U)\right\},\\
    C^\infty(\tilde\Gamma)&=\left\{f\in C(\tilde\Gamma)| f\circ\by \in C^\infty(U)\right\}.
  \end{aligned}
\end{equation} 
Standard density results in the plane then imply that~$C^\infty(\tilde\Gamma)$
is dense in~$H^1(\tilde\Gamma)$,  see \cite{Evans2010}.

Next, we note that if~$v\in C^\infty(\tilde\Gamma)$,
then
\begin{equation}
    \trace_{\by^{-1}(\gamma_{\by})} (v\circ \by)\circ \by^{-1} = \trace_{\bs z^{-1}(\gamma_{\bs z})} (v\circ \bs z)\circ \bs z^{-1},
\end{equation}
and therefore~$\trace_{\gamma_{\by}}v$ is independent of the smooth local
parameterization used in its definition. The density of
of~$C^\infty(\tilde\Gamma)$ then implies that the trace operator is independent
of the local parameterization.
\end{proof}
With this theorem, we can extend our definition of the trace onto the whole of~$\gamma$. Further, by choosing a standard set of
parameterizations, it is clear that the trace operator maps~$\Ho$
to~$H^{\frac12}(\gamma)\subset L^2(\gamma)$.

With the trace defined, we are now ready to prove the well-posedness of some
boundary value problems along surfaces.
When the Neumann boundary conditions 
\begin{equation}
     \bs \nu\cdot \gradg u|_{\gamma}=f_\gamma
\end{equation}
are applied, the proofs of the above existence results, Theorems~\ref{thm:well-posed}, \ref{thm:Helm_posed}, and~\ref{thm:closed_eval}, are the same except that the conjugate linear form is
replaced by
\begin{equation}
  L(\rho):=\int_\Gamma \bar \rho \, f+\int_\gamma \bar \rho \, f_\gamma.
\end{equation}
This conjugate linear form is still bounded on~$H^1(\Gamma)$ because the trace
operator is bounded.

We now consider the case where Dirichlet boundary conditions
\begin{equation}
    \left.u\right|_{\gamma} = f_\gamma,
\end{equation}
are applied. The proofs for when $h= 0$ are the same as the case without
boundaries, except that $u$ to must be restricted to live in the subspace of
trace-zero functions. For the inhomogeneous case, we must first show that for
any $f_\gamma\in C^1(\gamma),$ there exists a $u_{f_\gamma}\in \Ho$, whose trace is $f_\gamma$.

To see that~$u_{f_\gamma}$ exists, let~$\Pi(\bx)$ be the point on~$\gamma$ closest
to~$\bx$ in the Euclidean norm and let~$\epsilon(\Pi(\bx))$ be a distance such
that if $\operatorname{dist}(\bx,\gamma)<\epsilon(\Pi(\bx))$, then $\Pi(\bx)$ is
unique. Finally, let~$\eta(r)$ be a smooth function on~$[0,\infty)$ that is 1 at
$r=0$ and that is identically zero for $r>1$. A suitable choice of~$u_{f_\gamma}$ can be
written as
\begin{equation*}
  u_{f_\gamma}(\bx) = f_\gamma(\Pi(\bx)) \,
  \eta\lp\frac{\Vert\bx-\Pi(\bx)\Vert^2}{\epsilon(\Pi(\bx))^2}\rp.
\end{equation*}
This function will be continuously differentiable on the closure of~$\Gamma$
because the surface is smooth and $\Pi(\bx)$ will vary smoothly on the support of~$u$. This will be enough to guarantee that~$u_{f_\gamma}\in\Ho$. By the
linearity of~\eqref{eq:surf_PDE} and~\eqref{eq:Dirichlet_Bc}, we can then apply
the analogues of Theorems~\ref{thm:well-posed},~\ref{thm:Helm_posed},
and~\ref{thm:closed_eval} to $u-u_{f_\gamma}$ and see that the solution exists if and
only if it exists in the case where~$f_\gamma= 0$.

\subsection{Neumann boundary conditions}

In this section, we derive an integral equation supposing that Neumann
boundary conditions are imposed on~$\gamma$:
\begin{equation}
    \bs \nu\cdot \gradg u|_{\gamma}=f_\gamma, \label{eq:Neumann_BCs}
\end{equation}
where~$f_\gamma$ is given data. Our goal in this section will be to use the
parametrix~\eqref{eq:diff_param} to convert
equations~\eqref{eq:surf_PDE} and~\eqref{eq:Neumann_BCs} into a system
of integral equations.

In order to enforce the boundary conditions, we must include a boundary term in
our representation of the solution~$u$. The new augmented representation is
\begin{equation}
    u = \cK_{\Gamma}[\sigma] + \cK_{\gamma} [\mu] , \label{eq:BC_rep}
\end{equation}
where~$\cK_\Gamma$ is the surface operator
\begin{equation}
    \cK_{\Gamma}[\sigma] (\bx) := \int_{\Gamma}K(\bx,\bx') \, \sigma(\bx') \, da(\bx'), \quad\forall \bx\in\Gamma
\end{equation}
and~$\cK_\gamma$ is the edge operator
\begin{equation}
   \cK_{\gamma} [\mu] (\bx):= \int_{\gamma}K(\bx,\bx')\, \mu(\bx') \, ds(\bx'), \quad\forall \bx\in\Gamma.
\end{equation}
%
%
The following lemma gives how the differential operator in~\eqref{eq:surf_PDE}
acts on this representation.

\begin{lemma}\label{lem:surf_term}
If~$u$ is given by~\eqref{eq:BC_rep},~$\sigma\in L^2(\Gamma)$, and~$ \mu\in L^2(\gamma)$, then
\begin{equation}
  \label{eq:diff_repr}
    \DO u +\bs b\cdot\gradg u+ cu = \sigma+ \cR_{\Gamma}[\sigma] + \cR_{\gamma} [\mu],
\end{equation}
in the interior of~$\Gamma$, where
\begin{equation}
  \begin{aligned}
    \cR_{\Gamma}[\sigma] (\bx) :=& \int_{\Gamma}R(\bx,\bx') \, \sigma(\bx') \, da(\bx'), \\
    \cR_{\gamma}[\mu] (\bx):=& \int_{\gamma}R(\bx,\bx') \, \mu(\bx')\, ds(\bx'), \qquad
     \text{for all } \bx\in\Gamma.
   \end{aligned}
 \end{equation}
\end{lemma}
\begin{proof}
  We note that
  $$\cK_{\Gamma}[\sigma](\bx)= \int_{\tilde\Gamma} K(\bx,\bx') \,
  \tilde\sigma(\bx') \, da(\bx'),$$ where~$\tilde\sigma$ is the
  zero-extension of~$\sigma$ to the rest of~$\tilde
  \Gamma$. Theorem~\ref{thm:second_kind} thus gives the first two
  terms of~\eqref{eq:diff_repr}. For the other term, we note that
  since each~$\bx\in\Gamma$ is a finite distance away from~$\gamma$,
  the kernel of~$\cK_\gamma$ is smooth in a neighborhood of $\bx$ and
  we may simply pull the derivative inside the integral.
\end{proof}

To enforce the boundary conditions, we need the following two lemmas.
\begin{lemma}\label{lem:gradK}
If~$\sigma\in L^2(\Gamma)$, then~$\cK_\Gamma[\sigma] \in H^2(\Gamma)$. Further, if~$\sigma\in L^\infty(\Gamma)$ then
\begin{equation}
  \label{eq:normal_limit}
  \lim_{\bx\to\tilde\bx\in\gamma} {\bs\nu}(\tilde\bx)\cdot\gradg \cK_\Gamma[\sigma]
  (\bx) = \int_{\Gamma} \bs \nu(\tilde\bx)\cdot\gradg K(\bx,\bx')\,
  \sigma(\bx') \, da(\bx').
\end{equation}
\end{lemma}
\begin{proof}
  By the previous lemma, it is clear
  that~$ \lp\DO +\bs b\cdot\gradg + c\rp \cK_\Gamma [\sigma]\in L^2(\Gamma)$.
  Theorem 6.30 of~\cite{Warner2013} then gives
  that~$\cK_\Gamma[\sigma] \in H^2(\Gamma)$.

  To compute the gradient of~$K_\Gamma$, we extend our integral operators to be
  defined on~$\tilde\Gamma$ by extending~$\sigma$ by zero. We can then repeat
  the argument used in the proof of Theorem~\ref{thm:second_kindLB} to find that
  \begin{equation}
    \label{eq:gradKgamma}
    \gradg \cK_\Gamma[\sigma](\bx) =  \int_\Gamma \int_\Gamma \gradg K(\bx,\bx')
    \, \sigma(\bx') \, da(\bx').
  \end{equation}

  We now assume that~$\sigma$ is in~$L^\infty(\Gamma)$. Proposition 3.12
  in~\cite{Folland1995} applied on~$\tilde\Gamma$ gives
  that~$\gradg \cK_\Gamma[\sigma]$ is a continuous vector field on~$\tilde\Gamma$,
  and so the limit can simply be evaluated. Dotting the limiting value
  with~$\bs \nu(\tilde\bx)$ gives~\eqref{eq:normal_limit}.
\end{proof}
\begin{lemma}\label{lem:jumpcdt}
If~$\mu\in C(\gamma)$, then
\begin{equation}
  \lim_{\bx\to\tilde{\bx}\in \gamma} \bs \nu(\tilde\bx) \cdot \gradg \cK_{\gamma}[\mu]
  = -\frac1{2a(\tilde\bx)} \mu(\tilde\bx) +\cK'_{\gamma}[\mu] (\tilde\bx) \label{eq:neum_lim}
\end{equation}
for all~$\tilde\bx\in\gamma$,
where 
\begin{equation}
  \cK'_{\gamma} [\mu] (\bx)= \int_{\gamma}\bs \nu(\bx) \cdot \gradg K(\bx,\bx')
  (\bx') \, ds(\bx'), \qquad\text{for all } \bx\in\gamma.
\end{equation}

\end{lemma}
\begin{proof}
  We begin by showing~\eqref{eq:neum_lim} in the Laplace-Beltrami case.

  Let~$\by = \by(s,t)$ be a smooth parameterization of~$\Gamma$
  with domain~$V\subset\bbR^2$ such that~$\Gamma_V$, the range of~$\by$, contains the point~$\tilde \bx$. Also let $\tilde{\bs\nu}=\bs\nu(\tilde\bx)$. We shall consider the limit
\begin{equation}
    \lim_{\substack{\bx\to \tilde \bx\\ \bx\in \Gamma\setminus\gamma}}\tilde{\bs\nu}\cdot\gradg \cK_{\gamma}[\mu \, \chi_{\Gamma_V}] (\bx),
\end{equation}
  where $\chi_{\Gamma_V}$ is the characteristic function
  for $\Gamma_V$. The limit of the remaining part,
  \begin{equation}
    \lim_{\substack{\bx \to \tilde \bx\\ \bx\in \Gamma\setminus\gamma}}\tilde{\bs\nu}\cdot\gradg \cK_{\gamma}[\mu\lp 1-\chi_{\Gamma_V}\rp] (\bx),
\end{equation}
  may easily be computed by the dominated convergence theorem, since the kernel
  of $K_\gamma$ is smooth for $\bx\in \Gamma\setminus\Gamma_V$.

 Without loss of generality, we shall also
  suppose that~$\by$ maps~$(0,0)$ to~$\tilde\bx$ and maps the line
  segment~$[-L,L]\times\{0\}$ to~$\gamma\cap\Gamma_V$. Under this
  parameterization, the integral becomes
\begin{equation}
  \tilde{\bs\nu}\cdot\gradg \cK_{\gamma}[\mu \chi_{\Gamma_V}](\bx)
  = \frac{1}{2\pi}\int_{-L}^L \frac{\tilde{\bs\nu}\cdot (\bx - \by(s,0))}{\Vert\bx - \by(s,0)\Vert^2} \mu(\by(s,0)) \, \sqrt{g_{ss}} \, ds.
\end{equation}
We consider test points~$\bx$ that approach~$\by(0,0)$ along the
curve~$\by(ah,bh)$ and Taylor expand the kernel
of~$\tilde{\bs\nu}\cdot\gradg \cK_{\gamma}$ in the vicinity of~$(0,0)$:
\begin{multline}
   \frac{1}{2\pi}\frac{\tilde{\bs\nu}\cdot (\by(ah,bh) - \by(s,0))}{\Vert\by(ah,bh) - \by(s,0)\Vert^2}\sqrt{g_{ss}(s,t)} \\
   = \frac{1}{2\pi}\frac{\lp\tilde{\bs\nu}\cdot \nabla\by(0,0)\rp \cdot
     (ah-s,bh)+\cO(h^2+s^2)}{(ah-s,bh)^T \nabla\by(0,0)^T \nabla\by(0,0)(ah-s,bh)+\cO(h^3+s^3)}\lp \sqrt{g_{ss}(0,0)}+\cO(s) \rp.\label{eq:normal_kernel}
\end{multline}
We now note that
$\tilde{\bs\nu}=-g^{-1}_{ss}(0,0)^{-1/2} \nabla \by(0,0) g^{-1}(0,0)\hat e_t$, where~$\hat e_t\in\bbR^2$ is the standard basis vector in the~$t$ direction. We thus have that~$\tilde{\bs\nu}\cdot \nabla \by = -g^{-1}_{ss} \hat e_t$, so the kernel~\eqref{eq:normal_kernel} can thus
be written as
\begin{multline}
  \frac{1}{2\pi}\frac{\tilde{\bs\nu}\cdot
    (\by(ah,bh) - \by(s,0))}{\Vert\by(ah,bh) - \by(s,0)\Vert^2}\sqrt{g_{ss}(s,t)} \\
  = \frac{-1}{2\pi}\sqrt{\frac{g_{ss}(0,0)}{g^{-1}_{tt}(0,0)}}
  \frac{bh}{(ah-s,bh)^T g(0,0)(ah-s,bh)}+k_b(s,h),
\end{multline}
where~$k_b(s,h)=\cO(1)$ as~$s,h\to 0$ because~$g$ is positive definite.

To examine the limit as~$h\to 0$, we split the kernel into two parts: the
singular part,~$k_s(s,h)$ and the bounded part~$k_b(s,h)$. Expanding the
denominator of $k_s$ gives that
\begin{equation*}
    k_s(s,h) =\frac{-1}{2\pi}\sqrt{\frac{g_{ss}}{g^{-1}_{tt}}}\frac{bh}{(ah-s)^2g_{ss} + 2bh(ah-s)g_{st}+b^2h^2g_{tt}},
\end{equation*}
 where the metric $g$ is evaluated at $(0,0)$.
Substituting~$\tilde h = b\sqrt{g_{tt}/g_{ss}} h$, gives the slightly nicer expression
\begin{equation*}
    k_s(s,h) = \frac{-1}{2\pi}\frac{1}{\sqrt{g_{tt}\lp g^{-1}\rp_{ss}}}\frac{\tilde h}{\lp\frac{a}b\sqrt{\frac{g_{ss}}{g_{tt}}}\tilde h-s\rp^2 + 2\tilde h\lp\frac{a}b\sqrt{\frac{g_{ss}}{g_{tt}}} \tilde h-s\rp\frac{g_{st}}{\sqrt{g_{ss}g_{tt}}}+\tilde h^2}.
\end{equation*}
Since~$\tilde h>0$ we can analytically compute
\begin{equation*}
  \begin{aligned}
   \int_{-L}^L k_s(s,h)ds&=\int_{-\infty}^\infty k_s(s,h) \, ds -\int_{\bbR\setminus[-L,L]}k_s(s,h) \, ds\\
   &= \frac{-1}{2\pi}\frac{1}{\sqrt{g_{tt}\frac{g_{ss}}{g_{ss}g_{tt}-g_{st}^2}}} \frac{2\pi}{\sqrt{4-\frac{4g_{st}^2}{g_{ss}g_{tt}}}}-\int_{\bbR\setminus[-L,L]}k_s(s,h) \, ds\\
  &=-\frac{1}{2}-\int_{\bbR\setminus[-L,L]}k_s(s,h) \, ds.
  \end{aligned}
\end{equation*}
 Since 
 \begin{equation}
     k_s(s,h)\sim \frac{-1}{2\pi}\frac{1}{\sqrt{g_{tt}\lp g^{-1}\rp_{tt}}}\frac{\tilde h}{s^2}
 \end{equation} 
 for~$s$ larger than~$\tilde h$, the dominated convergence theorem gives that
\begin{equation*}
    \lim_{h\to 0^+}\int_{-L}^L k_s(s,h) \, ds=-\frac{1}2.
  \end{equation*}
We now compute the integral with a Lipschitz continuous density~$\mu$:
\begin{equation*}
\begin{aligned}
    \lim_{h\to 0^+}\int_{-L}^L k_s(s,h) \, \mu(\bx(s,0)) \, ds
    =&\lim_{h\to 0^+}\int_{-L}^L k_s(s,h)\, (\mu(\bx(s,0))-\mu(\tilde\bx)) \, ds \\
    &\qquad \qquad  + \lim_{h\to 0^+}\int_{-L}^L k_s(s,h)\, \mu(\tilde\bx) \, ds \\
    =& \lim_{h\to 0^+}\int_{-L}^L k_s(s,h)\, (\mu(\bx(s,0))-\mu(\tilde\bx)) \, ds -\frac{\mu(\tilde\bx)}2.
  \end{aligned}
\end{equation*}
    
Since~$\mu$ is Lipschitz continuous, we can apply the dominated convergence theorem provided $k_s(s,h)s$ is bounded. Clearly,~$k_s(s,h)s$ is a rational function with a second order zero
at~$(0,0)$ in the numerator and denominator (since~$g$ is positive definite). It
is thus a bounded function of~$s,h$. We may use the dominated convergence
theorem and the fact that~$k_s(s,0)=0$ to see
that
\begin{equation}
  \lim_{h\to 0^+}\int_{-L}^L k_s(s,h) \, \mu(\bx(s,0)) \, du=-\frac{\mu(\tilde\bx)}2.
\end{equation}

For the contribution of~$k_b(s,h),$ we note that, since~$k_b(s,h)=\cO(1)$
as~$h\to 0$, we may use the dominated convergence theorem to conclude that
\begin{equation}
  \begin{aligned}
   \lim_{\bs y\to\tilde\bx} \partial_{\tilde{\bs\nu}}\cK_{\gamma}[\mu \chi_{\Gamma_V}](\bx)
   &= -\frac{1}2 \mu(\by(0,0)) + \int_{-L}^L \frac{\tilde{\bs\nu}\cdot (\by(0,0) - \by(s,0))}{\Vert\by(0,0) - \by(s,0)\Vert^2} \mu(\by(s,0)) \, \sqrt{g_{ss}} \, ds\\
   &=-\frac{1}2 \mu(\tilde\bx) +\cK'_\gamma[\mu\chi_{\Gamma_V}](\tilde\bx) .
 \end{aligned}
\end{equation}

The above argument also applies to the case when $c\neq 0$ because the parametrix has the same singularity as in the Laplace-Beltrami case. Since introducing an~$a\neq 1$ divides the parametrix by that~$a$ and neither~$a$ nor~$\bs b$ change the singularity, we have~\eqref{eq:neum_lim}.

Finally, we note that since the kernel of~$\mathcal{K}'_\gamma$ is bounded, Proposition 3.10 in~\cite{Folland1995} gives that~$\mathcal{K}'_\gamma$ is bounded on~$L^\infty(\gamma)$. The density of Lipschitz continuous functions then shows that~\eqref{eq:neum_lim} holds for all continuous~$\mu$.
\end{proof}

Combining the above two lemmas gives the following theorem.
\begin{theorem}\label{thm:NeumannBCsrep}
If~$u$ is given by~\eqref{eq:BC_rep},~$\sigma\in L^\infty(\Gamma),$ and~$\mu\in C(\gamma)$, then
\begin{equation}
  \label{eq:neumann}
   \bs \nu\cdot\gradg u|_{\gamma} = -\frac 1{2a}\mu +  \cK'_{\Gamma}[\sigma] + \mathcal{K}'_{\gamma} [\mu],
\end{equation}
where
\begin{equation}
  \cK'_\Gamma [\sigma](\bx) := \int_\Gamma \bs \nu(\bx)\cdot\gradg K(\bx,\bx') \,
  \sigma(\bx'), \qquad \text{for all } \bx\in \gamma.
\end{equation}
\end{theorem}

In order to verify that our final integral equations are Fredholm second kind,
we now show that the integral operators in
equations~\eqref{eq:diff_repr} and~\eqref{eq:neumann} are compact in~$L^2$.
\begin{theorem}\label{thm:operators_compact}
The following operators are compact from $L^2$ to $L^2$ :
\begin{equation*}
      \cK_{\Gamma}, \quad \cK'_{\Gamma}, \quad \cR_{\Gamma}, \quad \cK_{\gamma}, \quad \cK'_{\gamma}, \quad \text{and}\quad \cR_{\gamma}.
\end{equation*}
\end{theorem}
\begin{proof}
  The operators~$\cK_{\Gamma}$,~$\cK_{\gamma}$, and~$\cR_{\gamma}$ are compact
  because they are Hilbert-Schmidt maps on their respective domains and ranges (see Proposition 0.45
  in~\cite{Folland1995}). The operator~$\cK'_{\gamma}$ is a compact map from~$L^2(\gamma)$ to itself by Proposition
  3.11 in~\cite{Folland1995}. The operator~$\cR_{\Gamma}$ is a compact from~$L^2(\Gamma)$ to itself by
  Theorem~\ref{thm:second_kind}.

  Finally, we note that $\cK'_\Gamma:=\bs \nu\cdot\operatorname{trace}\gradg \cK_\Gamma$
  with~$\cK':L^2(\Gamma)\to H^{\frac12}(\Gamma)$ because
  \begin{equation*}
    \begin{aligned}
      \cK_\Gamma&:L^2(\Gamma)\to H^2(\Gamma),\\
      \gradg&:H^{2}(\Gamma)\to \lp \Ho\rp^3,\\
      \bs \nu\cdot\operatorname{trace}&:\lp H^1(\Gamma)\rp^3\to H^{\frac12}(\gamma).
      \end{aligned}
  \end{equation*}
  The compact embedding of~$H^{\frac12}(\gamma)$
  into~$L^2(\gamma)$ then gives that $\mathcal{K}'_\Gamma$ is a compact from~$L^2(\Gamma)$ to~$L^2(\gamma)$.
\end{proof}

Assembling Lemma~\ref{lem:surf_term} and Theorems~\ref{thm:NeumannBCsrep}
and~\ref{thm:operators_compact} gives the following theorem.
\begin{theorem}If 
$\sigma\in L^\infty(\Gamma)$,~$\mu\in C(\gamma)$, and they solve the Fredholm second kind system of equations
\begin{equation}
  \label{eq:NeuBVPsystem}
\begin{aligned}
    \sigma+\cR_{\Gamma}[\sigma] +\cR_{\gamma} [\mu] &= f,\\
    -\frac 1{2a}\mu +  \cK'_{\Gamma}[\sigma] + \mathcal{K}'_{\gamma} [\mu] &= f_\gamma
\end{aligned}
\end{equation}
then
\begin{equation}
    u = \cK_{\Gamma}[\sigma] + \cK_{\gamma} [\mu],
\end{equation}
is the solution of~\eqref{eq:surf_PDE} with boundary
condition~\eqref{eq:Neumann_BCs}.
\end{theorem}
We conclude this section by noting that the kernel of $\mathcal{K}'_\gamma$ is
smooth, which will allow us to efficiently discretize it using the trapezoid
rule.
\begin{theorem}\label{thm:kern_smth}
    The kernel of $\mathcal{K}'_\gamma$ is smooth.
\end{theorem}
The proof of this fact is identical to the proof that the adjoint of the
Laplacian double-layer potential operator in the plane is smooth when the target
point is on the boundary curve, see Section 11.1
of~\cite{HsiaoG.C.GeorgeC.2021Bie}.

\subsection{Dirichlet boundary conditions}

We now wish to consider the case where Dirichlet boundary conditions
are applied on $\gamma$:
\begin{equation}
    \left.u\right|_{\gamma} = f_\gamma \label{eq:Dirichlet_Bc},
\end{equation}
where $f_\gamma$ is given data. 
In order to enforce the Dirichlet boundary conditions, we shall use
the representation
\begin{equation}
  \label{eq:DirichletRep}
    u = \cK_\Gamma [\sigma]  +  \cN_\gamma [\mu],
\end{equation}
where $\cK_\Gamma$ is the same as in the previous section and~$\cN_\gamma$
is the edge operator
\begin{equation}
  \cN_\gamma[\mu](\bx) :=\int_\gamma \bs \nu(\bx') \cdot \gradg' K(\bx,\bx')
  \, \mu(\bx') \, ds(\bx'),
  \qquad \text{for all } \bx \in\Gamma.    
\end{equation}
The kernel of~$\cN_\gamma$ is smooth because it has the same behaviour when~$\bx\approx \bx'$ as the kernel of~$\cK'_\gamma$.

The following lemma gives how the differential operator acts on this
representation:
\begin{lemma}
  \label{lem:Dirichlet_surf_term}
  If $u$ is given by~\eqref{eq:DirichletRep},~$\sigma \in L^2(\Gamma)$, $\Vert \bs b \Vert < \infty$,
  and~$\mu\in L^2(\gamma)$, then
  \begin{equation}
    \DO u +\bs b\cdot\gradg u+ c\, u = \sigma+ \cR_{\Gamma}[\sigma] +
   \cT_\gamma[\mu],
  \end{equation}
  in the interior of~$\Gamma$, where
  \begin{equation}
    \cT_\gamma[\mu](\bx):= \int_\gamma \bs\nu(\bx') \cdot \gradg' R(\bx,\bx')
    \, \mu(\bx') \, ds(\bx'), \qquad \text{for all } \bx\in\Gamma.
    \end{equation}
\end{lemma}
\begin{proof}
    The proof is the same as the proof of Lemma~\ref{lem:surf_term}.
\end{proof}
In order to enforce the boundary conditions, we need the following
theorem, whose proof is very similar to the proof of
Theorem~\ref{thm:NeumannBCsrep}.
\begin{theorem}\label{thm:Dirhclet_enforce}
  If $u$ is given by~\eqref{eq:DirichletRep},~$\sigma\in L^2(\Gamma)$,
  and~$\mu\in C(\gamma)$, then along the boundary~$\gamma$ we have that:
  \begin{equation}
    u|_\gamma = -\frac1{2a} \mu + \cK_\Gamma [\sigma] + \mathcal{N}_\gamma [\mu],
  \end{equation}
  where for $\bx \in \gamma$
  \begin{equation}
    \begin{aligned}
      \cK_\Gamma [\sigma](\bx) :=& \int_\Gamma K(\bx,\bx') \, \sigma(\bx') \, da(\bx'),\\
      \cN_\gamma[\mu](\bx) :=& \int_\gamma \bs\nu(\bx')\cdot\gradg' K(\bx,\bx')\, \mu(\bx') \, ds(\bx').
    \end{aligned}
  \end{equation}
\end{theorem}

As before, we consider the compactness of our new integral operators.
\begin{theorem}\label{thm:diriclet_compactness}
  The operators $\mathcal{K}_\Gamma$ and $\mathcal{N}_\gamma$ are
  compact from~$L^2$ to~$L^2$.  The operator $\cT_\gamma$ is a compact map
  from $L^2(\gamma) \to L^2(\Gamma)$ if $a$ is constant,
  $\bs b= \bs 0$, and~$c$ is either identically zero or bounded away from zero.
\end{theorem}
\begin{proof}
  The operators $\mathcal{K}_\Gamma$ and $\mathcal{N}_\gamma$ may be
  seen to be compact through Proposition 3.11 in~\cite{Folland1995}.

  We now prove that~$\cT_\gamma$ is compact in the Laplace-Beltrami case by proving that its adjoint,
  $\bs \nu \cdot \operatorname{trace}\gradg (\mathcal R_\Gamma)^*$, is
  a compact map from~$L^2(\Gamma)$ to~$L^2(\gamma)$. Since~$\bs \nu$ is smooth, it is enough to show that the
  operator~$\operatorname{trace}\bs e_j\cdot\gradg
  (\mathcal{R}_\Gamma)^*$ is compact for each of the standard three dimensional basis
  vectors~$\bs e_j$.

  The kernel of the operator~$\operatorname{trace}\bs e_j\cdot\gradg
  (\mathcal{R}_\Gamma)^*$ is
  \begin{equation}
    \begin{aligned}
      k(\bx,\bx') :=& \bs  e_j \cdot \gradg' \RLB(\bs x,\bx')\\
      =&\bs  e_j \cdot \gradg'\frac{1}{\pi }\left(\frac{ \bs n(\bx)\cdot (\dx)}{\dxnorm}
         \right)^2 - \frac{H(\bx)}{\pi }
         \bs  e_j \cdot \gradg'\frac{\bs n(\bx)\cdot(\dx)}{\dxnorm}. \label{eq:kern_adjop}
    \end{aligned}
  \end{equation}
  The second term in the expression can be expanded as
  \begin{equation}
    \label{eq:dirichlet_kernel_expansion}
    \bs  e_j\cdot \gradg' \frac{\bs n(\bx)\cdot(\dx)}{\dxnorm}
    =-\frac{\bs  e_j\cdot[\bs n(\bx)-(\bs n(\bx')\cdot \bs n(\bx))\bs n(\bx')] }{\dxnorm}+2   \frac{\bs n(\bx)\cdot(\dx)
      \bs  e_j\cdot(\dx)}{r^4},
  \end{equation}
  which has an~$\cO \lp r^{-1}\rp$ singularity as~$\bx\to \bx'$. The first term
  in~\eqref{eq:kern_adjop} has the same singularity by the chain rule. We therefore have that
  \begin{equation}
      k(\bx,\bx')=\frac{q(\bx,\bx')}{r},
  \end{equation}
  where~$q(\bx,\bx')$ is a bounded function on~$\Gamma\times\Gamma\setminus \{\bx=\bx'\}$. Lemma 3.4 and Theorem 4.1 from~\cite{punchin1988weakly} therefore give that~$\bs e_j\cdot\gradg
  (\mathcal{R}_\Gamma)^*$ is a
bounded map from~$L^2(\Gamma)$ to~$H^{1-\epsilon}(\Gamma)$.

Composing with the trace map sends the result from~$H^{1-\epsilon}(\Gamma)$
to~$H^{\frac12-\epsilon}(\gamma)$. Since~$H^{\frac12-\epsilon}(\gamma)$ is compactly
embedded in~$L^2(\gamma)$, we have
that~$\operatorname{trace}\bs e_j\cdot \gradg (R_\Gamma)^*$ is a
compact map from~$L^2(\Gamma)$ to~$L^2(\gamma)$. Multiplying by the
smooth function~$\bs \nu\cdot e_j$ and summing over~$j$ preserves compactness, and therefore we
have that~$\cT_\gamma^*$ is compact from~$L^2(\Gamma)$ to~$L^2(\gamma)$. We thus have that~$\cT_\gamma$ is a compact map from~$L^2(\gamma)$ to~$L^2(\Gamma)$.

Moving to the Helmholtz-Beltrami case ($c\neq 0$), does not change the singularity of the
  kernels involved and so will the above argument can be used to show that~$\cT_\gamma$ is still compact. Finally, setting~$a$ to be any constant only divides the parametrix and it's argument by that constant, so does not change the compactness.
\end{proof}

Assembling the above lemmas gives the following theorem.
\begin{theorem}If 
$\sigma\in L^\infty(\Gamma)$,~$\mu\in C(\gamma)$, and they solve the Fredholm second kind system of equations
\begin{equation}
  \label{eq:DirBVPsystem}
\begin{aligned}
     \sigma+ \cR_{\Gamma}[\sigma] +\cT_\gamma[\mu] &= f,\\
    -\frac 1{2a}\mu +  \cN_{\Gamma}[\sigma] + \mathcal{K}'_{\gamma} [\mu] &= f_\gamma
\end{aligned}
\end{equation}
then
\begin{equation}
    u = \cK_{\Gamma}[\sigma] + \cN_{\gamma} [\mu],
\end{equation}
is the solution of~\eqref{eq:surf_PDE} with boundary
condition~\eqref{eq:Dirichlet_Bc}.
\end{theorem}

\begin{remark}
  If~$\gamma$ can be split into two disjoint edges,~$\gamma_N$
  and~$\gamma_D$, the we can solve~\eqref{eq:surf_PDE} with Neumann
  boundary conditions applied on~$\gamma_N$ and Dirichlet boundary
  conditions applied on~$\gamma_D$. This could be done by representing
  the solution as
\begin{equation}
    u=\cK_{\Gamma} [\sigma] +\cK_{\gamma_N}[\mu] + \cN_{\Gamma_D}[\rho]
\end{equation}
and applying the above identities to derive a system of coupled
integral equations that enforce~\eqref{eq:surf_PDE} and the boundary
conditions.
\end{remark}

\section{A numerical solver}
\label{sec:Numerics}

In this section, we present a numerical method for
discretizing~\eqref{eq:IE} and using it to find the solution
of~\eqref{eq:surf_PDE} on any smooth surface without boundary. In short, equation~\eqref{eq:IE} is
discretized using a Nystr\"om-like method based on the one described
in~\cite{bremer2012weakly}, resulting in a finite dimensional linear
system whose solution is a vector of approximate point-values of the
function~$\sigma$. The linear system to be solved is dense, and we use
a modified version of the \emph{Fast Linear Algebra in MATLAB}
(\emph{FLAM}) software library~\cite{Ho2020} to accelerate the its
solution via fast matrix-vector applications coupled with GMRES~\cite{saad1986gmres}.
In the following subsections, we provide some additional
detail regarding the quadratures and matrix compression techniques
used in solving the linear system. Finally, we end with a discussion on how to discretize~\eqref{eq:NeuBVPsystem} and~\eqref{eq:DirBVPsystem} and use them to find the solution of~\eqref{eq:surf_PDE} on surfaces with boundaries.

\subsection{Nystr\"om discretization}
\label{sec:Nystrom}

In order to discretize integral equation~\eqref{eq:IE} along a
surface~$\Gamma$, we first need a description of the surface
itself. The surface~$\Gamma$ is assumed to be specified by a
collection of~$M$ non-overlapping curvilinear triangles~$\Gamma_i$ so
that~$\Gamma = \cup_i \Gamma_i$. Each~$\Gamma_{i}$ is parameterized by
a function~$\by_{i}$ which maps the standard simplex triangle~$T_{0}$
to~$\Gamma_{i} \subset \bbR^{3}$, where~$T_0$ is given by
\begin{equation}
  T_{0} = \{(u,v)\in\bbR^2\quad |\quad u,v\geq 0\quad\text{and}\quad u\leq 1-v \}.
\end{equation}
The integral equation along~$\Gamma$ is then divided  into several
pieces, each over~$T_{0}$:
\begin{equation}
  \label{eq:IEsum}
  \sigma(\bx) + \sum_{i = 1}^M \int_{T_0} R\lp\bx,\by_i(u,v) \rp \,
  \sigma(\by_i(u,v)) \, da(\by_i(u,v)) = f(\bx), \qquad \text{for } \bx \in \Gamma.
\end{equation}

In a pure Nystr\"om discretization, the above integral equation is transformed
into a finite dimensional linear system via sampling the integral at nodes on
each triangle, which we will denote by~$\bx_{{ij}}$,~$j = 1,\ldots,n_{p}$, and
associating with each of these nodes a quadrature weight~$w_{ij}$. The linear
system is then enforced at each of the discretization nodes, resulting in the
set of equations:
\begin{equation}
  \sigma(\bx_{ij}) + \sum_{k,l} w_{kl} \, R\lp\bx_{ij},\bx_{kl} \rp \,
  \sigma(\bx_{kl}) \, \sqrt{\det g(\bx_{kl})} = f(\bx_{ij}),
\end{equation}
for~$i = 1,\ldots,M$ and~$j=1,\ldots,n_{p}$, and where~$\bx_{kl}=\bs y_k(u_l,v_l)$ is the image of the Vioreanu-Rokhlin~\cite{B.Vioreanu2014} quadrature
node~$(u_{l},v_{l})$ on the reference simplex~$T_{0}$. The order of the quadrature
is denoted by~$p$, and the resulting number of nodes on each triangle is given
by~$n_{p} = p(p+1)/2$. See~\cite{bremer2012weakly,Greengard2021} for a thorough
discussion of this type of discretization and quadrature rule. The overall size
of the linear system is then given by~$N = M n_{p}$. We will abbreviate our
approximation to the solution~$\sigma$ at the discretization nodes
by~$\sigma_{ij} \approx \sigma(\bx_{ij})$.

In our modified Nystr\"om method different quadrature rules are used depending on the distance from the
image of a particular map~$\bx_{k}$ to a particular \emph{target}
node~$\bx_{ij}$.  These rules can be split into three categories, as
described in~\cite{Greengard2021}: (1) singular quadature, (2)
nearly singular quadrature, and (3) smooth quadrature. While not all
of the kernels discussed in this paper are actually singular at the
origin, their behavior near the origin differs sufficiently from their
far field behavior that different quadature rules should be used to
obtain suitable accuracy and speed. We now discuss the quadrature
schemes used for these three categories, as well as for \emph{edge
  quadrature} for surfaces with boundaries, for the kernels of
interest in this work.

\subsubsection{Singular quadrature}\label{sec:sing_quad}

For curvilinear patches of integration in the above integral equation that
contain the target point, it is necessary to use a quadrature rule that can
integrate irregular or singular functions to high-order accuracy. (We refer to
this as singular quadrature, even though in some cases the kernel is actually
bounded.) In order to design this quadrature rule, we make the following
observation about the remainder kernel.

\begin{theorem}
  Let~$\bs y = \by(r,\theta)$ be a local parameterization of~$\Gamma$
  containing the point~$\bx$, and let the coefficients $a$,~$\bs b$ in
  the underlying surface PDE be smooth. If~$(r,\theta)$ are polar
  coordinates centered at the pre-image of~$\bx$, then we have the
  following asymptotic formula as $r\to 0$:
  \begin{multline}
    \label{eq:asymptotic_form}
    R(\bx, \bs y(r,\theta)) =
    \frac{\lp\gradg a+\bs b\rp \cdot \gradg \bs y|_{(0,0)} \, \hat{\bs\theta}}{r}\\
    -\frac{H^2(\bx)\sin^2(2(\theta-\theta_\star))+\kappa_G(\bx)\cos^2(2(\theta-\theta_\star))}
    {4\pi }+ \cO(r)
\end{multline}
where~$\hat {\bs \theta}$ is the unit vector in the~$\theta$-direction in
the~$uv$-plane,~$\theta_\star$ is the angle corresponding to first the principal
direction, and~$\kappa_{G}$ is the Gaussian curvature.
\end{theorem}
\begin{proof}
  The proof is the same as in the derivation of~\eqref{eq:limiting_behaviour},
  except that we use the curve~$\gamma$ defined by~$\bs y(\cdot,\theta)$ and the
  asymptotic formulas for the derivatives of~$G$ appearing
  in~\eqref{eq:explicit_remainder}.
\end{proof}

Based on this observation, it is possible to use polar coordinates in
the~$uv$-plane to integrate the remainder kernel over the triangle containing
the target point. In these coordinates, the contribution to the integral becomes
\begin{equation}
  \label{eq:radial_int}
  \int_{\Gamma_i} R(\bx, \bx') \, \sigma(\bx') \, da(\bx') =
  \int_0^{2\pi} \int_0^{L(\theta)} R(\by(0,0), \by(r,\theta)) \, \sigma(\by(r,\theta)) \,
  \sqrt{\det g(r,\theta)} \, r \, dr \, d\theta,
\end{equation}
where~$L(\theta)$ is the distance from the pre-image of the target point to the
edge of~$T_{0}$ along the ray at angle~$\theta$ to the~$x$-axis. Due to the
asymptotic formula~\eqref{eq:asymptotic_form}, this change of coordinates has
removed the singularity from the integrand. We may therefore uses a smooth
quadrature rule for this integral.

To this end, we discretize the inner and outer integrals
in~\eqref{eq:radial_int} using Gauss-Legendre quadrature. The panels
for the inner integral are chosen adaptively. Additional care must be
taken for the outer integral in~$\theta$ since the upper bound of the
inner integral, $L(\theta)$, is only piecewise smooth and becomes
nearly singular as the center of the polar coordinates approaches the
edge of the triangle, see Figure~\ref{fig:L_theta}.  We therefore
adaptively choose our panels to resolve this near singularity.
Specifically, we choose panels such that~$L^2(\theta)$ is integrated
to a pre-specified tolerance; this choice corresponds to integrating a
bounded function of~$r$ to within the same
tolerance. Since~$L(\theta)$ is bounded by~$\sqrt{2}$, this refinement
will also be sufficient to integrate the logarithmic singularity in
the parametrix. An example of the quadrature nodes used is shown in
Figure~\ref{fig:local_coordinates}.

\begin{remark}
  It is unfortunately the case that the quadrature nodes used to
  accurately compute the integrals above do not coincide with the
  Vioreanu-Rokhlin nodes used in the discretization of the integral
  equation. In this case, it is required that we implicitly
  interpolate the function~$\sigma$ to the quadrature grid using
  interpolation matrices described in~\cite{Greengard2021}. Composing
  these interpolation operators with the quadrature described above
  yields the required matrix entries in the discretized system. An analogous procedure for integrals which are nearly singular is discussed in the following section, Section~\ref{sec:near_sing}.
\end{remark}

\afterpage{
  \clearpage

\begin{figure}[t]
    \centering
    \includegraphics[width=0.9\textwidth]{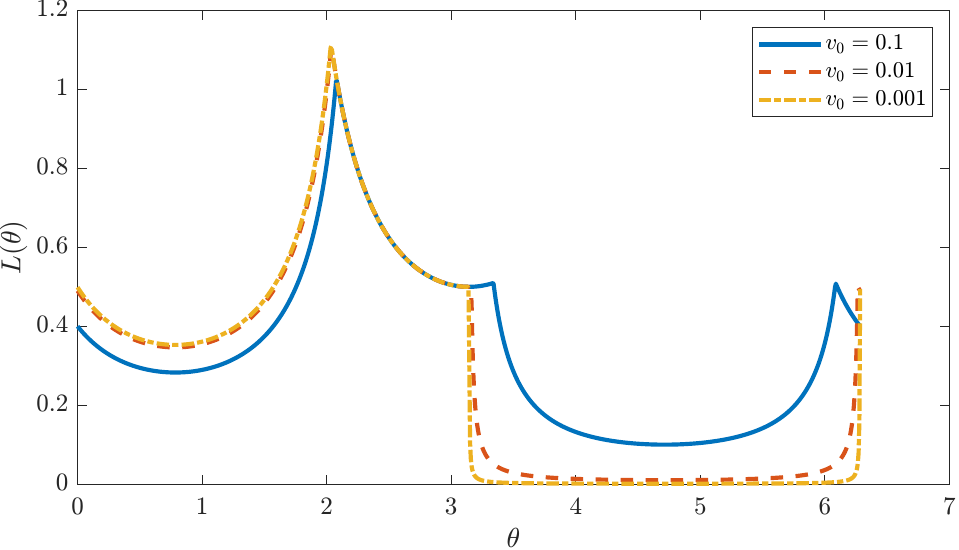}
    \caption[The nearly singular nature of~$L(\theta)$]{This figure shows an
      example of~$L(\theta)$, where the polar coordinates are centered
      at~$(0.5,v_0)$ for various choices of~$v_0$. For any center,~$L(\theta)$
      is a piecewise-smooth function, but as the center approaches the edge, it
      becomes nearly singular.}
    \label{fig:L_theta}
\end{figure}

\begin{figure}[b]
    \centering
    \includegraphics[width=0.9\textwidth]{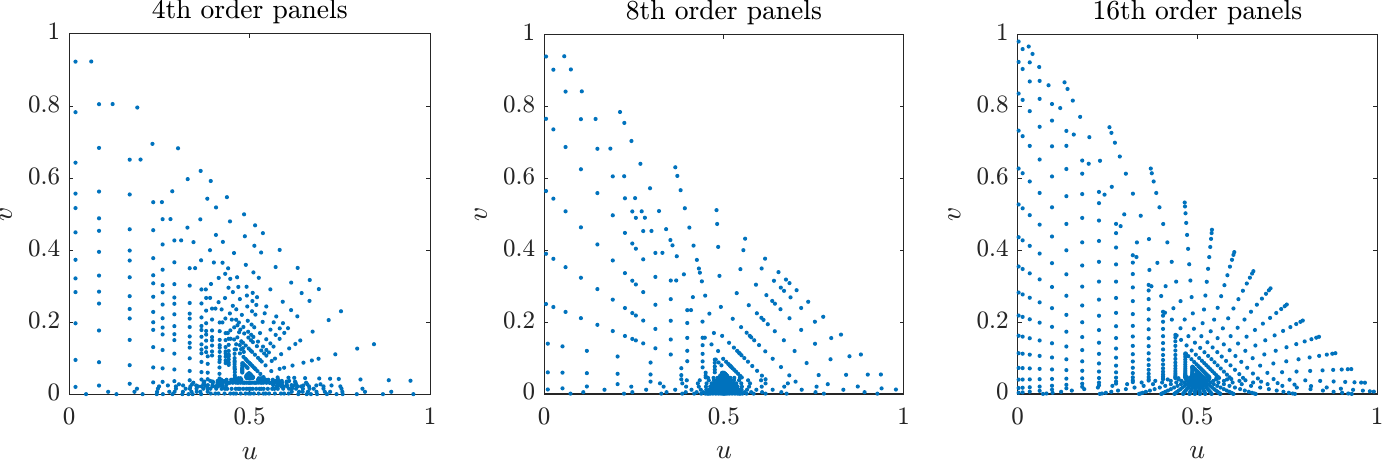}
    \caption[Radial grid for self quadrature]{The polar tensor product grid used
      to compute the self contribution when the pre-image of the target
      is~$(0.5,0.05)$. For illustrative purposes, this grid is generated using
      4th, 8th, and 16th order Gauss-Legendre panels. Our numerical examples
      will only use 16th order panels.}
    \label{fig:local_coordinates}
\end{figure}

\clearpage}

We conclude this section by testing the previously described singular quadrature
routine. We test it on three different functions of the
form~$F(r,\theta) \, \sigma(x,y)$, where~$(r,\theta)$ are polar coordinates
centered at the target point~$\bx$. The function~$F(r,\theta)$ is set to either $\cos(2\theta)$ or $\log r$, which have the same singularity as~$R$ and~$K$ respectively. The function~$\sigma$ was set to~$1$,~$\exp(x+y)$, or~$\cos(\pi x)$. For this test, we use 16th-order
Gauss-Legendre quadrature panels and check the convergence of the method with the
integration tolerance by comparing it to the value given by the same routine
with a finer tolerance. To fully test the method, we set~$\bx$ to each of the
target points that will be used in the global solver, i.e. the 16th-order
Vioreanu-Rohklin nodes. The results in~Figure~\ref{fig:radial_inttest} show that
the integration routine converges at the expected rate, independent of the test
function.

\subsubsection{Nearly singular quadrature}\label{sec:near_sing}

While changing to polar coordinates allows for the accurate computation of
integrals over a triangle containing the target point, it cannot be used for
adjacent triangles, as, in general, the pre-image of the target point
in~$uv$-plane is not known. For these nearly singular interactions, the adaptive
integration method introduced in~\cite{bremer2012weakly} and further developed
in~\cite{Greengard2021} can be used to compute the integrals to a specified
accuracy. Before continuing, we define what we mean by a triangle \emph{near}
the target point. We first define the centroid of each triangle as the average
position of its vertices. The radius of the triangle is then the largest
distance from the centroid to a vertex. We shall then say that a target is near
to a triangle if it is within~$1.5$ radii of the centroid of that triangle.

\begin{figure}[t]
    \centering
    \begin{subfigure}{0.45\linewidth} \centering{\includegraphics[width=.95\linewidth]
        {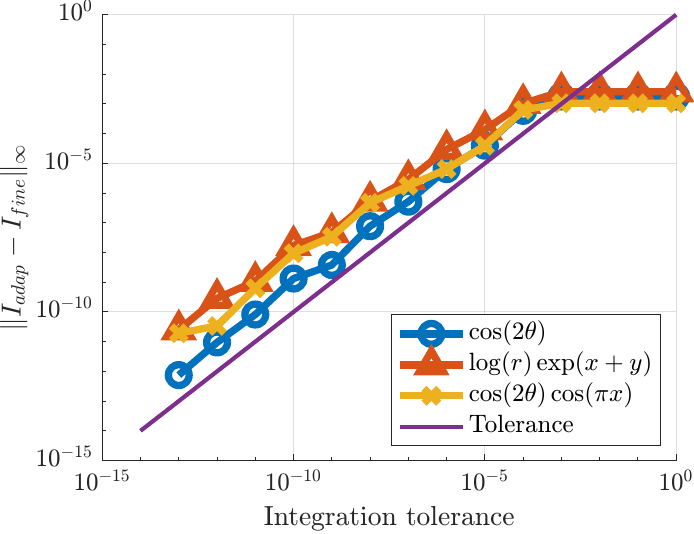}}
      \caption{The error in the singular quadrature routine when compared to the
        answer obtained using a finer tolerance. The solid line represents the
        error scaling with the integration tolerance.}
     \label{fig:radial_inttest}
   \end{subfigure}
   \quad
       \begin{subfigure}{0.45\linewidth}
     \centering{\includegraphics[width=.95\linewidth]{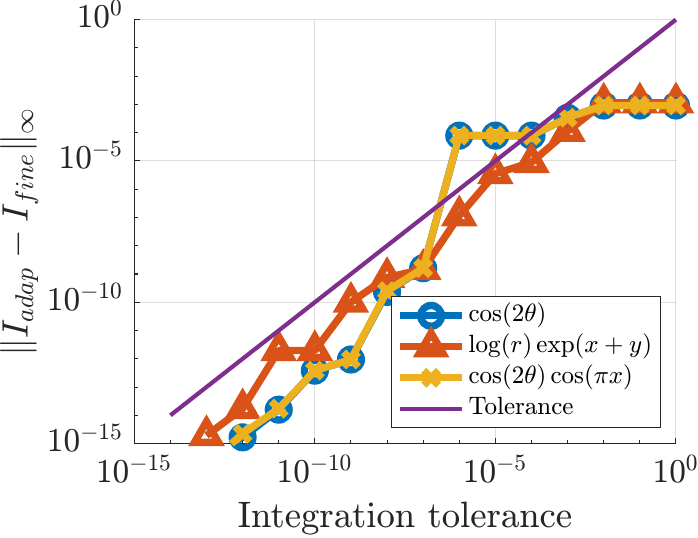}}
     \caption{The error in the adaptive integrator when compared to the answer
       obtained using a finer tolerance. The solid line represents the error
       scaling with the integration tolerance.}\label{fig:adap_inttest}
   \end{subfigure}
   \caption[The convergence of our singular and nearly singular quadrature
   methods.]{The results of a convergence test for singular quadrature and
     adaptive integration routines applied to functions of the form $F(r,\theta)\sigma(x,y)$ for different choices of~$F$ and~$\sigma$. Both tests involve repeating the experiment
     for a number of different target locations and reporting the maximum error
     observed over all test points. The figures demonstrates that both routines
     can be used to integrate functions with the (near-) singularities present in
     $\mathcal{R}$ and $\mathcal{K}$ to the desired tolerance.}
    \label{fig:inttest}
\end{figure}

To compute the integral over triangles in the near field, we implicitly construct the
polynomial interpolant~$\tilde \sigma$ that agrees with the
density~$\sigma(\by_i(u,v))$ at the Vioreanu-Rohklin nodes. We shall write the
interpolant in the basis of Koornwinder polynomials, a family of othorgonal
polynomials on the simplex:
\begin{equation}
    \tilde \sigma (u,v) = \sum_{n=1}^{n_p} c_{n} K_{n}(u,v).
\end{equation}
One feature of the Vioreanu-Rohklin nodes is that the above polymonial
interpolation is very well conditioned in this basis~\cite{B.Vioreanu2014}. To
find the coefficients, we define the matrix~$\mtx{V}$ with coefficients
\begin{equation}
    \elem{V}_{nj} = K_{n}(u_j,v_j).
\end{equation}
The coefficients~$c_{n}$ may be then found via the formula:
\begin{equation}
    c_{n} = \sum_{j=1}^{n_p} \elem{U}_{nj} \sigma(\bx_i(u_j,v_j)),
\end{equation}
where~$\elem{U}_{nj}$ is the~$nj$-entry of the matrix~$\mtx{U} = \mtx{V}^{-1}$.
We can now compute the integral over the neighboring triangle as
\begin{equation}
  \begin{aligned}
    \int_{\Gamma_i} R(\bx, \bx') \, \tilde \sigma(\bx') \, da(\bx')
    &= \sum_n c_n \int_T  R(\bx, \by(u,v)) \, K_n(u,v) \sqrt{\det g(u,v)} du \, dv \\
    &=\sum_{n,j} \sigma(\by_i(u_j,v_j))  \, \elem{U}_{nj} \int_T  R(\bx, \by(u,v)) \,
    K_n(u,v) \, \sqrt{\det g(u,v)} \, du \, dv .
  \end{aligned}
\end{equation}
The remaining integrals are independent of~$\sigma$ and  may be
precomputed to accelerate quadrature. In practice, they are computed
by adaptively splitting~$T$ and using Vioreanu-Rohklin quadrature on
each subtriangle~\cite{Greengard2021}.

We finish this section by testing our nearly singular quadrature
routine. We test it on three different functions of the
form~$F(r,\theta) \sigma(x,y)$, where~$(r,\theta)$ are polar
coordinates centered at the target point~$\bx$ outside the
simplex and~$F$ and~$\sigma$ are chosen to be the same as in Section~\ref{sec:sing_quad}. For this test, we compare it to the value given by our
routine using an oversample quadrature rule. To fully test the method,
we set~$\bx$ to each of the 16th-order Vioreanu-Rohklin nodes in a
neighbouring triangle. The results in~Figure~\ref{fig:adap_inttest}
show that the error in our adaptive integration routine scales with
our integration tolerance, independent of the test function.

\subsubsection{Smooth quadrature}

For integral contributions in~\eqref{eq:IEsum} for which the
target~$\bx$ is sufficiently far from the domain of integration,
i.e. not on~$\Gamma_i$ and not in the near field of~$\Gamma_i$, the
integrand can be treated as a smooth function and a quadrature rule
for smooth functions can be applied. In particular, these portions of
the integral equation are discretized using the Vioreanu-Rokhlin
quadrature rules for smooth functions on
triangles~\cite{B.Vioreanu2014, bremer2012weakly}. The resulting
contributions to the finite dimensional linear system are purely
Nystr\"om in style:
\begin{equation}
  \begin{aligned}
    \int_{\Gamma_i} R(\bx,\by) \, \sigma(\by) \, da(\by)
    &= \int_{T_0} R\lp\bx,\by_i(u,v) \rp \,
      \sigma(\by_i(u,v)) \, da(\by_i(u,v)) \\
    &\approx \sum_{k,l} w_{kl} \, R\lp\bx_{ij},\bx_{kl} \rp \,
      \sigma(\bx_{kl}) \, \sqrt{\det g(\bx_{kl})}.
  \end{aligned}
\end{equation}
The quadrature nodes and weights are obtained from a precomputed lookup table
based on the algorithm in~\cite{B.Vioreanu2014}.

\subsubsection{Edge quadrature}
\label{sec:edgequad}

For problems with boundaries, it is also necessary to construct a
discretization of integrals along one-dimensional curves
in~$\bbR^3$. We shall suppose that each boundary (curve) has a smooth
and periodic parameterization, $\gamma: [0,2\pi] \to \bbR^3$ with
$\gamma(0) = \gamma(2\pi)$. The curve~$\gamma = \gamma(t)$ is then
sampled at equispaced points in~$t$.  Since we only consider
boundaries that are smooth and closed curves, this allows for the use
of spectrally accurate trapezoidal quadrature rules to compute smooth
integrals.

Computing the boundary-to-surface integral operators requires
computing nearly singular integrals. For these integrals, the
boundaries and densities supported on them are over-sampled by some
fixed factor using Fourier interpolation; trapezoidal quadrature is
then used on this over-sampled discretization.  While this method is
not spatially adaptive, it can still be used to solve a wide range of
problems to high-order accuracy and is, more importantly, \emph{not} a
dominant cost of the discretization due to the one-dimensional nature
of the boundary curve.

The performance of this discretization and quadrature scheme is
discussed in Section~\ref{sec:bvps}.

\subsection{Fast matrix vector multiplication}

Coupling a fast matrix vector application with iterative Krylov methods, such as
GMRES, generally results in accelerated solvers when the number of iterations
required for convergence remains bounded and~$\cO(1)$. In order to efficiently
apply the dense integral operators appearing in the parametrix formulation of
this paper, we use a modification of the \emph{interpolative fast multipole
  method} (IFMM)~\cite{martinsson2007accelerated}. As with all fast multipole
methods (FMMs), this method is based on the observation that if a block of the
kernel matrix corresponds to an interaction between well-separated sources and
targets then it will be of low numerical rank due to the smoothness of the
kernel function. The method then exploits this observation by organizing the
discretization points~$\bx_i$ into an octree and compressing the interaction of
points in the blocks corresponding to the points in different boxes at the same
level in the octree.

In standard fast multipole method fashion, the computational domain
(in this case, discretization points on the surface~$\Gamma$) is
hierarchically partitioned via an octree.  Denoting the boxes in the
octree as~$B_k$ and indices of the nodes/points in box~$B_k$ by~$\elem{I}_k$,
then the IFMM approximates the interaction between~$B_k$ and~$B_l$ by
\begin{equation}
  \label{eq:ifmmskel}
  \sum_{j\in \elem{I}_k} K(\bx,\bx_j) \, a_j
  \approx \sum_{j\in \elem{J}_{kl}\subset \elem{I}_k} K(\bx,\bx_j) \, b_{jl}, \qquad \text{for all }
  \bx\in B_l,
\end{equation}
where $\elem{J}_{kl}$ is a specially chosen subset of~$\elem{I}_k$ and the new
coefficients/weights~$b_{kl}$ are computed based on the~$a_j$'s and the choice
of~$\elem{J}_{kl}$. This approximation is computed using an interpolative
decomposition. We recall that an exact interpolative decomposition
of a matrix is given by the following definition:
\begin{definition}
  Let~$\mtx{A}$ be an $m\times n$ matrix with rank~$k$. An interpolative
  decomposition (ID) of $\mtx{A}$ is a factorization of the form
\begin{equation}
    \mtx{A} = \mtx{A}(:,\elem{J}) \, \mtx{X},
\end{equation}
where~$\mtx{X}$ is a~$k\times n$ matrix and~$\elem{J}$ is a set
of~$k$ indices between~$1$ and~$n$.
\end{definition}
In practice, interpolative decompositions are computed so that the above
factorization is accurate (in a relative sense) to some user-specified
precision~$\epsilon >0$. These $\epsilon$-accurate IDs can be computed
numerically, for example, using a rank-revealing column-pivoted QR
decomposition, and described in~\cite{martinsson2019fast}. The ID is closely
related to the idea of skeletonization, and~$\elem{J}$ in the definition above
is referred to as the skeleton of the matrix~$\mtx{A}$.

The required skeletons~$\elem{J}_{kl}$ in~\eqref{eq:ifmmskel} can be computed by
finding an approximate ID of the
matrix~$$\mtx{K}^{lk} = \left[ K(\bx_i,\bx_j) \right]$$ for $i\in\elem{I}_l$
and~$j\in\elem{I}_k$. In practice however, this cost is prohibitively expensive
because the sizes of the largest blocks scale with the number of unknowns.
Instead, most implementations make use of a so-called \emph{proxy
  surface}~\cite{ying2004kernel,martinsson2005fast}. Typically, this involves
prescribing a number of test points~$\tilde{\bx}_i$ which lie on a
sphere~$P_{kl}$ containing~$B_k$ but not~$B_l$. The skeleton is then found by
computing the ID of the much smaller
matrix~$\left[ K(\tilde{\bx}_i,\bx_j)\right]$ with~$i\in\elem{P}_{kl}$
and~$j\in\elem{I}_k$. This approach is justified by using Green's identities for
elliptic PDEs in~$\bbR^3$~\cite{Greengard2021} and is therefore only applicable
if the kernel is some combination of derivatives of a three-dimension Green's
function.

\begin{figure}[!t]
    \centering
    \begin{subfigure}{0.45\linewidth}
     \centering{\includegraphics[width=.95\linewidth]{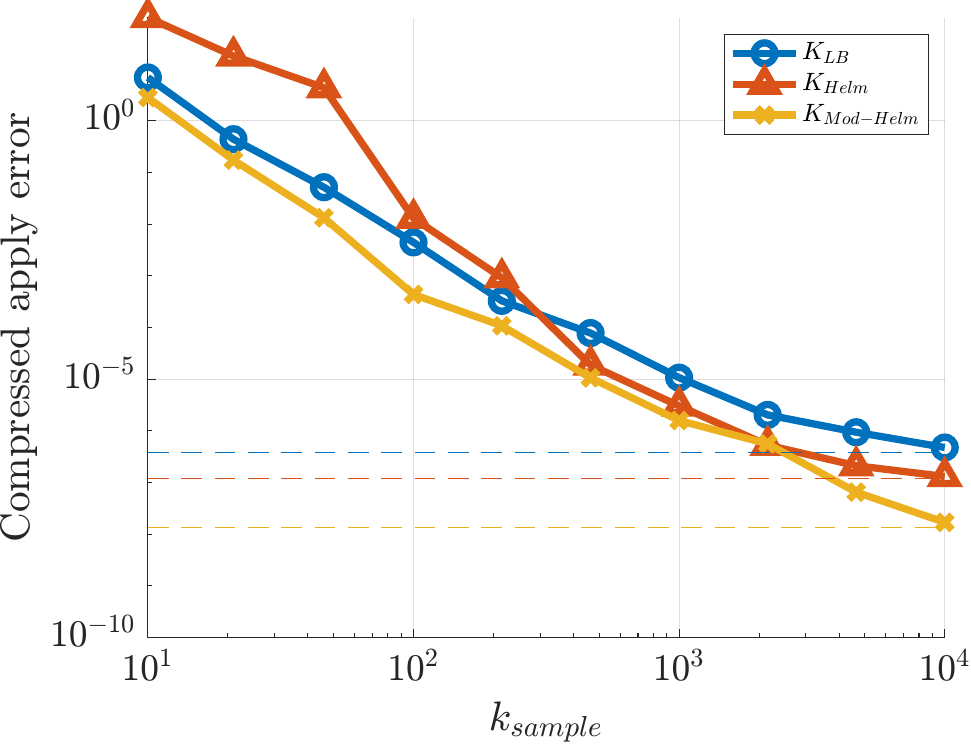}}
     \caption{The error induced in the~$N$-body calculation~\eqref{eq:Nbody}
       when the IFMM tolerance is set to~$10^{-9}$ and~$k_{s}$ is varied.
       The dashed line represents the error in the IFMM without randomized
       compression.}
     \label{fig:IFMM_confergencea}
   \end{subfigure}
   \quad
   \begin{subfigure}{0.45\linewidth}
     \centering{\includegraphics[width=.95\linewidth]{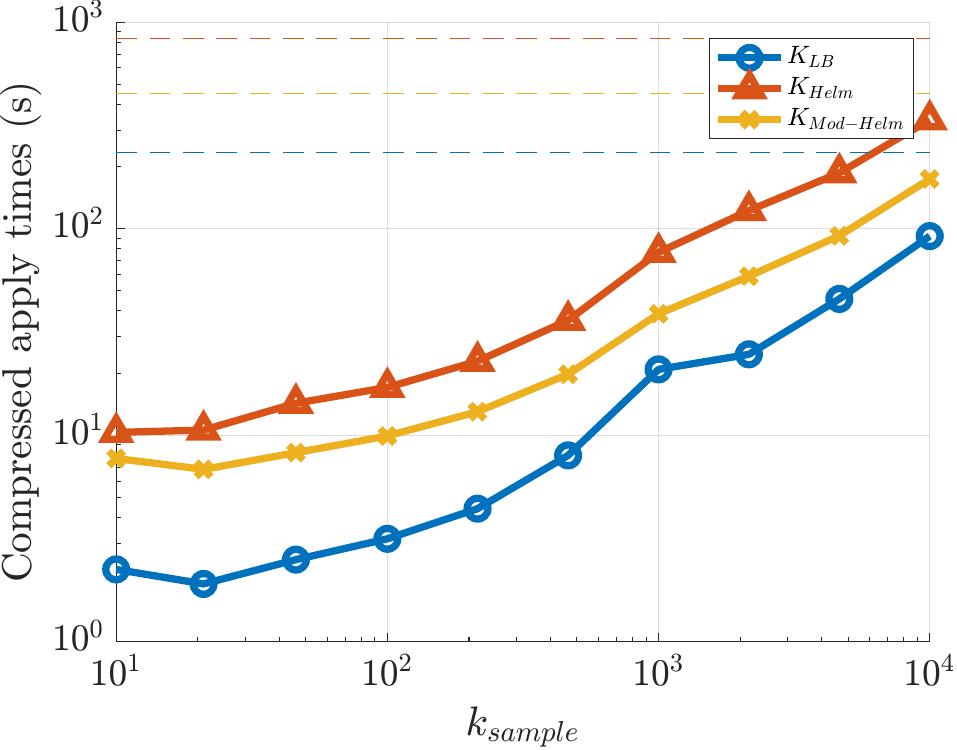}}
     \caption{The total time to use our IFMM to compute the~$N$-body
       calculation~\eqref{eq:Nbody} as a function of~$k_{s}$. The dashed
       line represents the time to use the IFMM without randomized compression.}
     \label{fig:IFMM_confergenceb}
   \end{subfigure}
   \caption[The convergence of our fast apply with respect to~$k_{s}$.]{The
     results of increasing~$k_{s}$ in our IFMM. We see that
     for~$k_{s}=10,000$, the IFMM has the same accuracy as the deterministic
     IFMM, but only involves a~$17\%$ of the points. This fact allows us to
     compute the IFMM to the same tolerance in a fraction of the time.}
    \label{fig:IFMM_confergence}
\end{figure}

Since the kernels appearing in the integral equations of this work are
\emph{not} formally PDE-kernels, as an alternative method to accelerate the
compression we compute and ID of the matrix~$\mtx{K}^{lk}$ by picking~$k_{s}$
target points (i.e. rows) from~$\mathcal{I}_l$ and~$k_{s}$ source points (i.e.
columns) from~$\mathcal{I}_k$. This class of methods has been previously
explored in~\cite{xing2020interpolative}. In this paper, we use the simplest
version of this algorithm where the sampled points are chosen randomly. In order
for this sampling to be robust, the random subset should be chosen to cover the
both~$B_k$ and~$B_l$. In practice, however, we have found that choosing points
independent of target and source locations has proved effective in the numerical
examples presented below. Our particular implementation is based on the
implementation of the IFMM found in the FLAM package~\cite{Ho2020}.

The approximation accuracy of our implementation of the
IFMM is validated by comparing it to the direct matrix apply. We
perform an $N$-body calculation
\begin{equation}
  \label{eq:Nbody}
  \phi_i = \sum_{j\neq i} F(\bx_i,\bx_j) \, \sigma_j,
\end{equation}
where the $\bx_i$'s are a discretization of a wavy torus shown
in~Figure~\ref{fig:LBsurfaces} using 58,752 points. (The same geometry is used
later on in the numerical experiments section.) We
choose~$\sigma_i=(\bx_i)_1+\eta_i$ where~$(\bx_i)_{1}$ denotes the
$x$-coordinate of~$\bx_{i}$ and $\eta_i$ is a realization of a standard normal
random variable. In~Figure~\ref{fig:IFMM_confergencea} we show the~$l^2$ error
between the direct calculation and the IFMM result, and how that error depends
on the choice of~$k_{s}$. In~Figure~\ref{fig:IFMM_confergenceb} we show how our
randomized compression accelerates the IFMM. In all of these tests, we set the
relative ID compression tolerance to~$10^{-9}$. We provide three different cases
in each subplot in Figure~\ref{fig:IFMM_confergence}: (1) with~$F$ set to the
parametrix for the Laplace-Beltrami problem, (2) $F$ set to the parametrix for
the Helmholtz-Beltrami problem, and (3) with $F$ set to the
modified-Helmholtz-Beltrami paramtrix. We also tested with~$F$ set to the
associated remainder functions and found similar results.

Lastly, we also test the complexity of this scheme using the same surface and
numerical tolerance as in the previous test; $k_{s}$ is set to $10,000$. The time
taken to build and apply with various different surface discretizations is shown
in~Figure~\ref{fig:sampling_time_scaling}. Once~$N>k_{s}$, we see that the
method scales as~$O(N)$.

\begin{figure}[!t]
    \centering
    \includegraphics[width=0.7\textwidth]{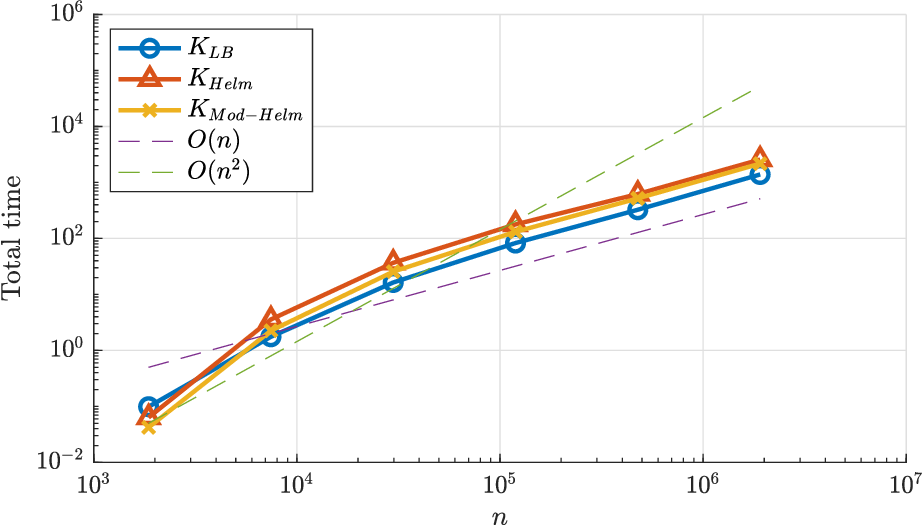}
    \caption{The time to build and apply our IFMM on successively finer
      discretizations of a wavy torus.}
    \label{fig:sampling_time_scaling}
\end{figure}

\section{Numerical Experiments}
\label{sec:Numeric_Exp}
In this section, we report the results of some numerical examples demonstrating
the integral equation solvers detailed above. In all of the tests, GMRES was
used to solve the linear system. The GMRES, integration, and compression
tolerances were set to~$10^{-9}$. In all tests, our errors are measured in the
relative~$L^2$ sense. In all tests we set~$k_{s}=10,000$ because the above
experiments indicate that that is enough samples for the IFMM to achieve the
same error as the deterministic method with this tolerance.

We shall conduct tests on three surfaces: the unit sphere (\figref{fig:sphere_testa}); an ellipsoid with semi axis lengths 4.5, 2.25, and 3 (\figref{fig:LBsurfaces}, left); and a wavy torus (\figref{fig:LBsurfaces}, right), given by the parameterization
\begin{equation}
    \bs r(u,v) = \begin{bmatrix}
        (3+\cos(v)+0.6\cos(5u))\cos(u) \\
        (3+\cos(v)+0.6\cos(5u))\sin(u) \\
        \sin(v)
    \end{bmatrix}\quad\text{where}\quad (u,\,v) \in [0,2\pi)^2.
\end{equation}
In Section~\ref{sec:bvps}, where we solve boundary value problems, we consider~$\Gamma$ to be the wavy torus with~$u$ restricted to lie in $(4\pi/3,2\pi)$.

\subsection{Spherical harmonic validation}
It is well known that the spherical harmonics~$Y_l^m$ are the eigenfunctions of
the Laplace-Beltrami operator on a sphere; their eigenvalues are~$-l(l+1)$. Since the
spherical harmonics form an orthogonal basis for~$L^2$ on the sphere, we can
expand any~$f$ in terms of the spherical harmonics by computing inner products.
We can then analytically solve of the Laplace-Beltrami problem by dividing the
expansion coefficients by the corresponding eigenvalues and then summing the
resulting expansion.

We used this observation to test our method by letting~$f$ be a random linear combination
of~$Y_1^0, Y_3^3, Y_{10}^6,$ and $Y_{10}^7$ and comparing the solution given by
our solver to the analytic solution (Figure~\ref{fig:sphere_testa}). Since
adding a constant~$c$ to the Laplace-Beltrami operator simply adds~$c$ to the
eigenvalues, we can use the same method to test our solver for the Helmholtz-
and modified Helmholtz-Beltrami problems. The errors in our computations are
shown in~Figure~\ref{fig:sphere_testb}. We can see that the method converges to
around the specified tolerance for each value of~$c$. The solver had slightly
less accuracy when~$c$ was negative because the singular quadrature routine was
not designed for the exponential decay in the parametrix.

\begin{figure}[t]
    \centering
    \begin{subfigure}{0.45\linewidth}
     \centering{\includegraphics[width=.95\linewidth]{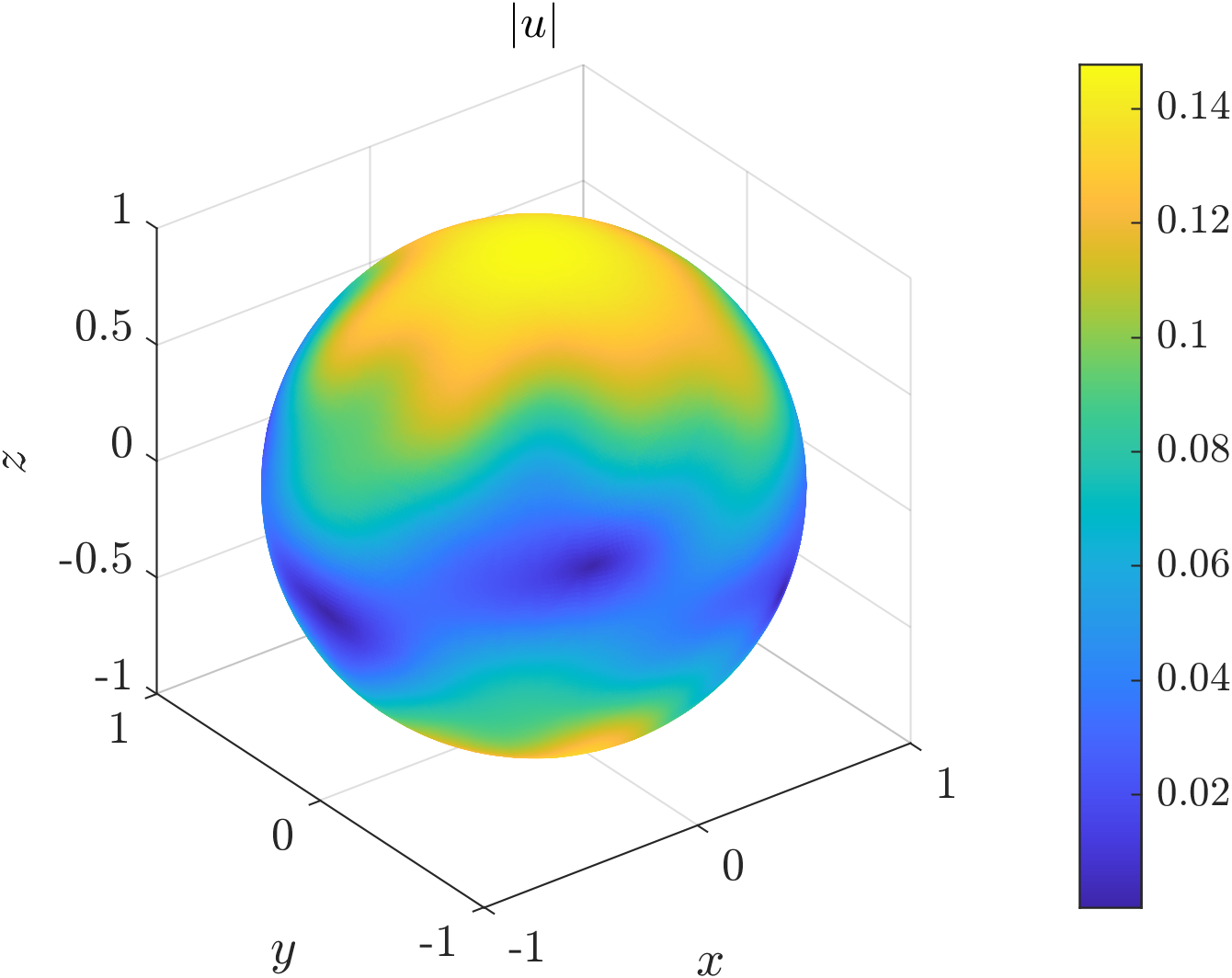}}
     \caption{A exact solution obtained from a random linear combination of four
       spherical harmonics with~$l$ varying from 1 to 10. The solution is shown
       on the finest discretization.}
     \label{fig:sphere_testa}
   \end{subfigure}
   \quad
       \begin{subfigure}{0.45\linewidth}
     \centering{\includegraphics[width=.95\linewidth]{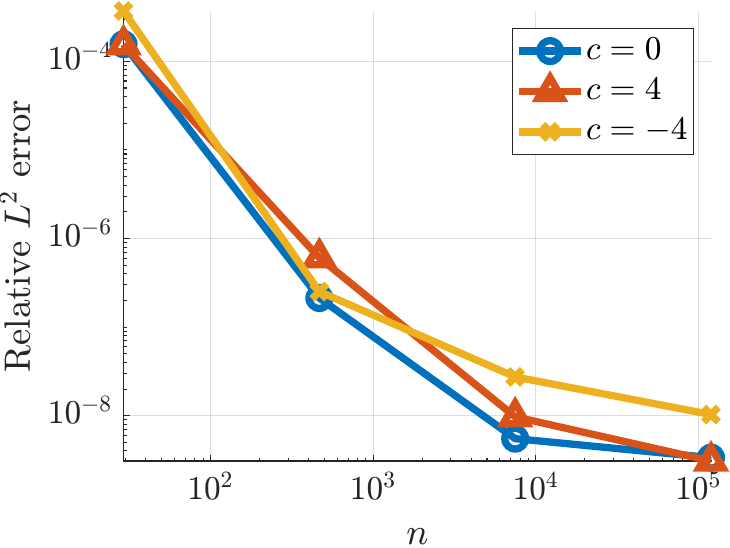}}
     \caption{The error in our method when~$c$ and the number of surface discretization points are varied.}\label{fig:sphere_testb}
   \end{subfigure}
    \caption[The results of testing our parametrix solver on a sphere.]{The results of testing our parametrix solver on a sphere.}
    \label{fig:sphere_test}
\end{figure}


\subsection{A Laplace-Beltrami problem}

In order to validate our Laplace-Beltrami solver on a broader class of surfaces,
we use the method of manufactured solutions. Following the approach
in~\cite{ONeil2018}, we did this by choosing~$\Gamma$ to be a smooth surface and
choosing the exact solution to be a constant plus the restriction of a smooth
function $v$ defined in all of $\bbR^3$. We then generated the corresponding
right hand side $f$ by applying the formula~\eqref{eq:diff_remainder} to~$v$:
\begin{equation}
  \label{eq:LB_rest_p}
  f=\LB \left(v|_\Gamma\right) = \Delta v-2H \frac{\partial
    v}{\partial \bs n}-\frac{\partial^2 v}{\partial \bs
    n^2} .
\end{equation}
We evaluate $f$ by analytically computing
the terms in~\eqref{eq:LB_rest_p} based on a global parameterization of
the surface.

In this example,~$\Gamma$ is given by either a wavy torus or an ellipsoid (Figure~\ref{fig:LBsurfaces}). We set~$v$ to be the Newtonian potential centered at~{$\bs x_0=(-3,-3,4)$}, which is a point outside of, but fairly close to, both choices for~$\Gamma$:
\begin{equation*}
    v(\bs x) = -\frac{1}{|\bs x - \bs x_0|}.
\end{equation*}
The exact solution to this problem is given by
$u=v|_\Gamma-\frac{1}{|\Gamma|}\int_\Gamma v$, where~$|\Gamma|$ is the surface
area of~$\Gamma$. The reference solutions are shown
in~Figure~\ref{fig:LBsurfaces}. We conducted a convergence test by repeatedly
refining the triangles used in our surface discretization and thus varying~$n$.
We can see from~Figures~\ref{fig:torus_testevp} and~\ref{fig:ell_testevp} that
our method is capable of achieving small relative errors, and therefore the
solver is working as expected. From~Figures~\ref{fig:torus_testtvp}
and~\ref{fig:ell_testtvp}, we can see that the computational time of our method
increases super-linearly, but not quadratically on the tested range, so our fast
apply is working and our matrices are indeed compressible.

\begin{figure}
    \centering
    \includegraphics[width=.95\textwidth]{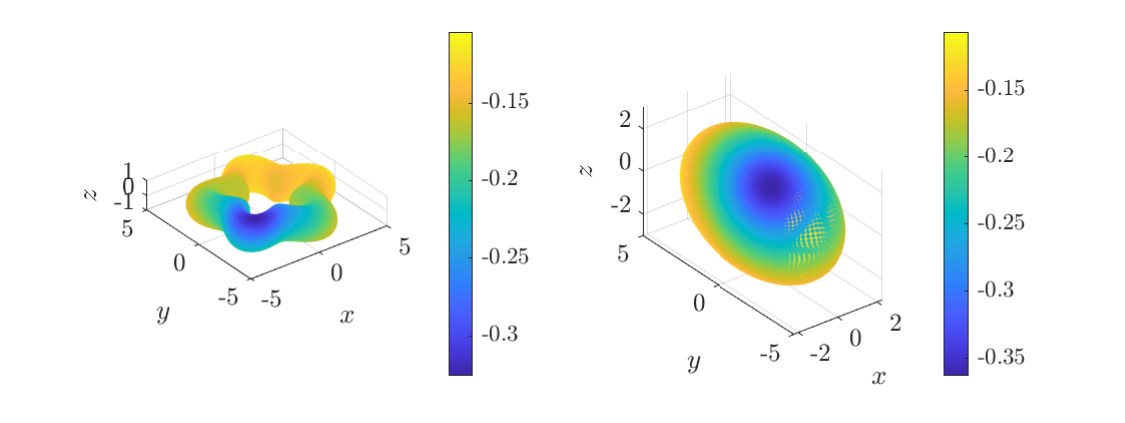}
    \caption{The reference solution for the Laplace-Beltrami solver on two choices of~$\Gamma$.}
    \label{fig:LBsurfaces}
\end{figure}

\begin{figure}
    \centering
    \includegraphics[width=.8\textwidth]{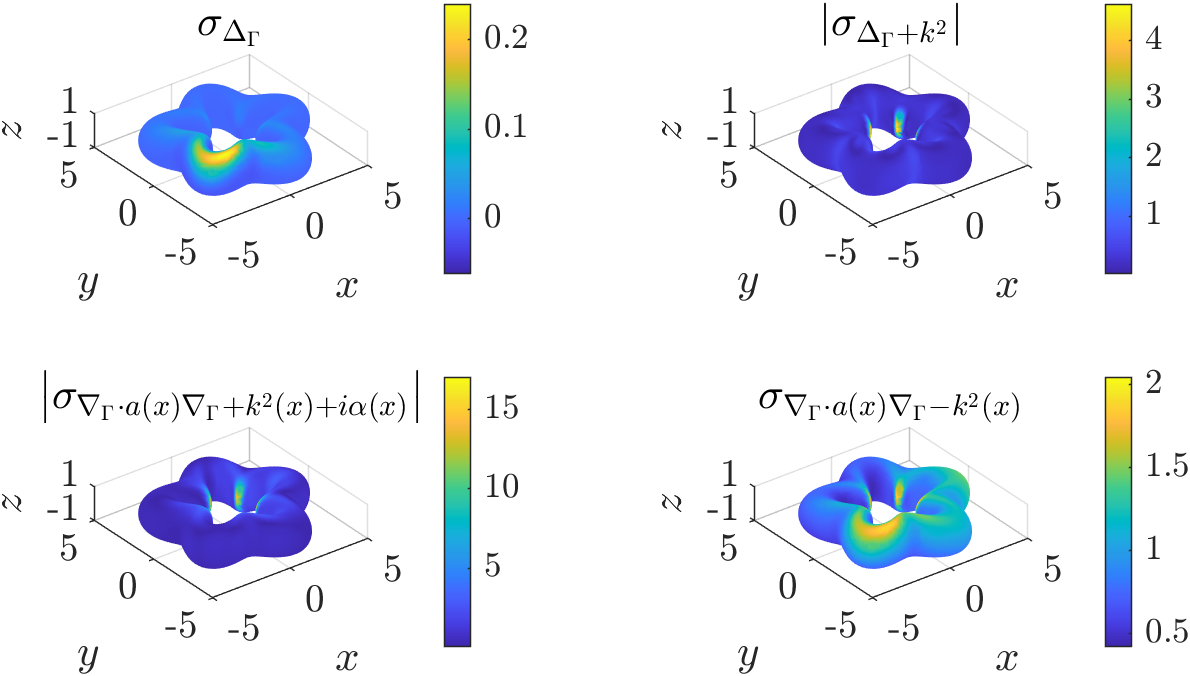}
    \caption{The solution densities on a wavy torus for various surface PDEs. All densities correspond to the same solution~$u$.}
    \label{fig:sigmas}
\end{figure}

\begin{figure}
    \centering
    \includegraphics[width=.8\textwidth]{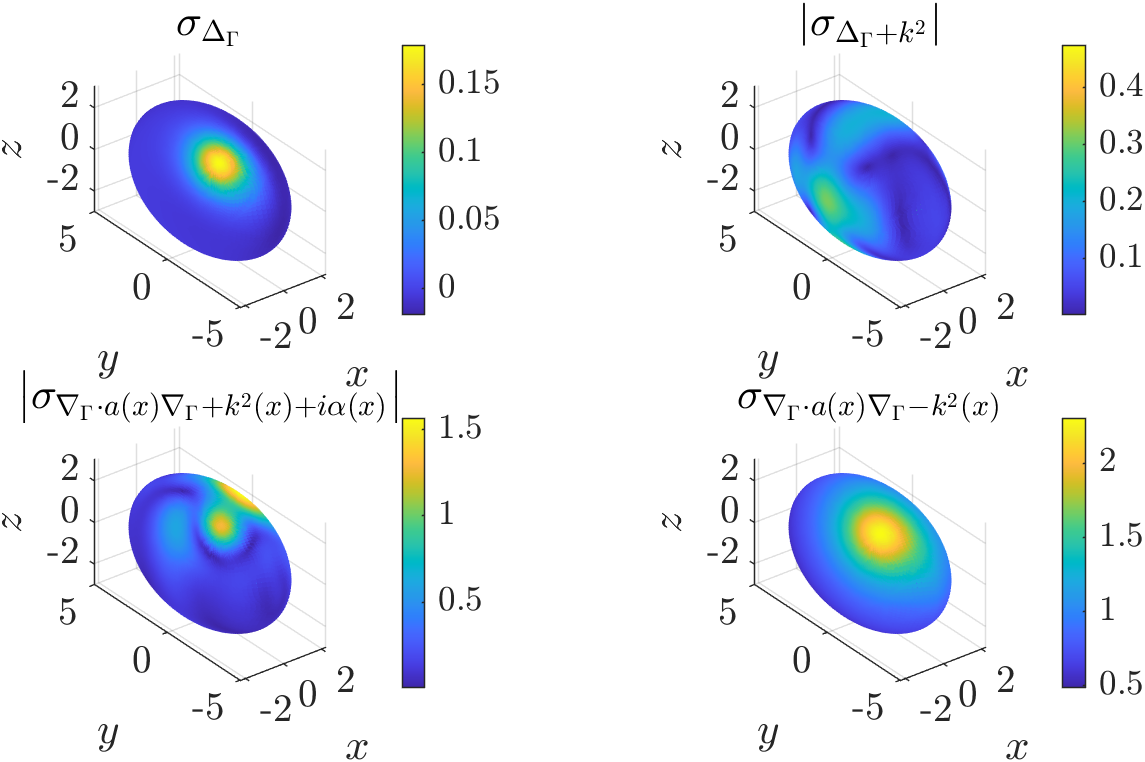}
    \caption{The solution densities on an ellipsoid for various surface PDEs. All densities correspond to the same solution~$u$}
    \label{fig:sigmas_ell}
\end{figure}

\begin{figure}
    \centering
        \begin{subfigure}{0.45\linewidth}
    \includegraphics[width=.95\textwidth]{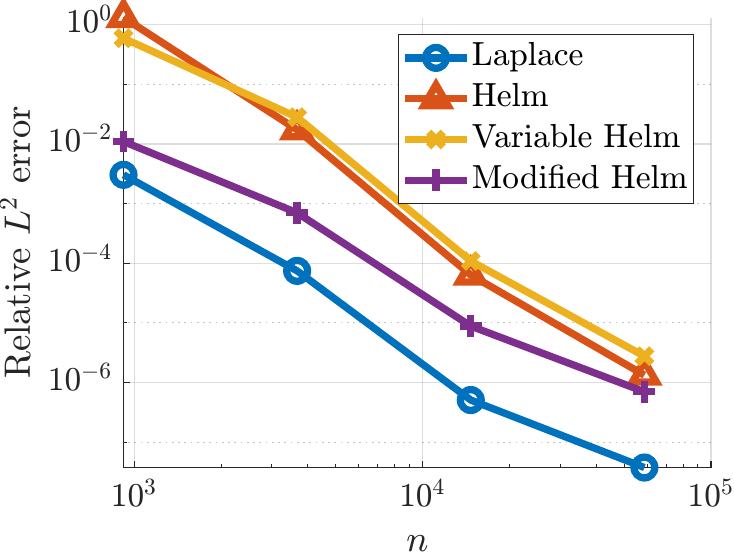}
    \caption{The errors of the computed solution in the wavy torus
      tests where $n$ is the number of points used to discretize the surface.}
    \label{fig:torus_testevp}
   \end{subfigure}
   \quad
       \begin{subfigure}{0.45\linewidth}
    \centering
    \includegraphics[width=.95\textwidth]{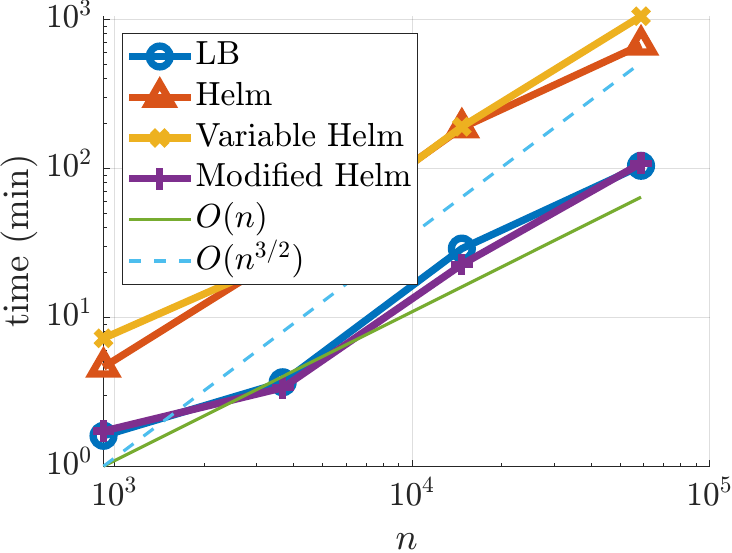}
    \caption{The time to compute the solution in the wavy torus
      tests where $n$ is the number of points used to discretize the surface.}
    \label{fig:torus_testtvp}
       \end{subfigure}
       \caption{Results of the wavy torus
      tests.}
       \label{fig:torus_test_res}
\end{figure}

\begin{figure}
    \centering
        \begin{subfigure}{0.45\linewidth}
    \includegraphics[width=.95\textwidth]{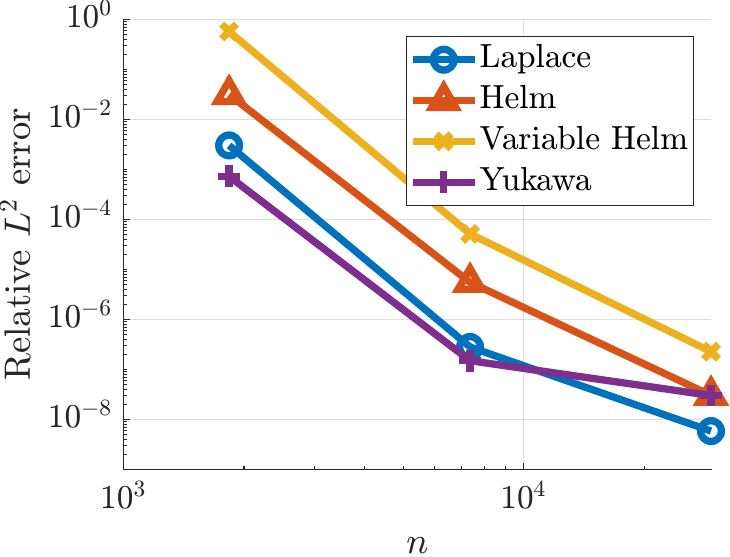}
    \caption{The errors of the computed solution in the  ellipsoid
      tests where $n$ is the number of points used to discretize the surface.}
    \label{fig:ell_testevp}
   \end{subfigure}
   \quad
       \begin{subfigure}{0.45\linewidth}
    \centering
    \includegraphics[width=.95\textwidth]{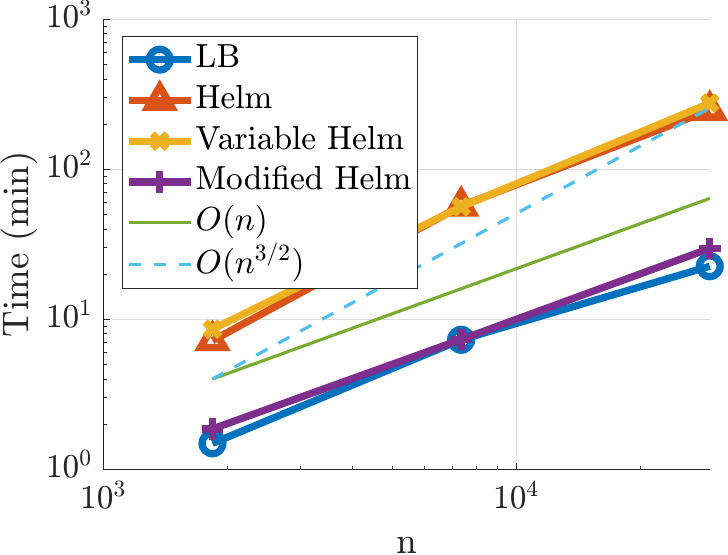}
    \caption{The time to compute the solution in the ellipsoid
      tests where $n$ is the number of points used to discretize the surface.}
    \label{fig:ell_testtvp}
       \end{subfigure}
       \caption{Results of the ellipsoid
      tests.}
       \label{fig:ell_test_res}
\end{figure}


\subsection{A Helmholtz-Beltrami problem}

We now test our method on the Helmholtz-Beltrami problem with~$c$ set to the
moderate value of 4. As before, we use the method of manufactured solutions. We
again let~$v$ be a Newtonian potential and we set the solution to
be~$u=v|_\Gamma$ and construct the right hand
side~$f=\LB \left(v|_\Gamma\right) +c v|_\Gamma$, where we have again
used~\eqref{eq:diff_remainder} to compute~$f$ analytically. The results of our
tests are shown in~Figures~\ref{fig:torus_test_res} and~\ref{fig:ell_test_res}.
We see that the method is less accurate on this surface PDE, but still converges
at a simliar rate. In Figures~\ref{fig:sigmas} and~\ref{fig:sigmas_ell}, we see
that this limited accuracy is likely due to the fact that the density~$\sigma$
is not sufficiently resolved. Figure~\ref{fig:torus_test_res} also shows that
the solver takes longer to solve the Helmholtz-Beltrami problem than the
Laplace-Beltrami problem. This was predominantly caused by the increased number
of GMRES iterations required to solve the linear system.

\subsection{A variable coefficient surface PDE}

In this experiment, we test our method on a problem where the parameters~$a$ and
$c$ are not constant. In this case, the remainder kernel~$R$ is weakly singular,
but adaptive integration can still be used to accurately discretize the integral
equation. We use the same singular quadrature routine as before due to the fact
that the singularity in~$R$ will be cancelled via the change to polar
coordinates. We test our code in the same way as before, by
choosing~$u=v|_\Gamma$, where~$v$ is a Newtonian potential and analytically
computing the corresponding right hand side.

For this experiment, we set~$a$ to be a linear function of the coordinate~$z$, chosen
to vary by a factor of three over the surface. We also
set~$c=-4-\exp(x/2)$. We also tested our solver with $a=1+\exp(z)$ and
$c=4 \, (1+0.2y+0.2 i \sin x)$.

We can see from~Figures~\ref{fig:torus_test_res} and~\ref{fig:ell_test_res} that
introducing variable coefficients did not greatly impact the error in the
solution or the time to compute it.

\subsection{Boundary value problems}
\label{sec:bvps}

We now test our method for solving the Laplace-Beltrami equation with Neumann or
Dirichlet boundary conditions. We choose~$\Gamma$ to be a subset of the wavy torus
from the previous three sections.
The performance of this discretization is tested by estimating the error in
applying the operators~$\cR_\gamma$ and~$\cT_\gamma$ to a specified
function~$\mu$ on the example surface shown in~Figure~\ref{fig:edge_conv_surf}.
As before, we set all of the our IFMM tolerances to~$10^{-9}$. We construct a
discretization of the surface using the method above, and sample each boundary
(curve) using 200 points equispaced in the parameterization variable. We
estimate the quadrature error by comparing to the values given by the finest
discretizations. In this test we chose~$\mu = \exp(x)$ and used the
operators~$\cR_\gamma$ and~$\cT_\gamma$ as in the Laplace-Beltrami problem. The
estimated errors for various oversampling factors are shown
in~Figure~\ref{fig:edge_conv}. We can see that once the near singularity is
resolved, the error decays quite quickly. We can also see that more points are
required to resolve the singularity when the surface discretization is refined
and discretization nodes are brought closer to the boundary.

\begin{figure}[t]
  \centering
  \includegraphics[width=.8\linewidth]{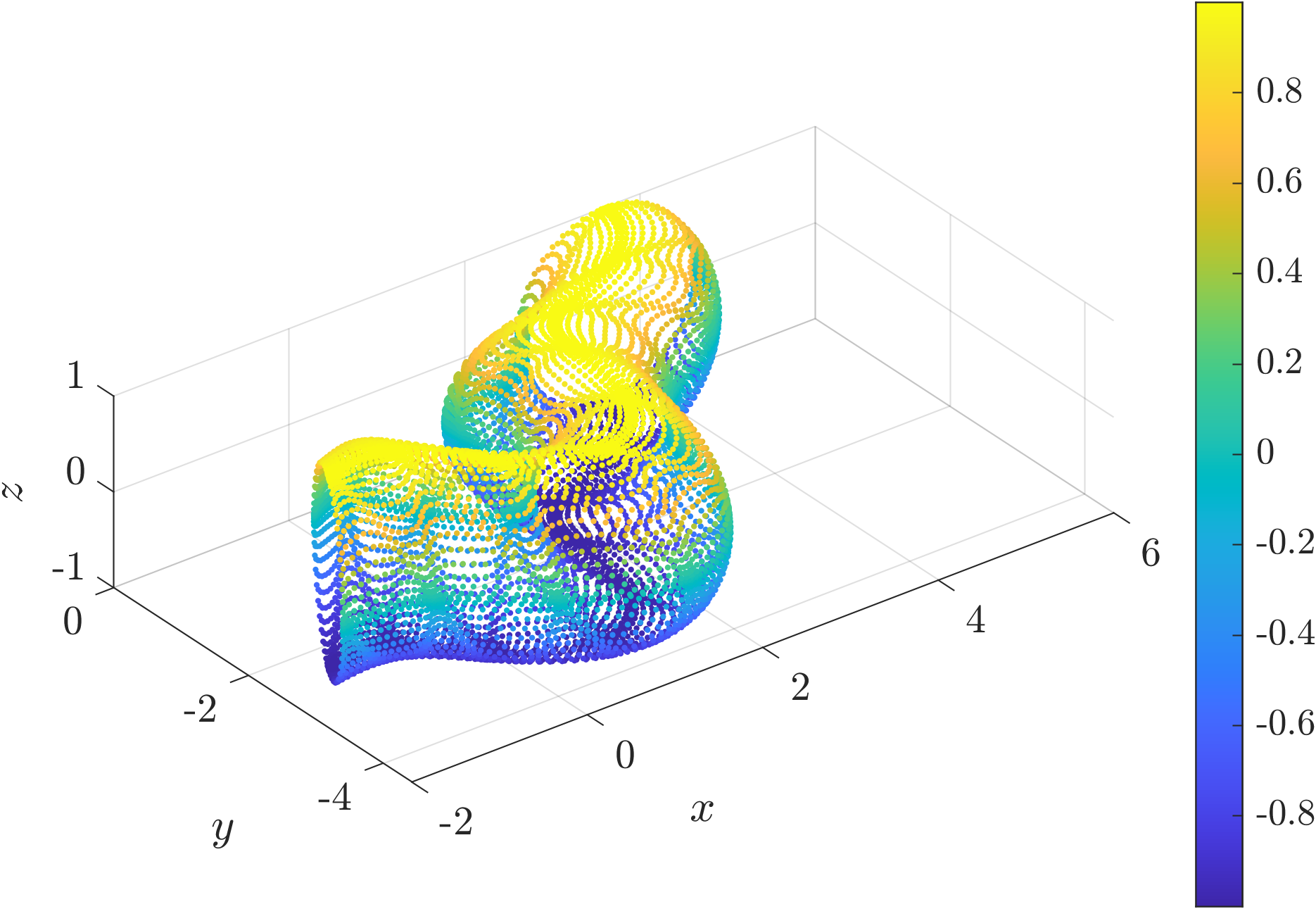}
  \caption{The example surface for the edge quadrature test.}
  \label{fig:edge_conv_surf}
\end{figure}

We now test our solver for surface boundary value problems.
We again choose the solution~$u$ to be the
restriction of a Newtonian potential to~$\Gamma$ and analytically compute the
corresponding~$f$ and~$f_\gamma$. Each subset of~$\gamma=\partial\Gamma$ is
discretized with~$n_\phi$ equispaced nodes and we set the oversampling factor to
64. In subsequent tests, we refined our discretization by independently
splitting each triangle in our surface discretization and doubling the number of
points used to discretize each boundary.

The results of our tests are shown in~Figure~\ref{fig:bdry-data}. We can see
that as both the surface and edge discretizations are refined, the error
decreases as expected. As in the edge quadrature test, the stronger singularity
in the kernel~$\cT_\gamma$ results in the Dirichlet boundary value problem
requiring an increased number of points in order to achieve a similar accuracy.

\afterpage{
  \clearpage

\begin{figure}[t]
  \centering
  \begin{subfigure}{0.45\linewidth}
    \centering{\includegraphics[width=.95\linewidth]{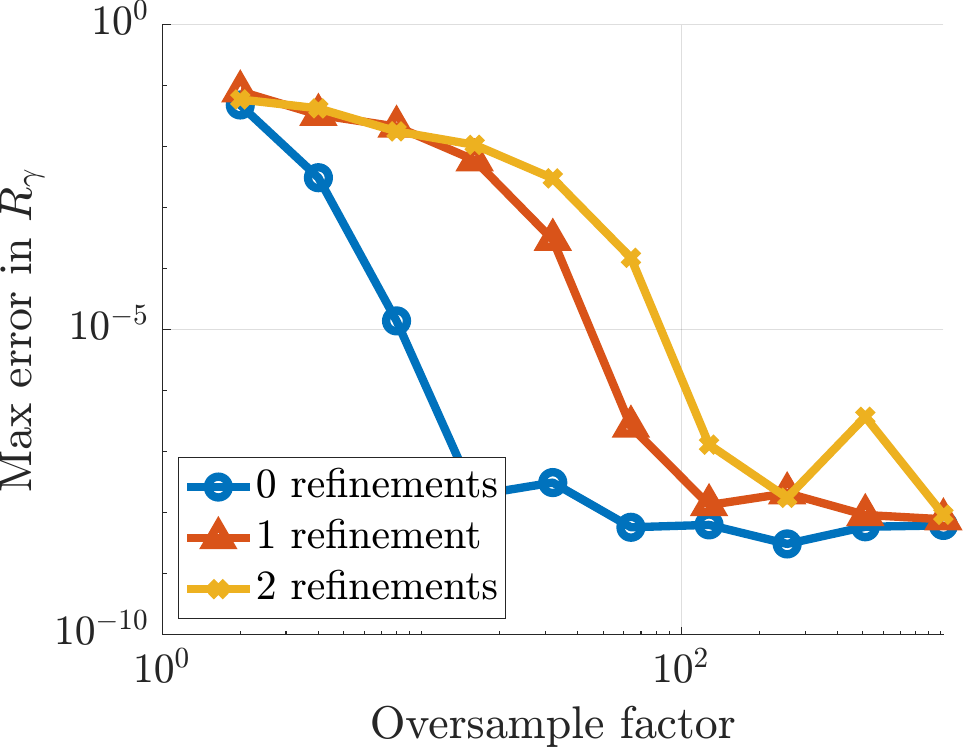}}
    \caption{The test with $\cR_\gamma$, which is required for the
      Neumann boundary value problems.}\label{fig:edge_conva}
  \end{subfigure}
  \quad
  \begin{subfigure}{0.45\linewidth}
    \centering{\includegraphics[width=.95\linewidth]{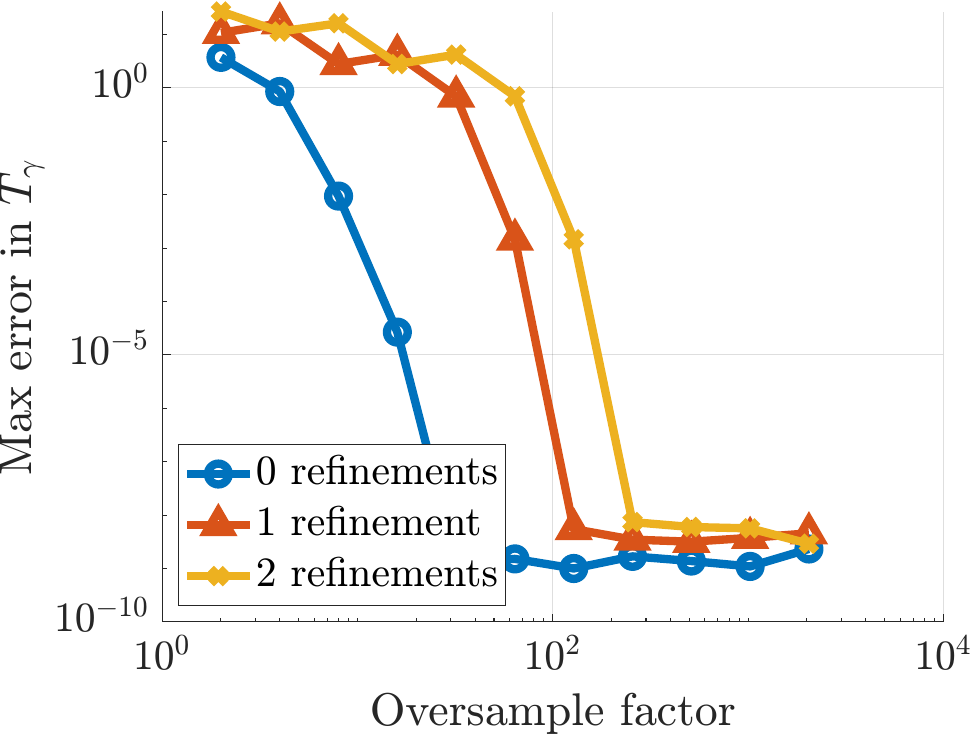}}
    \caption{The test with $\cT_\gamma$, which is required for the
      Dirichlet boundary value problems.}\label{fig:edge_convb}
  \end{subfigure}
  \caption[Edge quadrature test]{This graph shows how the error in the
    application of~$\cR_\gamma$ and~$\cT_\gamma$ depends on the
    oversampling factor and the density of the surface
    discretization.}
    \label{fig:edge_conv}
\end{figure}

\begin{figure}[b]
    \centering
    \begin{subfigure}{0.45\linewidth}
     \centering{\includegraphics[width=0.95\textwidth]{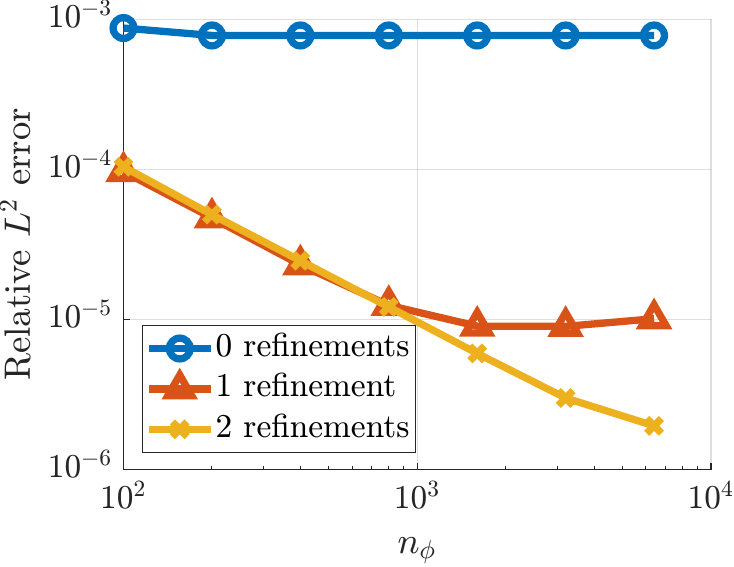}
    \caption{The error in the Neumann boundary value solver.}}
     
   \end{subfigure}
   \quad
    \begin{subfigure}{0.45\linewidth}
     \centering{\includegraphics[width=0.95\textwidth]{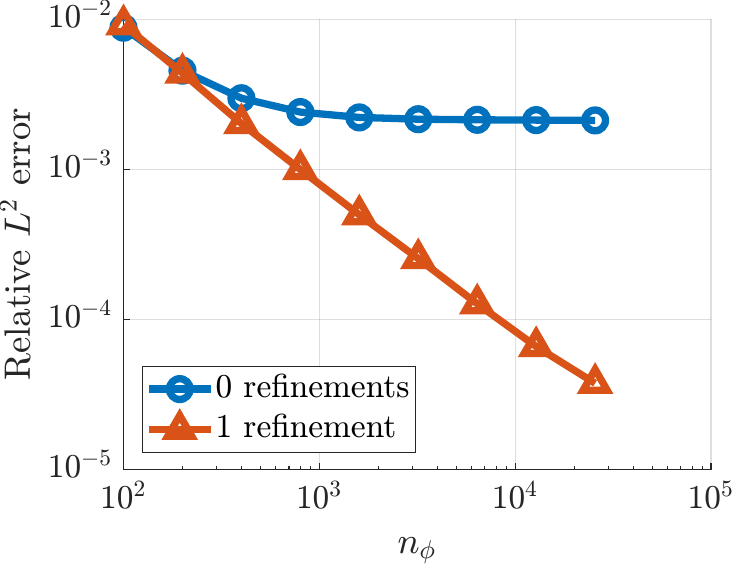}
    \caption{The error in the Dirichlet boundary value solver.}}
   \end{subfigure}

   \caption{The error in our boundary value problem solvers. In these tests we
     measure the~$L^2$ error as we refine the surface and edge discretizations.}
    \label{fig:bdry-data}
\end{figure}

\clearpage
}

\section{Conclusion}
\label{sec:conclussion}

In this work we have presented a general framework for reformulating variable
coefficient surface PDEs as boundary integral equations using a family of
parametrices that scale according to the coefficients in the particular surface
PDE in question. The resulting integral equations are second-kind and quite
well-conditioned. Furthermore, we've demonstrated via a battery of numerical
examples that they can be rapidly solved using a combination of hierarchical
linear algebraic compression coupled with a iterative solver, such as GMRES. The
associated kernels of the integral equations are either bounded or weakly
singular, and some special-purpose quadrature schemes were developed to
accurately compute the boundary integrals.


Going forward, the analytical and numerical ideas presented here can be extended
in various ways, including to solving problems on curves in two dimensions and
for solving anisotropic problems, i.e., problems with leading
term~$\nabla_{\Gamma} \cdot \mtx{A} \nabla_{\Gamma}$, where~$\mtx{A}$ is
non-degenerate (possibly non-constant) matrix. In this latter case, preliminary
choices for a parametrix resulted in integral equations having kernels which
were quite singular. Additional analysis and numerical experiments are required
to develop a more suitable parametrix.

Lastly, there is significant evidence that even faster matrix-vector multiplies
can be obtained using a more careful analysis of the remainder kernels, in the
style of multipole expansions for true PDE kernels (such as~$1/r$
or~$e^{ikr}/r$). Developing such fast algorithms is actively being investigated.

\bibliography{parametrix_bib.bib}

\begin{thebibliography}{10}

\bibitem{agarwal2023fmm}
D.~Agarwal, M.~O'Neil, and M.~Rachh.
\newblock {FMM-accelerated solvers for the Laplace-Beltrami problem on complex
  surfaces in three dimensions}.
\newblock {\em J. Sci. Comput.}, 97:25, 2023.

\bibitem{Alves2018}
C.~J. Alves and P.~R. Antunes.
\newblock {The method of fundamental solutions applied to boundary value
  problems on the surface of a sphere}.
\newblock {\em Comp. Math. Appl.}, 75(7):2365--2373, 2018.

\bibitem{andreux2015}
M.~Andreux, E.~Rodol{\`a}, M.~Aubry, and D.~Cremers.
\newblock {Anisotropic Laplace-Beltrami Operators for Shape Analysis}.
\newblock In L.~Agapito, M.~M. Bronstein, and C.~Rother, editors, {\em Computer
  Vision - ECCV 2014 Workshops}, pages 299--312, Cham, 2015. Springer
  International Publishing.

\bibitem{angenent1999brain}
S.~Angenent, S.~Haker, A.~Tannenbaum, and R.~Kikinis.
\newblock {On the Laplace-Beltrami Operator and Brain Surface Flattening}.
\newblock {\em IEEE Trans. Med. Imag.}, 18(8):700--711, 1999.

\bibitem{bansch2005finite}
E.~B{\"a}nsch, P.~Morin, and R.~H. Nochetto.
\newblock {A finite element method for surface diffusion: The parametric case}.
\newblock {\em J. Comput. Phys.}, 203(1):321--343, 2005.

\bibitem{beirao2014hitchhiker}
L.~Beir{\~a}o~da Veiga, F.~Brezzi, L.~D. Marini, and A.~Russo.
\newblock The hitchhiker's guide to the virtual element method.
\newblock {\em Math. Models and Meth. Appl. Sci.}, 24(08):1541--1573, 2014.

\bibitem{bertalmio2001variational}
M.~Bertalm{\'i}o, L.-T. Cheng, S.~Osher, and G.~Sapiro.
\newblock Variational problems and partial differential equations on implicit
  surfaces.
\newblock {\em J. Comput. Phys.}, 174(2):759--780, 2001.

\bibitem{Beshley2018}
A.~Beshley, R.~Chapko, and B.~T. Johansson.
\newblock {An integral equation method for the numerical solution of a
  Dirichlet problem for second-order elliptic equations with variable
  coefficients}.
\newblock {\em J. Eng. Math.}, 112(1):63--73, 2018.

\bibitem{Bonito2020}
A.~Bonito, A.~Demlow, and R.~H. Nochetto.
\newblock {\em {Finite element methods for the Laplace–Beltrami operator}},
  volume~21.
\newblock Elsevier B.V., 1 edition, 2020.

\bibitem{bonito2018}
A.~Bonito, A.~Demlow, and J.~Owen.
\newblock {A Priori Error Estimates for Finite Element Approximations to
  Eigenvalues and Eigenfunctions of the Laplace--Beltrami Operator}.
\newblock {\em SIAM J. Num. Anal.}, 56(5):2963--2988, 2018.

\bibitem{bremer2012weakly}
J.~Bremer and Z.~Gimbutas.
\newblock A {N}ystr\"om method for weakly singular integral operators on
  surfaces.
\newblock {\em J. Comput. Phys.}, 231:4885--4903, 2012.

\bibitem{Brezis2011}
H.~Brezis.
\newblock {\em Functional analysis, Sobolev spaces and partial differential
  equations}.
\newblock Universitext. Springer, New York 

\bibitem{Burman2017}
E.~Burman, P.~Hansbo, M.~G. Larson, and A.~Massing.
\newblock {A cut discontinuous Galerkin method for the Laplace-Beltrami
  operator}.
\newblock {\em IMA J. Num. Anal.}, 37(1):138--169, 2017.

\bibitem{Burman2020}
E.~Burman, P.~Hansbo, M.~G. Larson, and A.~Massing.
\newblock {A stable cut finite element method for partial differential
  equations on surfaces: The Helmholtz–Beltrami operator}.
\newblock {\em Comput. Meth. Appl. Mech. Eng.}, 362:112803, 2020.

\bibitem{chen2015closest}
Y.~Chen and C.~B. Macdonald.
\newblock The closest point method and multigrid solvers for elliptic equations
  on surfaces.
\newblock {\em SIAM J. Sci. Comput.}, 37(1):A134--A155, 2015.

\bibitem{Chkadua2009}
O.~Chkadua, S.~E. Mikhailov, and D.~Natroshvil.
\newblock {Analysis of direct boundary-domain integral equations for a mixed
  BVP with variable coefficient, I: Equivalence and invertibility}.
\newblock {\em J. Integral Equ. Appl.}, 21(4):499--543, 2009.

\bibitem{Chkadua2010}
O.~Chkadua, S.~E. Mikhailov, and D.~Natroshvili.
\newblock {Analysis of direct boundary-domain integral equations for a mixed
  BVP with variable coefficient, II: Solution regularity and asymptotics}.
\newblock {\em J. Integral Equ. Appl.}, 22(1):19--37, 2010.

\bibitem{Chkadua2013}
O.~Chkadua, S.~E. Mikhailov, and D.~Natroshvili.
\newblock {Localized Boundary-Domain Singular Integral Equations Based on
  Harmonic Parametrix for Divergence-Form Elliptic PDEs with Variable Matrix
  Coefficients}.
\newblock {\em Integral Equ. Oper. Theory}, 76(4):509--547, 2013.

\bibitem{Chkadua2018}
O.~Chkadua, S.~E. Mikhailov, and D.~Natroshvili.
\newblock {Singular localised boundary-domain integral equations of acoustic
  scattering by inhomogeneous anisotropic obstacle}.
\newblock {\em Math. Meth. Appl. Sci.}, 41(17):8033--8058, 2018.

\bibitem{chung2004}
M.~Chung and J.~Taylor.
\newblock Diffusion smoothing on brain surface via finite element method.
\newblock In {\em {2004 2nd IEEE International Symposium on Biomedical Imaging:
  Nano to Macro (IEEE Cat No. 04EX821)}}, pages 432--435 Vol. 1, 2004.

\bibitem{demlow2007adaptive}
A.~Demlow and G.~Dziuk.
\newblock {An adaptive finite element method for the Laplace--Beltrami operator
  on implicitly defined surfaces}.
\newblock {\em SIAM J. Num. Anal.}, 45(1):421--442, 2007.

\bibitem{Dziuk1988}
G.~Dziuk.
\newblock {Finite elements for the Beltrami operator on arbitrary surfaces}.
\newblock In {\em Partial differential equations and calculus of variations},
  pages 142--155. Springer, Berlin, Heidelberg, 1988.

\bibitem{dziuk2013finite}
G.~Dziuk and C.~M. Elliott.
\newblock {Finite element methods for surface PDEs}.
\newblock {\em Acta Numer.}, 22:289--396, 2013.

\bibitem{Epstein2009}
C.~L. Epstein and L.~Greengard.
\newblock {Debye sources and the numerical solution of the time harmonic
  Maxwell equations}.
\newblock {\em Comm. Pure Appl. Math.}, dec 2009.

\bibitem{Epstein2013}
C.~L. Epstein, L.~Greengard, and M.~O'Neil.
\newblock {Debye Sources and the Numerical Solution of the Time Harmonic
  Maxwell Equations II}.
\newblock {\em Comm. Pure Appl. Math.}, 66(5):753--789, 2013.

\bibitem{Epstein2019}
C.~L. Epstein, L.~Greengard, and M.~O'Neil.
\newblock {A high-order wideband direct solver for electromagnetic scattering
  from bodies of revolution}.
\newblock {\em J. Comput. Phys.}, 387:205--229, 2019.

\bibitem{escher1998surface}
J.~Escher, U.~F. Mayer, and G.~Simonett.
\newblock The surface diffusion flow for immersed hypersurfaces.
\newblock {\em SIAM J. Math. Anal.}, 29(6):1419--1433, 1998.

\bibitem{Evans2010}
L.~C. Evans.
\newblock {\em {Partial Differential Equations}}.
\newblock American Mathematical Society, Providence, Rhode Island, 2nd edition,
  2010.

\bibitem{Folland1995}
G.~B. Folland.
\newblock {\em {Introduction to Partial Differential Equations}}.
\newblock Princeton University Press, Princeton, New Jersey, 2nd edition, 1995.

\bibitem{fortunato2022high}
D.~Fortunato.
\newblock {A high-order fast direct solver for surface PDEs}.
\newblock {\em SIAM J. Sci. Comput.}, 2023.

\bibitem{Fresneda-Portillo2021}
C.~Fresneda-Portillo and Z.~Woldemicheal.
\newblock {A new family of boundary-domain integral equations for the Dirichlet
  problem of the diffusion equation in inhomogeneous media with H-1($\Omega$)
  source term on Lipschitz domains}.
\newblock {\em Math. Meth. Appl. Sci.}, 44(12):9817--9830, 2021.

\bibitem{Fresneda-Portillo2020}
C.~Fresneda-Portillo and Z.~W. Woldemicheal.
\newblock {On the existence of solution of the boundary-domain integral
  equation system derived from the 2D Dirichlet problem for the diffusion
  equation with variable coefficient using a novel parametrix}.
\newblock {\em Complex Var. Elliptic Equ.}, 65(12):2056--2070, 2020.

\bibitem{frittelli2018virtual}
M.~Frittelli and I.~Sgura.
\newblock {Virtual element method for the Laplace-Beltrami equation on
  surfaces}.
\newblock {\em ESAIM: Math. Modell. Num. Anal.}, 52(3):965--993, 2018.

\bibitem{Garabedian1986}
P.~Garabedian.
\newblock {\em {Partial Differential Equations}}.
\newblock Chelsea Publish Company, New York, NY, 2nd edition, 1986.

\bibitem{Greengard2021}
L.~Greengard, M.~O'Neil, M.~Rachh, and F.~Vico.
\newblock {Fast multipole methods for the evaluation of layer potentials with
  locally-corrected quadratures}.
\newblock {\em J. Comput. Phys.: X}, 10:100092, 2021.

\bibitem{greengard1987fast}
L.~Greengard and V.~Rokhlin.
\newblock A fast algorithm for particle simulations.
\newblock {\em J. Comput. Phys.}, 73(2):325--348, 1987.

\bibitem{Grinfeld2010}
P.~Grinfeld.
\newblock {Hamiltonian Dynamic Equations for Fluid Films}.
\newblock {\em Studies in Applied Mathematics}, 125(3):223--264, 2010.

\bibitem{Grinfeld2012}
P.~Grinfeld.
\newblock {Small oscillations of a soap bubble}.
\newblock {\em Stud. Appl. Math.}, 128(1):30--39, 2012.

\bibitem{hadamard1932probleme}
J.~Hadamard.
\newblock Le probl{\`e}me de cauchy et les {\'e}quations aux d{\'e}riv{\'e}es
  partielles lin{\'e}aires hyperboliques.
\newblock {\em Paris}, 11:243--264, 1932.

\bibitem{HebeyEmmanuel1999Naom}
E.~Hebey.
\newblock {\em {Nonlinear analysis on manifolds : Sobolev spaces and
  inequalities}}, volume~5 of {\em {Courant Lecture Notes in Mathematics}}.
\newblock American Mathematical Society, New York, NY, 2000.

\bibitem{Ho2020}
K.~Ho.
\newblock {FLAM: Fast Linear Algebra in MATLAB - Algorithms for Hierarchical
  Matrices}.
\newblock {\em J. Open Source Softw.}, 5(51):1906, 2020.

\bibitem{Hormander1994}
L.~Hormander.
\newblock {\em {The Analysis of Linear Partial Differential Operators III}}.
\newblock Springer-Verlag, 1994.

\bibitem{HsiaoG.C.GeorgeC.2021Bie}
G.~C. Hsiao and W.~L. Wendland.
\newblock {\em Boundary integral equations}.
\newblock Applied mathematical sciences (Springer-Verlag New York Inc.) ; v.
  164. Springer, 2nd edition, 2021.

\bibitem{Imbert-Gerard2017}
L.~M. Imbert-G{\'{e}}rard and L.~Greengard.
\newblock {Pseudo-Spectral Methods for the Laplace-Beltrami Equation and the
  Hodge Decomposition on Surfaces of Genus One}.
\newblock {\em Numerical Methods for Partial Differential Equations},
  33(3):941--955, 2017.

\bibitem{John1982}
F.~John.
\newblock {\em {Partial Differential Equations}}.
\newblock Springer-Verlag, New York, NY, 4th edition, 1982.

\bibitem{king2023closest}
N.~King, H.~Su, M.~Aanjaneya, S.~Ruuth, and C.~Batty.
\newblock A closest point method for surface pdes with interior boundary
  conditions for geometry processing.
\newblock {\em arXiv:2305.04711}, 2023.

\bibitem{kromer2018}
J.~Kromer and D.~Bothe.
\newblock {Highly accurate numerical computation of implicitly defined volumes
  using the Laplace-Beltrami operator}.
\newblock {\em arXiv:1805.03136}, pages 1--25, 2018.

\bibitem{kropinski2014fast}
M.~C. Kropinski and N.~Nigam.
\newblock {Fast integral equation methods for the Laplace-Beltrami equation on
  the sphere}.
\newblock {\em Adv. Comput. Math.}, 40(2):577--596, 2014.

\bibitem{Kropinski2016}
M.~C. Kropinski, N.~Nigam, and B.~Quaife.
\newblock {Integral equation methods for the Yukawa-Beltrami equation on the
  sphere}.
\newblock {\em Adv. Comput. Math.}, 42(2):469--488, 2016.

\bibitem{Kuo2012}
K.~A. Kuo and H.~E. Hunt.
\newblock {The vibrations of bubbles and balloons}.
\newblock {\em Acoustics Australia}, 40(3):183--187, 2012.

\bibitem{lax2002functional}
P.~D. Lax.
\newblock {\em Functional analysis}, volume~55.
\newblock John Wiley \& Sons, 2002.

\bibitem{Leoni2017}
G.~Leoni.
\newblock {\em {A First Course in Sobolev Spaces: Second Edition}}.
\newblock American Mathematical Society, Providence, Rhode Island, second
  edition, 2017.

\bibitem{macdonald2008level}
C.~B. Macdonald and S.~J. Ruuth.
\newblock Level set equations on surfaces via the closest point method.
\newblock {\em J. Sci. Comput.}, 35:219--240, 2008.

\bibitem{macdonald2010implicit}
C.~B. Macdonald and S.~J. Ruuth.
\newblock The implicit closest point method for the numerical solution of
  partial differential equations on surfaces.
\newblock {\em SIAM Journal on Scientific Computing}, 31(6):4330--4350, 2010.

\bibitem{Malhotra2019}
D.~Malhotra, A.~Cerfon, L.-M. Imbert-G{\'{e}}rard, and M.~O'Neil.
\newblock {Taylor States in Stellarators: A Fast High-order Boundary Integral
  Solver}.
\newblock {\em J. Comput. Phys.}, Feb 2019.

\bibitem{martinsson2019fast}
P.-G. Martinsson.
\newblock {\em Fast direct solvers for elliptic PDEs}.
\newblock SIAM, 2019.

\bibitem{martinsson2005fast}
P.-G. Martinsson and V.~Rokhlin.
\newblock A fast direct solver for boundary integral equations in two
  dimensions.
\newblock {\em J. Comput. Phys.}, 205(1):1--23, 2005.

\bibitem{martinsson2007accelerated}
P.-G. Martinsson and V.~Rokhlin.
\newblock An accelerated kernel-independent fast multipole method in one
  dimension.
\newblock {\em SIAM J. Sci. Comput.}, 29(3):1160--1178, 2007.

\bibitem{naumovets1985surface}
A.~Naumovets and Y.~S. Vedula.
\newblock Surface diffusion of adsorbates.
\newblock {\em Surface Science Reports}, 4(7-8):365--434, 1985.

\bibitem{nedelec2001}
J.-C. Ned\'el\'ec.
\newblock {\em {Acoustic and Electromagnetic Equations}}.
\newblock Springer-Verlag, New York, NY, 2001.

\bibitem{NOVAK20071271}
I.~L. Novak, F.~Gao, Y.-S. Choi, D.~Resasco, J.~C. Schaff, and B.~M.
  Slepchenko.
\newblock {Diffusion on a curved surface coupled to diffusion in the volume:
  Application to cell biology}.
\newblock {\em J. Comput. Phys.}, 226(2):1271--1290, 2007.

\bibitem{ONeil2018}
M.~O'Neil.
\newblock {Second-kind integral equations for the Laplace-Beltrami problem on
  surfaces in three dimensions}.
\newblock {\em Adv. Comput. Math.}, 44(5):1385--1409, 2018.

\bibitem{o2018integral}
M.~O'Neil and A.~J. Cerfon.
\newblock An integral equation-based numerical solver for {T}aylor states in
  toroidal geometries.
\newblock {\em J. Comput. Phys.}, 359:263--282, 2018.

\bibitem{punchin1988weakly}
J.~Punchin.
\newblock Weakly singular integral operators as mappings between function
  spaces.
\newblock {\em The Journal of Integral Equations and Applications}, pages
  303--320, 1988.

\bibitem{Rachh2015}
M.~Rachh.
\newblock {\em {Integral equation methods for problems in electrostatics,
  elastostatics and viscous flow}}.
\newblock PhD thesis, New York University, 2015.

\bibitem{rahimian2010petascale}
A.~Rahimian, I.~Lashuk, S.~Veerapaneni, A.~Chandramowlishwaran, D.~Malhotra,
  L.~Moon, R.~Sampath, A.~Shringarpure, J.~Vetter, R.~Vuduc, et~al.
\newblock Petascale direct numerical simulation of blood flow on 200k cores and
  heterogeneous architectures.
\newblock In {\em SC'10: Proceedings of the 2010 ACM/IEEE International
  Conference for High Performance Computing, Networking, Storage and Analysis},
  pages 1--11. IEEE, 2010.

\bibitem{ruuth2008simple}
S.~J. Ruuth and B.~Merriman.
\newblock A simple embedding method for solving partial differential equations
  on surfaces.
\newblock {\em J. Comput. Phys.}, 227(3):1943--1961, 2008.

\bibitem{saad1986gmres}
Y.~Saad and M.~H. Schultz.
\newblock {GMRES: A generalized minimal residual algorithm for solving
  nonsymmetric linear systems}.
\newblock {\em SIAM J. Sci. Stat. Comput.}, 7(3):856--869, 1986.

\bibitem{vallet2008spectral}
B.~Vallet and B.~L\'evy.
\newblock {Spectral Geometry Processing with Manifold Harmonics}.
\newblock {\em Computer Graphics Forum}, 27(2):251--260, 2008.

\bibitem{veerapaneni2011fast}
S.~K. Veerapaneni, A.~Rahimian, G.~Biros, and D.~Zorin.
\newblock A fast algorithm for simulating vesicle flows in three dimensions.
\newblock {\em J. Comput. Phys.}, 230(14):5610--5634, 2011.

\bibitem{B.Vioreanu2014}
B.~Vioreanu and V.~Rokhlin.
\newblock {Spectra of Multiplication Operators as a Numerical Tool}.
\newblock {\em SIAM J. Sci. Comput.}, 36(1):A267--A288, 2014.

\bibitem{wang2018modified}
M.~Wang, S.~Leung, and H.~Zhao.
\newblock {Modified Virtual Grid Difference for Discretizing the
  Laplace--Beltrami Operator on Point Clouds}.
\newblock {\em SIAM J. Sci. Comput.}, 40(1):A1--A21, 2018.

\bibitem{Warner2013}
F.~W. Warner.
\newblock {\em {Foundations of Differentiable Manifolds and Lie Groups}}.
\newblock Springer, New York, NY, 2013.

\bibitem{xing2020interpolative}
X.~Xing and E.~Chow.
\newblock Interpolative decomposition via proxy points for kernel matrices.
\newblock {\em SIAM J. Matrix Anal. Appl.}, 41(1):221--243, 2020.

\bibitem{ying2004kernel}
L.~Ying, G.~Biros, and D.~Zorin.
\newblock A kernel-independent adaptive fast multipole algorithm in two and
  three dimensions.
\newblock {\em J. Comput. Phys.}, 196(2):591--626, 2004.

\end{thebibliography}

\end{document}